\documentclass[12pt,reqno]{amsart}
\usepackage[margin=1in]{geometry}
\usepackage{amsmath,amssymb,amsthm,graphicx,amsxtra, setspace}
\usepackage[utf8]{inputenc}
\usepackage{mathrsfs}
\usepackage{hyperref}
\usepackage{upgreek}
\usepackage{mathtools}
\allowdisplaybreaks

\newtheorem{theorem}{Theorem}[section]
\newtheorem{lemma}[theorem]{Lemma}
\newtheorem{proposition}[theorem]{Proposition}

\newtheorem{corollary}[theorem]{Corollary}
\newtheorem{definition}[theorem]{Definition}
\newtheorem{example}[theorem]{Example}
\newtheorem{remark}[theorem]{Remark}

\let\originalleft\left
\let\originalright\right
\renewcommand{\left}{\mathopen{}\mathclose\bgroup\originalleft}
\renewcommand{\right}{\aftergroup\egroup\originalright}

\newcommand{\Tr}{\mathop{\mathrm{Tr}}}
\renewcommand{\d}{\/\mathrm{d}\/}

\def\w{\textbf{W}^{\varepsilon}_{{\theta}^{\varepsilon}}}

\def\e{\varepsilon}

\def\L{\mathbb{L}}
\def\A{\mathrm{A}}
\def\I{\mathrm{I}}
\def\F{\mathrm{F}}
\def\C{\mathrm{C}}
\def\f{\mathbf{f}}
\def\J{\mathrm{J}}
\def\B{\mathrm{B}}
\def\D{\mathrm{D}}
\def\y{\mathbf{y}}

\def\Z{\mathrm{Z}}
\def\U{\mathrm{U}}
\def\K{\mathrm{K}}
\def\E{\mathbb{E}}
\def\X{\mathbb{X}}
\def\x{\mathbf{x}}
\def\k{\mathbf{k}}

\def\h{\mathbf{h}}
\def\z{\mathbf{z}}
\def\v{\mathbf{v}}
\def\V{\mathbb{v}}
\def\w{\mathbf{w}}
\def\W{\mathrm{W}}
\def\G{\mathrm{G}}
\def\Q{\mathrm{Q}}
\def\M{\mathrm{M}}

\def\no{\nonumber}
\def\V{\mathbb{V}}
\def\wi{\widetilde}
\def\Q{\mathrm{Q}}
\def\u{\mathrm{U}}
\def\P{\mathrm{P}}
\def\u{\mathbf{u}}
\def\H{\mathbb{H}}
\def\n{\mathbf{n}}

\newcommand{\R}{\mathbb{R}}

\renewcommand{\d}{\/\mathrm{d}\/}

\newcommand{\Addresses}{{
		\footnote{
			\noindent \textsuperscript{1,2}Department of Mathematics, Indian Institute of Technology Roorkee-IIT Roorkee,
			Haridwar Highway, Roorkee, Uttarakhand 247667, INDIA.\par\nopagebreak
			\noindent  \textit{e-mail:} \texttt{Manil T. Mohan: maniltmohan@ma.iitr.ac.in, maniltmohan@gmail.com.}
			
			\textit{e-mail:} \texttt{Ankit Kumar: akumar14@mt.iitr.ac.in.}
			
			\noindent \textsuperscript{*}Corresponding author.
			
			\textit{Key words:} Stochastic convective Brinkman-Forchheimer equations,  Irreducibility, Strong Feller, Invariant measure, Large deviation principle, Occupation measures.
			
			Mathematics Subject Classification (2020): Primary 60F10; Secondary 60J35; 76D03; 37L40.

}}}

\begin{document}
	
	\title[LDP for  2D SCBF equations]{Large deviation principle for occupation measures of  two dimensional stochastic convective Brinkman-Forchheimer equations
		\Addresses}
	
	\author[A. Kumar and M. T. Mohan]
	{Ankit Kumar\textsuperscript{1} and Manil T. Mohan\textsuperscript{2*}}

	\maketitle
	
	\begin{abstract}
		The present work is concerned about  two-dimensional stochastic convective Brinkman-Forchheimer (2D SCBF) equations perturbed by a white noise (non degenerate) in smooth bounded domains in $\R^{2}$. We establish two important properties of the Markov semigroup associated with the solutions of 2D SCBF equations (for the absorption exponent $r=1,2,3$), that is, irreducibility and strong Feller property. These two properties implies the uniqueness of invariant measures and ergodicity also.  Then, we discuss about the ergodic behavior of 2D SCBF equations by providing a Large Deviation Principle (LDP) for the occupation measure for large time (Donsker-Varadhan), which describes the exact rate of exponential convergence. 
	\end{abstract}
	
	\section{Introduction}
	The convective Brinkman-Forchheimer (CBF) equations characterize the motion of incompressible fluid flows in a saturated porous medium. 	In this paper, we consider the CBF  equations in a bounded domain $\mathcal{O}\subset\R^{2}$ with a smooth boundary $\partial\mathcal{O}$ subject to an external random forcing and study some asymptotic behavior of its solutions. The motion of the incompressible fluid  governed by the CBF equations for $(t,\xi)\in[0,T]\times\mathcal{O}$ is given by (see \cite{MTM1}): 
	\begin{equation}\label{1p1}
	\left\{
	\begin{aligned}
	\frac{\partial \u(t,\xi)}{\partial t}-\mu \Delta\u(t,\xi)+(\u(t,\xi)\cdot\nabla)\u(t,\xi)&+\rho\u(t,\xi)+\beta|\u(t,\xi)|^{r-1}\u(t,\xi)\\+\nabla p(t,\xi)&=\mathbf{f}(\xi)+\eta(t,\xi),\\
	\nabla\cdot\u(t,\xi)&=0,
	\end{aligned}
	\right.
	\end{equation}
	where $\u(t,\xi)\in\mathbb{R}^2$ denotes the velocity field at time $t\in[0,T]$ and position $\xi\in\mathcal{O}$, $p(t,\xi)$ represents the pressure field, $\f(\xi)\in\mathbb{R}^2$ stands for a deterministic forcing and $\eta(t,\xi)\in\mathbb{R}^2$ is a white noise type in time. Since our interest is on the long time behavior of \eqref{1p1}, both the forcing terms are assumed to be stationary in order to have an autonomous system (the white noise is by definition a stationary process). We can also consider these equations as a modification (by an absorption term $\rho\u+\beta|\u|^{r-1}\u$) of the classical 2D Navier-Stokes equations, and one can consider  it as the damped Navier-Stokes equations. The constant $\mu$ represents the positive Brinkman coefficient (effective viscosity), the positive constants $\rho$ and $\beta$ denote the Darcy (permeability of porous medium) and Forchheimer (proportional to the porosity of the material) coefficient, respectively. The absorption exponent $r\in[1,\infty)$ and  $r=3$ is known as the critical exponent. We associate \eqref{1p1} with the homogeneous Dirichlet boundary condition
	\begin{align}\label{1p2}
	\u(t,\xi)=\mathbf{0},\ \text{ for }\ t\in[0,T],\ \xi\in\partial\mathcal{O}, 
	\end{align} 
	and the initial condition 
	\begin{align}\label{1p3}
	\u(0,\xi)=\x(\xi), \ \text{ for }\ \xi\in\mathcal{O}. 
	\end{align}
	We impose the following condition for the uniqueness of the pressure $p(\cdot,\cdot)$: 
	\begin{align}\label{1p4}
	\int_{\mathcal{O}}p(t,\xi)\d \xi=0, \ \text{ for }\ t\in[0,T]. 
	\end{align}

	The existence and uniqueness of strong solutions (in the probabilistic sense) for 2D and 3D SCBF equations (for $r\geq1$ any $\mu,\beta>0$ in 2D, and for $r> 3$ any $\mu,\beta>0$ and $r=3$ with $2\beta\mu\geq 1$ in 3D) perturbed by multiplicative Gaussian noise in bounded subsets of $\R^{2}$ is obtained in the work \cite{MTM1}. The author in \cite{MTM1} exploited the monotonicity property of the linear and nonlinear operators (local monotonicity for $r\in[1,3]$) and a stochastic generalization of the  Minty-Browder technique to obtain the global solvability results.  Using the classical Faedo-Galerkin approximation and compactness method, the existence of martingale solutions for the stochastic 3D Navier-Stokes equations with nonlinear damping is established in \cite{LHGH1}. The global solvability of 2D and 3D SCBF equations perturbed by multiplicative jump noise is discussed in \cite{MTM4}. 
	
	For the past three decades, the ergodic properties for infinite-dimensional systems have been extensively investigated (cf. \cite{GDJZ}). The ergodic behavior of solutions of the  stochastic  Navier-Stokes equations is well studied in the literature and we refer the interested readers to \cite{GDJZ,GDAD,GDP,ADe,BF,FF,FFBM,MH,MSS}, etc for more details and the references therein. The  existence and uniqueness of invariant measures for the two and three dimensional SCBF equations subjected to multiplicative Gaussian noise as well as jump noise  by using the exponential stability of strong solutions is established in the works \cite{MTM1,MTM4}, respectively. The large time  behavior of  solutions by establishing the existence of random attractors for the stochastic flow generated by the 2D SCBF equations \eqref{1p1}-\eqref{1p4} perturbed by small additive noise and the existence of an invariant measure is discussed in \cite{KKMTM}.  The existence of a random attractor and the existence of a unique invariant measure for the stochastic 3D Navier-Stokes equations with damping driven by a multiplicative noise is established in \cite{LHGH}.  The asymptotic log-Harnack inequality for the transition semigroup associated with the SCBF equations driven by additive as well as multiplicative degenerate noise via the asymptotic coupling method is established in \cite{MTM2}. As applications of the asymptotic log-Harnack inequality, the author  derived the gradient estimate, asymptotic irreducibility, asymptotic strong Feller property, asymptotic heat kernel estimate and ergodicity.

	The theory of large deviations, which provides asymptotic estimates for probabilities of rare events, is one of the active and important research topics in probability theory and has rightly received attention. The framework for the theory of large deviations along with the applications can be found in  \cite{VA}. Several authors  developed  this theory over the years and expanded its applications to variety of areas (see for example, \cite{DZ,DE,ST}, etc).	Many authors have established the Wentzell-Freidlin type large deviation principle (LDP) for different classes of stochastic partial differential equations (cf. \cite{Chow1,KX,SW}, etc).  By employing a weak convergence approach, the Wentzell-Freidlin type large deviation principle  for the two-dimensional stochastic Navier-Stokes equations (SNSE) perturbed by a small multiplicative noise in both bounded and unbounded domains is established in \cite{SSSP}.  Using a weak convergence approach of Budhiraja and Dupuis \cite{BD1},  the Wentzell-Freidlin LDP for the two and three dimensional SCBF equations (in 2D, $r\geq 1$, $\beta,\mu>0$ and in 3D, $r\geq 3$, where $\beta,\mu>0$ for $r>3$ and $\beta\mu>1$ for $r=3$) is established in \cite{MTM}. The large deviations for short time  as well as the exponential estimates on certain exit times associated with the solution trajectory of SCBF equations is also studied in \cite{MTM}. A central limit theorem and a moderate deviation principle (MDP) for  2D SCBF equations  using a variational method based on weak convergence approach is established in \cite{MTM3}. For the stochastic tamed 3D Navier-Stokes equations driven by multiplicative Gaussian noise  in the whole space or on a torus, a large deviation principle of Freidlin-Wentzell type   is established in \cite{MRTZ1}. LDP for the 3D tamed Navier-Stokes equations driven by multiplicative L\'evy noise in periodic domains is established in \cite{HGHL}. Small time large deviations principles  for the stochastic 3D tamed Navier-Stokes equations in bounded domains is established in the work \cite{MRTZ}  and for the stochastic 3D Navier-Stokes equation with damping in bounded domains  is obtained in the work \cite{LHGH1}. Due to the technical difficulties described in the works \cite{MTM1,MTM}, etc some of the above mentioned results may not hold in bounded domains. 
	
	A stochastic partial differential equation (SPDE) is ergodic means that the occupation measures of it's solution converge to a unique invariant measure. A Donsker-Varadhan LDP provides an estimate on the probability of occupation measures deviations from the invariant measure (cf. \cite{DSW,MDD}, etc). Thus, it is important to ask  whether the occupation measures satisfy a Donsker-Varadhan LDP. Likewise Wentzell-Freidlin type LDP, there are several results available in the literature regarding Donsker-Varadhan LDP (cf. \cite{MDD,MG2,MG1,JVNV,JVNV1,LW1} etc and the references therein). A criterion  for  LDP of occupation measures for strong Feller and irreducible Markov processes by checking the hyper-exponential recurrence is provided in \cite{LW1}.  However, the  hyper-exponential recurrence is a very strong condition and  verifying the condition is hard for SPDEs. The author in \cite{MG2} and \cite{MG1} verified the hyper-exponential recurrence  and proved the LDP of occupation measures for  stochastic Burgers equation and 2D SNSE, respectively. Authors in  \cite{JVNV1} investigated an LDP for occupation measures of the Markov process associated with the solutions of a class of dissipative PDE's perturbed by a bounded random kick force. LDP of occupation measures for a class of dissipative PDE's perturbed by an unbounded kick force is studied in \cite{JVNV} and for stochastic reaction-diffusion equations driven by subordinate Brownian motions is established in \cite{RWLX}. Using the hyper-exponential recurrence criterion, the occupation measures LDP for a class of non-linear monotone SPDEs including the stochastic $p$-Laplace equation, stochastic porous medium equation,  stochastic fast-diffusion equation, etc, is proved in \cite{WA}. 
	
	In this work, we first prove the irreducibility and strong Feller property of the Markov semigroup associated with the SCBF equations \eqref{1p1}-\eqref{1p4} for $r=1,2,3$ (for physically relevant cases, linear, quadratic and cubic growth) and hence the uniqueness of invariant measure. The irreducibility and strong Feller property  for the Markov semigroup associated with the 2D SNSE is established in \cite{BF,BF1} (see \cite{FFBM} also),  and we use similar ideas for our model also. We point out here that  the  asymptotic irreducibility and  asymptotic strong Feller property of the transition semigroup associated with the 2D SCBF equations \eqref{1p1}-\eqref{1p4} perturbed by  degenerate noise  are available in \cite{MTM2} for sufficiently large $\mu$ (for $r\in[1,3]$).  In the second part, we discuss about the ergodic behavior of the 2D SCBF equations by proving the Donsker-Varadhan LDP  of the occupation measures with respect to the topology $\tau$   as well as  weak convergence topology.  In order to prove the LDP with respect to topology $\tau$ for SCBF equations, we verify the hyper-exponential recurrence given in \cite{LW1}. Some technical difficulties arise due to the presence of nonlinear damping term $|\u|^{r-1}\u$ appearing in \eqref{1p1}, and we overcome these difficulties by properly choosing the regularity of the noise coefficient. We take the following form of noise in this work (see section \ref{sec2} for functional framework): 
	\begin{align*}
	\mathcal{P}\eta(t,\xi)\d t=\G\d\W(t),
	\end{align*}
	where $\mathcal{P}$ is the Helmholtz-Hodge projection operator (\cite{DFHM}), $\W(t)$ is a standard cylindrical Wiener process in $\H$ (see \cite{DaZ} for more details)  defined on a fixed probability space $(\Omega,\mathscr{F},\mathbb{P})$ and $\G:\H\to\H$ is a bounded linear operator. In order to prove the Donsker-Varadhan LDP for the 2D SCBF equations \eqref{1p1}-\eqref{1p4}, we assume that $\G$ satisfies the following:  
	\begin{align}\label{1p5}
	\D(\A^{2\alpha})\subset\mathrm{Im}(\G)\subset\D(\A^{\frac{1}{2}+\e}), \ \text{ for some }\ \frac{r-1}{2r}<\alpha<\frac{1}{2}, \ \e\in\left(0,2\alpha-\frac{1}{2}\right],
	\end{align}
	for $r=2,3$ and $\frac{1}{4}<\alpha<\frac{1}{2}$ for $r=1$, where $\mathrm{Im}(\G)$ denotes the range of the operator $\G$. That is, for  $r=1,2$ (linear and quadratic growth), we obtain $\frac{1}{4}<\alpha<\frac{1}{2}$, which is same as the case 2D SNSE (cf. \cite{BF,MG1}) and for $r=3$ (cubic growth), we need $\frac{1}{3}<\alpha<\frac{1}{2}$. The first embedding in \eqref{1p5} indicates that the noise is not too degenerate and the second embedding implies $\Q=\Tr(\G^*\G)<\infty$ (Lemma 2.3, \cite{MG1}), and also provides more spatial regularity for the solutions of \eqref{1p1}-\eqref{1p4}. Under \eqref{1p5}, we show that the solution $\u(\cdot)$ of the system \eqref{2.7} (see below) is a Markov process with a unique invariant measure $\nu$ supported by $\D(\A^{\alpha})$. By the uniqueness of invariant measure $\nu$, we know that $\nu$ is ergodic (\cite{GDJZ}) in the sense that $\lim\limits_{T\to\infty}\frac{1}{T}\int_0^T\psi(\u(t,\x))\d t=\int_{\H}\psi(\x)\d\nu,$ $\mathbb{P}$-a.s., for all initial values $\x\in\H$ and all continuous and bounded functions $\psi$. Then, we  follow the work \cite{MG1} for obtaining the  LDP with respect to topology $\tau$ of occupation measures for the 2D SCBF equations \eqref{1p1}-\eqref{1p4}.
	
	 We mention here that the condition on $\G$ given in \eqref{1p5} is restrictive and we impose such a condition on $\G$ to obtain the LDP for occupation measures of 2D SCBF equations (as non-degeneracy and finite trace of $\G^*\G$ is needed). As discussed in \cite{BF1}, in order to prove the existence of a unique invariant measure for  2D SCBF equations on bounded domains, one can relax the condition on $\G$ in the following way (see Remark \ref{rem3.7} below): 
		\begin{align}\label{1p6}
	\D(\A^{\frac{\ell+1}{2}})\subset\mathrm{Im}(\G)\subset\D(\A^{\frac{\ell}{2}+\e}), \ \text{ for some }\ 1< \ell<\frac{3}{2}, \ \e\in\left(0,\frac{1}{2}\right].
	\end{align}
	Note that for $1<\ell<\frac{3}{2}$ and $\frac{r-1}{r}<2\alpha<1,$ we obtain $\D(\A^{\frac{\ell+1}{2}})\subset\D(\A^{2\alpha})\subset\D(\A^{\frac{1}{2}+\e})\subset\D(\A^{\frac{\ell}{2}+\e})$. The restriction on $\ell$ in \eqref{1p6} is due to the technical difficulty arising on bounded domains (cf. \eqref{2p4}, \eqref{2p13}, \eqref{2p19}, Remark \ref{rem3.7}, etc). The case of $\ell=1$ is more delicate and the bounds on the noise become
	\begin{align}\label{1p7}
	\D(\A^{1-\e_2})\subset\mathrm{Im}(\G)\subset\D(\A^{\frac{1}{2}+\e_1}),\ \text{ for some }\ 0<\e_1+\e_2\leq\frac{1}{2}. 
	\end{align}
	 On periodic domains, the restriction on $\ell$ can be removed, and one can assume \eqref{1p6}, for all $\ell\geq 1$, and the existence and uniqueness of invariant measures for 2D SCBF equations in periodic domains will be discussed in a separate work.

	The rest of the paper is structured  as follows. In section \ref{sec2}, we define the linear and nonlinear operators, and provide the necessary functional setting to obtain our main results. Then, we provide the abstract formulation of the 2D SCBF equations \eqref{1p1}-\eqref{1p4} perturbed by non-degenerate additive noise and discuss about the existence and uniqueness of pathwise  strong solution. Section \ref{sec3} is devoted for establishing  the existence and uniqueness of invariant measures for  the 2D SCBF equations (Theorem \ref{main}) under the assumption \eqref{1p5}.  Following similar arguments as in \cite{BF}, we prove that the  Markov  semigroup associated with the strong solution of 2D SCBF equations is strong Feller and irreducible (Propositions \ref{prop3.6} and \ref{prop3.7}).  In section \ref{sec4}, we state and prove our main result on the Donsker-Varadhan LDP of occupation measures for 2D SCBF equations (Theorem \ref{main-t} and Corollary \ref{cor4.2}) by using the hyper-exponential recurrence criterion given in \cite{LW1}.  We first provide some general results about large deviations with different level entropy functionals of Donsker-Varadhan. Then, we prove the crucial exponential estimates for the solutions of the 2D SCBF equations (Proposition \ref{prop4.6}).  In order to establish our main Theorem \ref{main-t}, we first prove the LDP on a restricted space, which is $\M_1(\D(\A^{\alpha}))$ (Lemma \ref{lem4.11}) and then we extend the result  to the required space $\M_1(\H)$. 
	
	\section{Mathematical Formulation}\label{sec2}\setcounter{equation}{0}
	This section provides the necessary function spaces needed to obtain the global solvability results of the system \eqref{1p1}-\eqref{1p4}.   In our analysis, the parameter $\rho$ does not play a major role and we set $\rho$ to be zero in \eqref{1p1}-\eqref{1p4} in the rest of the paper. Also in this section, we discuss the existence and uniqueness of strong solutions of the 2D SCBF equations.
	\subsection{Function spaces} Let $\C_0^{\infty}(\mathcal{O};\R^{2})$ denote the space of all infinitely differentiable functions  ($\R^2$-valued) with compact support in $\mathcal{O}\subset\R^{2}$.  We define 
	$	\mathcal{V}:=\{\u\in\C_0^{\infty}(\mathcal{O},\R^{2}):\nabla\cdot\u=0\}$. Let 
	$	\mathbb{H}$, $	\mathbb{V}$ and $\widetilde{\L}^{p}$ stand for the   closure of $\mathcal{V} $ in the Lebesgue space  $\L^2(\mathcal{O})=\mathrm{L}^2(\mathcal{O};\R^{2})$,
	Sobolev space $\H^1(\mathcal{O})=\mathrm{H}^1(\mathcal{O};\R^{2}),$ and
	Lebesgue space $ \L^p(\mathcal{O})=\mathrm{L}^p(\mathcal{O};\R^{2}),$
	for $p\in(2,\infty)$, respectively. Then under some smoothness assumptions on the boundary (for instance, one can take $\C^2$-boundary), we characterize the spaces $\H$ as 
	$
	\H=\{\u\in\L^2(\mathcal{O}):\nabla\cdot\u=0,\u\cdot\mathbf{n}\big|_{\partial\mathcal{O}}=0\}$,  with norm  $\|\u\|_{\H}^2:=\int_{\mathcal{O}}|\u(x)|^2\d x,
	$
	where $\mathbf{n}$ is the outward normal to $\partial\mathcal{O}$, and $\u\cdot\n\big|_{\partial\mathcal{O}}$ should be understood in the sense of trace in $\H^{-1/2}(\partial\mathcal{O})$ (cf. Theorem 1.2, Chapter 1, \cite{RT}). Similarly, we characterize the spaces $\V$ and $\wi\L^p$ as 
	$
	\V=\{\u\in\H_0^1(\mathcal{O}):\nabla\cdot\u=0\},$  with norm $ \|\u\|_{\V}^2:=\int_{\mathcal{O}}|\nabla\u(x)|^2\d x,
	$ and $\widetilde{\L}^p=\{\u\in\L^p(\mathcal{O}):\nabla\cdot\u=0, \u\cdot\mathbf{n}\big|_{\partial\mathcal{O}}=0\},$ with norm $\|\u\|_{\widetilde{\L}^p}^p=\int_{\mathcal{O}}|\u(x)|^p\d x$, respectively.
	Let $(\cdot,\cdot)$ stand for the inner product in the Hilbert space $\H$ and $\langle \cdot,\cdot\rangle $ represent the induced duality between the spaces $\V$  and its dual $\V'$ as well as $\widetilde{\L}^p$ and its dual $\widetilde{\L}^{p'}$, where $\frac{1}{p}+\frac{1}{p'}=1$. Note that $\H$ can be identified with its dual $\H'$ and we have the Gelfand triple $\V\subset\H\subset\V'$. In the sequel, the Sobolev spaces will be denoted by $\mathbb{W}^{s,p}(\mathcal{O})=\W^{s,p}(\mathcal{O};\R^2)$, for $s\in\R$ and $p\in[1,\infty]$  with $\H^{s}(\mathcal{O})=\mathbb{W}^{s,2}(\mathcal{O})$.
	\vskip 0.2cm
	\noindent\textbf{Notations}: For $\mathrm{E}=\H$ or $\mathrm{E}=\D(\A^{\alpha})$, we  denote the space of probability measures on $\mathrm{E}$ equipped with the Borel $\sigma$-field $\mathcal{B}$ by $\M_{1}(\mathrm{E})$, the space of signed $\sigma$-additive measures of bounded variation on $\mathrm{E}$ by  $\M_{b}(\mathrm{E})$,  the space of all bounded Borel measurable functions on $\mathrm{E}$ by $\mathcal{B}_b(\mathrm{E})$ and the space of all bounded continuous functions on $\mathrm{E}$ by $\C_b(\mathrm{E})$. On the space $\M_{b}(\mathrm{E})$, we consider $\sigma(\M_{b}(\mathrm{E}), \mathcal{B}_{b}(\mathrm{E}))$, the  $\tau$-topology of convergence against measurable and bounded functions which is much stronger than the usual weak convergence topology $\sigma(\M_{b}(\mathrm{E}),\C_{b}(\mathrm{E}))$ (\cite{MDD}, Section 6.2, \cite{ADOZ}). Let us denote
	$\|\cdot\|_{\sup}$ for the supremum norm in $\C_{b}\ (\text{or } \mathcal{B}_{b})$. We denote the duality relation between $\nu \in \M_{b}(\mathrm{E})$ and $\psi \in \mathcal{B}_{b}(\mathrm{E})$  by $ \nu(\psi) := \int_{E} \psi \d \nu.$
	
	
	\subsection{Linear operator}
	Let $\mathcal{P}: \L^p(\mathcal{O}) \to\wi\L^p,$ $p\in[1,\infty)$, denote the Helmholtz-Hodge projection (cf.  \cite{DFHM}). For $p=2$, $\mathcal{P}$ becomes an orthogonal projection (\cite{OAL}). We define (see \cite{RT})
	\begin{equation*}
	\A\u:=-\mathcal{P}\Delta\u,\;\u\in\D(\A):=\V\cap\H^{2}(\mathcal{O}).
	\end{equation*}
	It can be easily seen that the operator $\A$ is a non-negative self-adjoint operator in $\H$ with $\V=\D(\A^{1/2})$ and \begin{align}\label{2.2}\langle \A\u,\u\rangle =\|\u\|_{\V}^2,\ \textrm{ for all }\ \u\in\V, \ \text{ so that }\ \|\A\u\|_{\V'}\leq \|\u\|_{\V}.\end{align}
	Since $\mathcal{O}$ is a bounded domain,  the operator $\A$ is invertible and its inverse $\A^{-1}$ is bounded, self-adjoint and compact in $\H$. 
	Making use of the spectral theorem, we know that the spectrum of $\A$ consists of an infinite sequence $0< \lambda_1\leq \lambda_2\leq\ldots\leq \lambda_k\leq \ldots,$ with $\lambda_k\to\infty$ as $k\to\infty$ of eigenvalues. 
	Furthermore, there exists an orthogonal basis $\{e_k\}_{k=1}^{\infty} $ of $\H$ consisting of eigenvectors of $\A$ such that $\A e_k =\lambda_ke_k$,  for all $ k\in\mathbb{N}$.  We know that any $\u\in\H$ can be expressed as $\u=\sum_{k=1}^{\infty}(\u,e_k)e_k$ and hence $\A\u=\sum_{k=1}^{\infty}\lambda_k(\u,e_k)e_k$, for all $\u\in\D(\A)$. 
	Thus, it is immediate that 
	\begin{align}\label{2.3}
	\|\nabla\u\|_{\mathbb{H}}^2=\langle \A\u,\u\rangle =\sum_{k=1}^{\infty}\lambda_k|(\u,e_k)|^2\geq \lambda_1\sum_{k=1}^{\infty}|(\u,e_k)|^2=\lambda_1\|\u\|_{\mathbb{H}}^2,
	\end{align}
	which is the Poincar\'e inequality.  In the sequel, we require the fractional powers of $\A$ also.  For $\u\in \H$ and  $\alpha>0,$ we define
	$\A^\alpha \u=\sum_{k=1}^\infty \lambda_k^\alpha (\u,e_k) e_k,  \ \u\in\D(\A^\alpha), $ where $\D(\A^\alpha)=\left\{\u\in \H:\sum_{k=1}^\infty \lambda_k^{2\alpha}|(\u,e_k)|^2<+\infty\right\}.$ 
	Here  $\D(\A^\alpha)$ is equipped with the norm 
	\begin{equation} \label{2.4}
	\|\A^\alpha \u\|_{\H}=\left(\sum_{k=1}^\infty \lambda_k^{2\alpha}|(\u,e_k)|^2\right)^{1/2}.
	\end{equation}
	It can be easily seen that $\D(\A^0)=\H$, $\D(\A^{1/2})=\V$ and $\D(\A^{-1/2})=\V'$. We set $\V_\alpha= \D(\A^{\alpha/2})$ with $\|\u\|_{\V_{\alpha}} =\|\A^{\alpha/2} \u\|_{\H}.$   Using Rellich-Kondrachov compactness embedding theorem, we infer that for any $0\leq s_1<s_2,$ the embedding $\D(\A^{s_2})\subset \D(\A^{s_1})$ is  compact. We infer from \cite{OHF} that 
	\begin{align}\label{2p4}
	\D(\A^{\alpha})\approx\left\{\begin{array}{ll}\V\cap\H^{2\alpha},&\text{ for } \ 0<\alpha<\frac{1}{4},\\
\V\cap\H^{2\alpha}_0,&\text{ for } \ \frac{1}{4}<\alpha<1,	\end{array}\right. 
	\end{align} where $\approx$ stands for equality of sets and equivalence of norms. 
	
	We use the following interpolation inequality also in the sequel.  For $\mu',\gamma',\beta'\in\mathbb{R}$ and $ \mu' = \theta' \gamma' +(1-\theta')\beta'$, where $\theta \in [0,1]$, we have
	\begin{align}\label{Interpolation}
	\|\A^{\mu'}\u\|_{\H} \leq \|\A^{\gamma'}\u\|^{\theta'}_{\H}\|\A^{\beta'}\u\|^{1-\theta'}_{\H},
	\end{align}
	for all $\u\in\D(\A^{\mu'})\cap\D(\A^{\gamma'})\cap\D(\A^{\beta'})$. 
	The following fractional form of Gagliardo-Nirenberg inequality (see \cite{CMLP} and \cite{LN}) is also used frequently in the paper. 
	Fix $1\leq q, \; l\leq \infty$ and a natural number $n$. Suppose also that a real number $\theta$ and a non-negative number $j$ are such that \begin{align*}
	\frac{1}{p} = \frac{j}{n}+\bigg(\frac{1}{l}-\frac{m}{n}\bigg)\theta + \frac{1-\theta}{q}, \;\; \frac{j}{m}\leq \theta \leq 1, 
	\end{align*} then we have 
	\begin{align}\label{2.1}
	\|\D^{j}\u\|_{\L^{p}} \leq C\|\D^{m}\u\|^{\theta}_{\L^{l}}\|\u\|^{1-\theta}_{\L^{q}},
	\end{align} 
	for all $\u\in\mathbb{W}^{m,l}(\mathcal{O})\cap\H_0^1(\mathcal{O})$. Using Sobolev's inequality, we also have 
	\begin{align}\label{26}
	\|\u\|_{\wi\L^p}\leq C\|\A^{\alpha}\u\|_{\H}, \ \text{ for all }\ \u\in\D(\A^{\alpha})\ \text{ and }\ \left(\frac{1}{2}-\frac{1}{p}\right)\leq\alpha, \ p\in[2,\infty).
	\end{align}
	\begin{remark}\label{rem2.1}
		For $j=0$, $m=1$ and $l=q=2$, we obtain $$\|\u\|_{\wi\L^p}\leq  C\|\u\|_{\V}^{1-\frac{2}{p}}\|\u\|_{\H}^{\frac{2}{p}},\ \text{ for } \ p\in[2,\infty)\ \text{ and }\ \u\in\V.$$ Thus for $\u\in\mathrm{C}([0,T];\H)\cap\mathrm{L}^2(0,T;\V)$, H\"older's inequality gives 
		\begin{align*}
		\int_0^T\|\u(t)\|_{\wi\L^p}^p\d t\leq T^{\frac{4-p}{2}}\sup_{t\in[0,T]}\|\u(t)\|_{\H}^2\left(\int_0^T\|\u(t)\|_{\V}^2\d t\right)^{\frac{p-2}{2}}, \ \text{ for }\ p\in[2,4],
		\end{align*}
		so that $\mathrm{C}([0,T];\H)\cap\mathrm{L}^2(0,T;\V)\subset\mathrm{L}^p(0,T;\wi\L^p)$, for $p\in[2,4]$. 
	\end{remark}
	\subsection{Bilinear operator}
	Let us define the \emph{trilinear form} $b(\cdot,\cdot,\cdot):\V\times\V\times\V\to\R$ by $$b(\u,\v,\w)=\int_{\mathcal{O}}(\u(x)\cdot\nabla)\v(x)\cdot\w(x)\d x=\sum_{i,j=1}^2\int_{\mathcal{O}}\u_i(x)\frac{\partial \v_j(x)}{\partial x_i}\w_j(x)\d x.$$ If $\u, \v$ are such that the linear map $b(\u, \v, \cdot) $ is continuous on $\V$, the corresponding element of $\V'$ is denoted by $\B(\u, \v)$. We also denote $\B(\u) = \B(\u, \u)=\mathcal{P}(\u\cdot\nabla)\u$.	An integration by parts yields 
	\begin{equation}\label{2.5}
	\left\{
	\begin{aligned}
	b(\u,\v,\v) &= 0,\ \text{ for all }\ \u,\v \in\V,\\
	b(\u,\v,\w) &=  -b(\u,\w,\v),\ \text{ for all }\ \u,\v,\w\in \V.
	\end{aligned}
	\right.\end{equation}
	Using H\"older's inequality, we deduce the following inequality:
	\begin{align*}
	|b(\u,\v,\w)|=|b(\u,\w,\v)|\leq \|\u\|_{\wi\L^4}\|\nabla\w\|_{\H}\|\v\|_{\wi\L^4},\ \text{ for all }\ \u,\v,\w\in\V,
	\end{align*}
	and hence  applying Ladyzhenskaya's inequality (Lemma 1, 2, Chapter 1, \cite{OAL}), we get 
	\begin{align}\label{2.9a}
	\|\mathrm{B}(\u,\v)\|_{\V'}\leq \|\u\|_{\wi\L^4}\|\v\|_{\wi\L^4}\leq \sqrt{2}\|\u\|_{\H}^{1/2}\|\u\|_{\V}^{1/2}\|\v\|_{\H}^{1/2}\|\v\|_{\V}^{1/2}, \ \text{ for all } \ \u,\v\in\V. 
	\end{align}
	Using Sobolev's inequality, we infer that $\|\u\|_{\wi\L^4}\leq C\|\A^{\frac{1}{4}}\u\|_{\H}$, for all $\u\in\D(\A^{\frac{1}{4}})$ and hence we obtain 
	\begin{align*}
	|\langle\B(\u,\u),\v\rangle|\leq C\|\A^{\frac{1}{4}}\u\|_{\H}^2\|\A^{\frac{1}{2}}\v\|_{\H},
	\end{align*}
	for all $\u\in\D(\A^{\frac{1}{4}})$ and $\v\in\D(\A^{\frac{1}{2}})$. A calculation similar to \eqref{2.9a} gives 
	\begin{align}\label{2p9}
	\|\B(\u)-\B(\v)\|_{\V'}\leq \left(\|\u\|_{\wi\L^4}+\|\v\|_{\wi\L^4}\right)\|\u-\v\|_{\wi\L^4}\leq\sqrt{\frac{2}{\lambda_1}}(\|\u\|_{\V}+\|\v\|_{\V})\|\u-\v\|_{\V},
	\end{align}
	so that $\B:\V\to\V'$ is locally Lipschitz. 
	For $\u,\v\in\V$, we further have 
	\begin{align}\label{2p8}
	\langle\B(\u)-\B(\v),\u-\v\rangle&=\langle\B(\u-\v,\v),\u-\v\rangle \leq\|\v\|_{\V}\|\u-\v\|_{\wi\L^4}^2,
	\end{align}
	using H\"older's inequality. The following inequalities are well-known and used in the sequel \cite{Te1}: 
	\begin{align}
	|\langle\B(\u,\v),\w\rangle|&\leq C\|\u\|_{\H}^{1/2}\|\u\|_{\V}^{1/2}\|\v\|_{\V}^{1/2}\|\A\v\|_{\H}^{1/2}\|\w\|_{\H},\ \text{ 	for all }\ \u\in\V, \v\in\D(\A),\w\in\H, \label{219}\\
	|\langle\B(\u,\v),\w\rangle|&\leq C\|\u\|_{\H}^{1/2}\|\A\u\|_{\H}^{1/2}\|\v\|_{\V}\|\w\|_{\H}, \ \text{ for all }\ \u\in\D(\A),\v\in\V,\w\in\H, \label{220}
	\end{align}
the final estimate is just an application of Agmon's inequality.	The following property of bilinear operator is used frequently in the paper. 
	\begin{lemma}[Lemma 2.2, \cite{GM}]\label{lem2.2}
		Let $0\leq \delta <1$. Then there exists some constant $C=C(\rho,\theta,\delta)$ such that
		\begin{align}\label{2.14}
		\|\A^{-\delta} \B(\u,\v)\|_{\H} \leq C \|\A^{\theta}\u\|_{\H}\|\A^{\rho}\v\|_{\H}, \ \text{ for all }\  \u \in \D(\A^{\theta}), \; \v \in \D(\A^{\rho}) ,		\end{align}
		where $\theta>0, \; \rho>0,\; \delta+\theta+\rho\geq 1, \; \delta+\rho>\frac{1}{2}$.
	\end{lemma}
Note that $\B(\u,\u)$ does not vanish on the boundary and $\B(\u,\v)\cdot\mathbf{n}=0$ only on the boundary (for $\theta>1/4$), and  from the definition \eqref{2p4}, it is clear that (for $\eta>\theta$) \begin{align}\label{2p13}\langle\B(\u),\A^{\eta}\u\rangle=\langle\A^{\theta}\B(\u),\A^{\eta-\theta}\u\rangle, \ \text{ for }\ \theta<\frac{1}{4} \ \mbox{ only}.\end{align}
For $0<\theta<\frac{1}{4}$, using Lemma 2.1, \cite{Te1}, we obtain 
\begin{align*}
|\langle\A^{\theta}\B(\u),\v\rangle|=|\langle\B(\u),\A^{-\theta}\v\rangle|\leq C\|\A^{\theta+\frac{1}{2}}\u\|_{\H}^2\|\v\|_{\H},
\end{align*}
for all $\v\in\H$, so that 
\begin{align}\label{2p14}
\|\A^{\theta}\B(\u)\|_{\H}\leq C\|\A^{\theta+\frac{1}{2}}\u\|_{\H}^2,
\end{align}
for all $\u\in\D(\A^{\theta+\frac{1}{2}})$. 
	
	\subsection{Nonlinear operator}\label{sub2.4}
	Let us now define the operator $\mathcal{C}(\u):=\mathcal{P}(|\u|^{r-1}\u)$, for $r=1,2,3$. It can be easily seen that $\langle\mathcal{C}(\u),\u\rangle =\|\u\|_{\widetilde{\L}^{r+1}}^{r+1}$. Furthermore, for all $\u\in\wi\L^{r+1}$, the map is Gateaux differentiable with Gateaux derivative 
	
	\begin{align}\label{29}
	\mathcal{C}'(\u)\v&=\left\{\begin{array}{cc}\mathcal{P}(\v),&\text{ for }r=1,\\ \left\{\begin{array}{cc}\mathcal{P}(|\u|\v)+\mathcal{P}\left(\frac{\u}{|\u|}(\u\cdot\v)\right),&\text{ if }\u\neq \mathbf{0},\\\mathbf{0},&\text{ if }\u=\mathbf{0},\end{array}\right.&\text{ for } r=2,\\ \mathcal{P}(|\u|^{2}\v)+2\mathcal{P}(\u(\u\cdot\v)), &\text{ for }r= 3,\end{array}\right.
	\end{align}
	for all $\v\in\widetilde{\L}^{r+1}$.  For $\u,\v\in\wi\L^{r+1}$, it is immediate that 
	\begin{align}\label{2.9}
	\langle\mathcal{C}'(\u)\v,\v\rangle=\int_{\mathcal{O}}|\u(x)|^{r-1}|\v(x)|^2\d x+(r-1)\int_{\mathcal{O}}|\u(x)|^{r-3}|\u(x)\cdot\v(x)|^2\d x\geq 0,
	\end{align}
	for $r= 1,2,3$. For $r= 3$, we further have 
	\begin{align}\label{30}
	\mathcal{C}''(\u)(\v\otimes\w)&=2\mathcal{P}\left\{\left[(\u\cdot\w)\v+(\u\cdot\v)\w+(\w\cdot\v)\u\right]\right\},
	\end{align}
	for all $\u,\v,\w\in\widetilde{\L}^{4}$. For $0<\theta<1$ and $\u,\v\in\widetilde{\L}^{r+1}$, using Taylor's formula (Theorem 7.9.1, \cite{PGC}), we find
	\begin{align*}
	\langle \mathcal{P}(|\u|^{r-1}\u)-\mathcal{P}(|\v|^{r-1}\v),\w\rangle&\leq \|(|\u|^{r-1}\u)-(|\v|^{r-1}\v)\|_{\widetilde{\L}^{\frac{r+1}{r}}}\|\w\|_{\widetilde{\L}^{r+1}}\nonumber\\&\leq \sup_{0<\theta<1}r\|\theta\u+(1-\theta)\v\|_{\widetilde{\L}^{r+1}}^{r-1}\|\u-\v\|_{\widetilde{\L}^{r+1}}\|\w\|_{\widetilde{\L}^{r+1}}\nonumber\\&\leq r\left(\|\u\|_{\widetilde{\L}^{r+1}}+\|\v\|_{\widetilde{\L}^{r+1}}\right)^{r-1}\|\u-\v\|_{\widetilde{\L}^{r+1}}\|\w\|_{\widetilde{\L}^{r+1}},
	\end{align*}
	for all $\w\in\widetilde{\L}^{r+1}$. Thus, it is immediate that 
	\begin{align}\label{214}
	\|\mathcal{C}(\u)-\mathcal{C}(\v)\|_{\wi\L^{\frac{r+1}{r}}}\leq r\left(\|\u\|_{\widetilde{\L}^{r+1}}+\|\v\|_{\widetilde{\L}^{r+1}}\right)^{r-1}\|\u-\v\|_{\widetilde{\L}^{r+1}},
	\end{align}
	for all $\u,\v\in\wi\L^{r+1}$ and hence $\mathcal{C}:\wi\L^{r+1}\to\wi\L^{\frac{r+1}{r}}$ (also from $\V\to\V'$ due to the embedding of $\V\subset\wi\L^{r+1}\subset\wi\L^{\frac{r+1}{r}}\subset\V'$) is a locally Lipschitz operator. 
	Moreover, we have (\cite{MTM1})
	\begin{align}\label{2.23}
	&\langle\mathcal{P}(\u|\u|^{r-1})-\mathcal{P}(\v|\v|^{r-1}),\u-\v\rangle\geq \frac{1}{2}\||\u|^{\frac{r-1}{2}}(\u-\v)\|_{\H}^2+\frac{1}{2}\||\v|^{\frac{r-1}{2}}(\u-\v)\|_{\H}^2\geq 0.
	\end{align}
	It is important to note that 
	\begin{align}\label{a215}
	\|\u-\v\|_{\wi\L^{r+1}}^{r+1}\leq 2^{r-2}\||\u|^{\frac{r-1}{2}}(\u-\v)\|_{\L^2}^2+2^{r-2}\||\v|^{\frac{r-1}{2}}(\u-\v)\|_{\L^2}^2,
	\end{align}
	for $r\geq 1$ (replace $2^{r-2}$ with $1,$ for $1\leq r\leq 2$). Note that $\mathcal{C}(\u)\big|_{\partial\mathcal{O}}\neq \mathbf{0}$ (for $r>1$) and from the definition \eqref{2p4}, it is immediate that (for $\eta>\theta$) \begin{align}\label{2p19}\langle\mathcal{C}(\u),\A^{\eta}\u\rangle=\langle\A^{\theta}\mathcal{C}(\u),\A^{\eta-\theta}\u\rangle, \ \text{ for }\ \theta<\frac{1}{4} \ \mbox{ only}.\end{align}
	Using the chain rule available in Theorem A.6, \cite{CEK}, one can estimate $\|\A^{\theta}\mathcal{C}(\u)\|_{\H}$ as 
	\begin{align}\label{2p22}
	\|\A^{\theta}\mathcal{C}(\u)\|_{\H}\leq C\|\u\|_{\wi\L^{p(r-1)}}^{r-1}\|\u\|_{\mathbb{W}^{2\theta,q}}
	\end{align}
	for all $\u\in\mathbb{W}^{2\theta,q}\cap\wi\L^{p(r-1)}$,  $0\leq \theta<\frac{1}{4}$ and $\frac{1}{p}+\frac{1}{q}=\frac{1}{2}$, where $p,q\in(1,\infty)$. 
	
	The following result provides the local monotonicity property of the operator $\G(\u)=\mu \A\u+\B(\u)+\beta\mathcal{C}(\u)$, which is useful in obtaining the global solvabililty of the system \eqref{1p1}-\eqref{1p4}. 
	\begin{lemma}[\cite{MTM1}]\label{thm2.2}
		Let $\u,\v\in\V$. Then,	for the operator $\G(\u)=\mu \A\u+\B(\u)+\beta\mathcal{C}(\u)$, we  have 
		\begin{align}\label{fe2}
		\langle(\G(\u)-\G(\v),\u-\v\rangle+ \frac{27}{32\mu ^3}N^4\|\u-\v\|_{\H}^2\geq 0,
		\end{align}
		for all $\v\in{\mathbb{B}}_N$, where ${\mathbb{B}}_N$ is an $\widetilde{\L}^4$-ball of radius $N$, that is,
		$
		{\mathbb{B}}_N:=\big\{\z\in\widetilde{\L}^4:\|\z\|_{\widetilde{\L}^4}\leq N\big\}.
		$
	\end{lemma}
	\subsection{Abstract formulation}
	Let $(\Omega,\mathscr{F},\mathbb{P})$ be a complete probability space equipped with an increasing family of sub-sigma fields $\{\mathscr{F}_t\}_{0\leq t\leq T}$ of $\mathscr{F}$ satisfying  the usual conditions. We take the Helmholtz-Hodge projection $\mathcal{P}$ in \eqref{1p1}-\eqref{1p4} to obtain  the abstract formulation as: 
	\begin{equation}\label{2.7}
	\left\{
	\begin{aligned}
	\d\u + \{\mu \A\u + \B(\u) + \beta \mathcal{C}(\u) \} \d t & = \f \d t + \G\d \W(t), \ \ t\geq 0,\\
	\u(0)&=\x, 
	\end{aligned}
	\right.
	\end{equation} 
	where   $ \x \in \H,\ \f \in \V' $ and $ \W(t),\ t\geq 0,$ is cylindrical Wiener process in $\H$ defined on the given filtered probability space $(\Omega, \mathscr{F}, \{\mathscr{F}_{t} \}_{t\geq 0}, \mathbb{P})$.	
	The operator $\G:\H \to \H $ is bounded and satisfies \eqref{1p5}.
	\begin{example}[\cite{MG1}]
		1. We know that the  cylindrical Wiener process $\{\W(t) : 0\leq t\leq T\}$  on $(\Omega,\mathscr{F},\{\mathscr{F}_t\}_{t\geq 0},\mathbb{P})$  can be expressed as  $\W(t)=\sum\limits_{k=1}^\infty  \beta_k(t)e_k,$ where $\beta_k(t), k\in \mathbb{N}$ are independent, one dimensional Brownian motions on the space $(\Omega,\mathscr{F},\{\mathscr{F}_t\}_{t\geq 0},\mathbb{P})$ (cf. \cite{DaZ}). Let us define $\G e_k=\sigma_k e_k$, for $k=1,2,\ldots$, so that $$\G\W(t)=\sum_{k=1}^{\infty}\sigma_k\beta_k(t)e_k.$$ In two dimensions, we know that $\lambda_k\sim k$ as $k\to\infty$. Thus, the condition given in \eqref{1p5} becomes $$\frac{c}{k^{2\alpha}}\leq \sigma_k\leq \frac{C}{k^{\frac{1}{2}+\e}},$$ for some positive constants $c$ and $C$, and large enough $k$. Thus, the cylindrical Wiener process with values in $\D(\A^{2\alpha})$, that is, when $\sigma_k\lambda_k^{2\alpha}=1$ is allowed.

		2. An another example of noise for which our assumption holds for $\frac{r-1}{2r} <\alpha < \frac{1}{2}$ and $r=2,3$ ($\frac{1}{4}<\alpha<\frac{1}{2}$ for $r=1$), fixed  is $\G := \A^{-\beta}\F$ where $\F$ is any linear bounded and invertible operator on $\H$ and $ \frac{1}{2}<\beta <2\alpha$. 
	\end{example}
	\subsection{Solution of SCBF equations} In this subsection, we provide the results regarding the existence of unique global  strong solution to the system \eqref{2.7}.
	\begin{definition}[Global strong solution]
		Let $\x\in\H$ and $\f\in\V'$ be given. An $\H$-valued $\{\mathscr{F}_t\}_{t\geq 0}$-adapted stochastic process $\u(\cdot)$ is called a \emph{strong solution} to the system \eqref{2.7} if the following conditions are satisfied: 
		\begin{enumerate}
			\item [(i)] the process $\u\in\mathrm{L}^4(\Omega;\mathrm{L}^{\infty}(0,T;\H))\cap\mathrm{L}^2(\Omega;\mathrm{L}^2(0,T;\V))\cap\mathrm{L}^{r+1}(\Omega;\mathrm{L}^{r+1}(0,T;\widetilde{\L}^{r+1}))$ and $\u(\cdot)$ has a $\V$-valued  modification, which is progressively measurable with continuous paths in $\H$ and $\u\in\C([0,T];\H)\cap\mathrm{L}^2(0,T;\V)$, $\mathbb{P}$-a.s.,
			\item [(ii)] the following equality holds for every $t\in [0, T ]$, as an element of $\V',$ $\mathbb{P}$-a.s.,
			\begin{align*}
			\u(t)&=\x-\int_0^t\left[\mu \A\u(s)+\B(\u(s))+\beta\mathcal{C}(\u(s))\right]\d s+\f t+\int_0^t \G\d \W(s),
			\end{align*}
			\item [(iii)] the following It\^o's formula holds true: 
			\begin{align*}
			&	\|\u(t)\|_{\H}^2+2\mu \int_0^t\|\u(s)\|_{\V}^2\d s+2\beta\int_0^t\|\u(s)\|_{\L^{r+1}}^{r+1}\d s\nonumber\\&=\|{\x}\|_{\H}^2+2\int_0^t\langle\f,\u(s)\rangle\d s+\Tr(\G\G^*)t+2\int_0^t(\G\d\W(s),\u(s)),
			\end{align*}
			for all $t\in(0,T)$, $\mathbb{P}$-a.s.
		\end{enumerate}
	\end{definition}
	An alternative version of condition (ii) is to require that for any  $\v\in\V$ and $t\in(0,T)$:
	\begin{align*}
	(\u(t),\v)&=(\x,\v)-\int_0^t\langle\mu \A\u(s)+\B(\u(s))+\beta\mathcal{C}(\u(s))-\f,\v\rangle\d s+\int_0^t\left(\G\d \W(s),\v\right),\ \mathbb{P}\text{-a.s.}
	\end{align*}	
	\begin{definition}
		A strong solution $\u(\cdot)$ to (\ref{2.7}) is called a
		\emph{pathwise  unique strong solution} if
		$\widetilde{\u}(\cdot)$ is an another strong
		solution, then $$\mathbb{P}\big\{\omega\in\Omega:\u(t)=\widetilde{\u}(t),\ \text{ for all }\ t\in[0,T]\big\}=1.$$ 
	\end{definition}

	\begin{theorem}[Theorem 3.7, \cite{MTM1}]\label{exis2}
		Let $\x\in\H$ and $\f\in\V'$ be given.  For $r\in[1,3]$, there exists a \emph{pathwise unique strong solution}
		$\u(\cdot)$ to the system \eqref{2.7} such that \begin{align*}\u&\in\mathrm{L}^4(\Omega;\mathrm{L}^{\infty}(0,T;\H))\cap\mathrm{L}^2(\Omega;\mathrm{L}^2(0,T;\V))\cap\mathrm{L}^{r+1}(\Omega;\mathrm{L}^{r+1}(0,T;\widetilde{\L}^{r+1})),\end{align*} with a modification having  continuous paths in $\H$ and $\u\in\C([0,T];\H)\cap\mathrm{L}^2(0,T;\V)$, $\mathbb{P}$-a.s.
	\end{theorem}
	\section{Existence and uniqueness of invariant measure }\label{sec3}\setcounter{equation}{0}
	In this section, we prove the existence and uniqueness of invariant measures for the system \eqref{2.7}, by showing that the Markov semigroup associated with the 2D SCBF equations is irreducibile and strong Feller. The ergodicity results for the 2D Navier-Stokes equations is available in \cite{ADe,BF,FFBM} etc and we follow the work \cite{BF} for obtaining the  uniqueness of invariant measures for the 2D SCBF equations.  Even though the ergodic results for the 2D SCBF equations are available in \cite{MTM1}, in order to discuss about the LDP for occupation measures of the system \eqref{2.7} in the next section,  it is necessary to prove that the Markov semigroup associated with the 2D SCBF equations \eqref{2.7} is irreducibile and strong Feller. We emphasize here that the irreducibility and strong Feller property of the transition semigroup associated with the solutions of the 2D SCBF equations  \eqref{2.7} can be obtained under the weaker assumptions on the noise given in \eqref{1p6} and \eqref{1p7}, and it will be discussed in a separate work (see Remark \ref{rem3.7} below). 
	
	\subsection{Preliminaries}
	Let $\mathrm{E}$ be a Borel subset of $\H$. The transition probability measures $\P(t,\x,\cdot)$ are defined as $\P(t,\x,B)=\mathbb{P}\{\u(t,\x)\in B\},$ for all $ t>0, \; \x \in \mathrm{E}$ and all Borelian sets $B \in \mathcal{B}(\mathrm{E})$,  where $\u(t,\x)$ denotes the solution of the 2D SCBF equations \eqref{2.7} with the initial condition $\x$; such a process $\u(t,\x)$ is shown to exists and is Markovian (for $\x\in\H$)  in \cite{MTM}. Let $\{\P_{t}\}_{t\geq 0}$ be the Markov semigroup in the space $\C_{b}(\mathrm{E})$ associated with the strong solution of the 2D SCBF equations \eqref{2.7}, defined as 
	\begin{align*}
	(\P_{t}\varphi )(\x) = \mathbb{E}[\varphi(\u(t,\x))], \ \text{ for all }\  \varphi \in \C_{b}(\mathrm{E}).
	\end{align*}
	A Markov semigroup is \emph{irreducible} if $\P(t,\x,B) >0,$ for all $t>0, \ \x\in \mathrm{E}$ and non empty open set $B \subset \mathrm{E}$; it is \emph{Feller} if $\P_{t}:\C_{b}(\mathrm{E}) \to \C_{b}(\mathrm{E})$ for arbitrary $t>0 $ and \emph{strong Feller} if $\P_{t}$ can be extended to the space $\mathcal{B}_{b}(\mathrm{E}) \text{ for any } t>0$, that is, $\P_{t}\varphi$ is continuous and bounded in $\mathrm{E}$ for all Borel bounded function $\varphi$ in $\mathrm{E}$.
	These two important properties are essentially related to invariant measures. Let us first consider the dual semigroup $\{\P^{*}_{t}\}_{t\geq0}$ in the space $\M_1(\mathrm{E})$, which is defined as 
	\begin{align*}
	\int_{\mathrm{E}} \varphi \d(\P^{*}_{t}\nu) = \int_{\mathrm{E}}\P_{t}\varphi\d\nu,
	\end{align*}
	for all $\varphi \in \C_{b}(\mathrm{E}) \text{ and } \nu \in \M_1(\mathrm{E})$.
	A measure $\nu \in \M_1(\mathrm{E})$ is called \emph{invariant} if $\P^{*}_{t}\nu = \nu $ for all $t\geq 0 $.
	Existence of invariant measures for the 2D SCBF equations is already shown in  \cite{MTM1}. Using the exponential stability of strong solutions, the author in \cite{MTM} proved the uniqueness of  invariant measure (for multiplicative Gaussian noise) also. In this paper, we use Doob's Theorem to show the uniqueness of invariant measures. By Doob's Theorem, it is sufficient that the measures $\P(t,\x,\cdot)$ and $\P(s,\y,\cdot)$ are absolutely continuous for arbitrary $\x,\y\in\H$ and $t,s>0$. The strong Feller and irreducibility properties imply this condition. Thus, we have the following result (cf. \cite{GDJZ,BM}). 
	\begin{theorem}\label{thm3.1}
		Assume that the Markov semigroup $\P_{t}$ is irreducible and strong Feller. Then there exists at most one invariant measure $\nu$, which is ergodic and equivalent to each transition probability $\P(t,\x,\cdot)$.
	\end{theorem}
	One can obtain the strong mixing property also, that is, $$\lim\limits_{t\to\infty}\|\P_t^*\varrho-\mu\|_{\mathrm{var}}=0,\ \text{ for all }\ \varrho\in\M_1(\mathrm{E}),$$  where $\|\cdot\|_{\mathrm{var}}$ represents the total variation of a measure. The total variation of a measure $\nu$ is defined as \begin{align}\label{7.1}\|\nu\|_{\mathrm{var}}:=\sup_{\psi\in\mathcal{B}_b(\H),\|\psi\|_{\sup}\leq 1}\left|\int \psi(x)\d\nu(x)\right|. \end{align}
	\emph{Ergodicity} means that as the averaging interval becomes infinitely large the time averages converge to the corresponding ensemble averages, that is, $\nu$ is the equilibrium measure over the phase space $\mathrm{E}$ such that 
	\begin{align*}
	\lim_{T\to \infty}\frac{1}{T}\int_{0}^{T}\varphi(\u(t,\x))\d t= \int_{\mathrm{E}}\varphi \d \nu, \quad \mathbb{P} \text{-a.s.,}
	\end{align*}
	for all $\x \in \mathrm{E}$ and Borel measurable functions $\varphi:\mathrm{E}\to\R$ such that $\int_{\mathrm{E}}|\varphi|\d \nu <\infty$.

	\subsection{General procedure and main result} Let us now establish the  irreducibility and strong Feller properties of the transition semigroup associated with the solutions of the system \eqref{2.7}  in the space $\mathrm{E}=\D(\A^{\alpha}),$ for arbitrary $\alpha \in \big[\frac{1}{4},\frac{1}{2}\big)$ for $r=1,2$ and $\alpha \in \big[\frac{1}{3},\frac{1}{2}\big)$ for $r=3$.
	Let us first consider the Ornstein-Uhlenbeck process $\z$, which is the solution of \begin{equation}\label{3.1}\left\{\begin{aligned}
	\d \z(t) + \mu\A \z(t)\d t &=\G\d \W(t), \\ \quad \quad\quad\quad\quad
	\z(0)& = \mathbf{0}.
	\end{aligned}
	\right.
	\end{equation}
	Under proper assumptions on the noise, the above system possesses a unique solution $\z(\cdot)$, which is progressively measurable  with  $\mathbb{P}$-a.s. continuous  trajectories taking values in some appropriate spaces.
	Let us take $ \v= \u-\z$; subtracting \eqref{3.1} from SCBF equations \eqref{2.7}, we obtain 
	\begin{equation}\label{3.2}
	\left\{
	\begin{aligned}
	\frac{\d \v(t)}{\d t}+\mu \A \v(t)+ \B(\v(t)+\z(t),\v(t)+\z(t)) + \beta \mathcal{C}(\v(t)+\z(t)) &=\f,  \\
	\hspace{10cm}\v(0) &=\x,
	\end{aligned}
	\right.
	\end{equation}
	which is, for $\mathbb{P}$-a.s. $ \omega \in \Omega$, a deterministic system.
	First, we recall a result on the existence of unique weak solution (in the deterministic sense) to the system \eqref{3.2} (cf. \cite{MTM1} also). 	The assumption \eqref{1p5} on $\G$ gives $\z \in \C_{\mathbf{0}}([0,T];\D(\A^{\frac{1}{4}})) \; \mathbb{P}$-a.s., (see \cite{BF}, for more details), where $\C_{\mathbf{0}}([0,T];\D(\A^{\frac{1}{4}}))$ is the space of all continuous  functions $\z:[0,T]\to\D(\A^{\frac{1}{4}})$ such that $\z(0)=\mathbf{0}$. Thus, it  can be easily  verified that $ \v \in \C([0,T];\H)\cap \mathrm{L}^{2}(0,T;\V),\ \mathbb{P}$-a.s., and we obtain the following result: 
	\begin{proposition}\label{prop3.2}
		If $\mathrm{Im}(\G) \subseteq \D(\A^{\frac{1}{4}+\e}),$ for some $\e >0$, then for arbitrary $\x \in \H, \; \f \in \D(\A^{-\frac{1}{2}})$, there exists a unique solution of \eqref{2.7} such that, for $\mathbb{P}$-a.s $\omega \in \Omega,$ 
		$$ \u(\cdot,\omega) \in \C([0,T];\H)\cap \mathrm{L}^{2}(0,T;\V)$$ 
		and $$ \u(\cdot, \omega)- \z(\cdot,\omega) \in \mathrm{L}^{2}(0,T;\V).$$
		This is a Markov process satisfying the Feller property in $\H$. Furthermore, if $\f \in \D(\A^{-\frac{1}{2}+\gamma})$ for some $\gamma \in (0,1/2)$, then there exists an invariant measure associated to the Markov semigroup.
	\end{proposition} 
	Let us now discuss about the irreducibility and strong Feller properties of the transition semigroup associated with the 2D SCBF equations \eqref{2.7}. 
	\vskip 0.2 cm
	\noindent\textbf{Step I.} \emph{Irreducibility:} 
	Let us first define the mapping $$ \Phi:\z \mapsto \u = \v+\z.$$ Following \cite{BF}, we provide an outline of the proof of irreducibility to point out the importance of the map $\Phi$. In order to prove irreducibility, it is enough to show that 
	\begin{align*}
	\mathbb{P}\{\u(\cdot,\omega)\in \mathcal {U}_{\rho}\} \geq \mathbb{P}\{\z(\cdot,\omega)\in \mathcal {Z}_{\delta_{\rho}}\}>0,
	\end{align*}
	where $\mathcal {U}_{\rho}, \;\mathcal {Z}_{\delta_{\rho}}$  are open sets in suitable spaces, which will be specified later ($\delta_{\rho}$ means that, for given $\rho, \; \mathcal {Z}_{\delta_{\rho}}$, has to be chosen as a function of $\mathcal {U}_{\rho}$). In order to obtain the final inequality, the law $\mathcal{L}(\z)$ has to be a full measure on a suitable space. The first inequality holds true, since given $ \bar{\u}$, there exists a suitable $\bar{\z}$ such that if $\z$ in the ball  $\mathcal {Z}_{\delta_{\rho}}(\bar{\z})$ of center $\bar{\z}$ and radius $\delta_{\rho}$, then $\u$ belongs to a ball $\mathcal {U}_{\rho}(\bar{\u})$ of center $\bar{\u}$ and radius $\rho$ in suitable topologies.
	Therefore we will examine the validity of this inequality in the following three steps:
	\begin{itemize}
		\item[(i)] $\Phi:\X_{\z} \to \X_{\u}$ is well defined, that is, $\u = \Phi(\z)$ has the required regularity,
		\item[(ii)] given $\bar{\u}\in \X_{\u}$, there exists $\bar{\z}\in \X_{\z}$ such that $\bar{\u}= \Phi(\bar{\z})$,
		\item[(iii)] $\Phi$ is continuous in the assigned topologies, that is, $\|\u-\bar{\u}\|_{\X_{\u}} \leq C \|\z-\bar{\z}\|_{\X_{\z}}$.
	\end{itemize}
	\vskip 0.2 cm
	\noindent\textbf{Step II.} \emph{Strong Feller:} 
	Regarding the strong Feller property of the 2D SCBF equations \eqref{2.7}, we have to consider intermediate auxiliary equations. By the Mean Value Theorem, we obtain that the Markov semigroup is Lipschitz Feller, if we are able to estimate the derivative of the semigroup $\P_{t}\varphi$. In order to estimate this derivative of the semigroup $\P_{t}\varphi,$ we use the Bismut-Elworthy-Li formula (\cite{BEL}), which holds for the  finite dimensional Galerkin approximated system associated with the 2D SCBF equations \eqref{2.7}. 
	
	\begin{remark}
		One can also use the following hypothesis also on $\G$ to obtain the ergodicity results for the 2D SCBF system \eqref{2.7} (cf. \cite{BF} for 2D NSE). Given $\alpha \in \big[\frac{1}{4},\frac{1}{2})$ for $r=1,2$ and $\alpha \in \big[\frac{1}{3},\frac{1}{2})$ for $r=3$, let 
		$\G:\H\to\H$ be a linear bounded injective  operator with $\mathrm{Im}(\G)$ dense in $\D(\A^{\frac{1}{4}+\frac{\alpha}{2}})$ and such that \begin{align}\label{33}\D(\A^{2\alpha}) \subseteq \mathrm{Im}(\G) \subseteq \D(\A^{\frac{1}{4}+\frac{\alpha}{2}+\epsilon'}), \ \text{ for some } \ \epsilon'>0.\end{align} But the second embedding in \eqref{33} is clearly implied by the second embedding in \eqref{1p5}. 
	\end{remark}
	\begin{theorem}\label{main}
		Under the assumption \eqref{33}, for all $ \x \in \D(\A^{\alpha}), \; \f \in \D(\A^{\alpha-\frac{1}{2}})$, there exists a unique $\u(\cdot)$ of the solution of the system \eqref{2.7} such that 
		$$ \u(\cdot,\omega) \in \C([0,T];\D(\A^{\alpha}))\cap \mathrm{L}^{2}(0,T;\D(\A^{\frac{1}{4}+\frac{\alpha}{2}}))$$ 
		and $$ \u(\cdot, \omega)- \z(\cdot,\omega) \in \mathrm{L}^{2}(0,T;\D(\A^{\alpha+\frac{1}{2}})),$$
		for $\mathbb{P}$-a.a. $\omega \in \Omega$.
		
		For the Markov semigroup, the irreducibility and strong Feller properties hold in $\D(\A^{\alpha})$, and there exists a unique invariant measure for the system \eqref{2.7} concentrated on $\D(\A^{\alpha})$. 
	\end{theorem}
	We divide the proof of Theorem \ref{main} into two parts; first we prove irreducibility and then strong Feller. Later, we  use Theorem \ref{thm3.1} to obtain  the uniqueness of invariant measure and ergodicity.
	\subsection{Irreducibility}
	In this subsection, we discuss about the irreducibility of the Markov semigroup $\{\P_t\}_{t\geq 0}$.
	Let us define
	\begin{align*}
	\Phi: \z \mapsto \u, \ \C_{\mathbf{0}}([0,T],\D(\A^{\alpha}))\cap \mathrm{L}^{\frac{4}{1-2\alpha}}(0,T;\D(\A^{\frac{1}{4}+\frac{\alpha}{2}})) \to \C([0,T];\D(\A^{\alpha})).
	\end{align*}
	Since the assumption \ref{1p5} implies the condition \eqref{33}, the condition $\mathrm{Im}(\G) \subseteq \D(\A^{\frac{1}{4}+\frac{\alpha}{2}+\epsilon'})$ yields that  $\z$ has a continuous version taking values in $\D(\A^{\frac{1}{4}+\frac{\alpha}{2}})$. Furthermore,  using Lemma 2.6, Proposition 2.7, \cite{BM1}, the assumption that $\mathrm{Im}(\G)$ is densely embedded in $\D(\A^{\frac{1}{4}+\frac{\alpha}{2}})$ implies that 
	\begin{align*}
	\overline{\mathrm{supp}} \ \mathcal{L}(\z) = \C_{\mathbf{0}}([0,T];\D(\A^{\frac{1}{4}+\frac{\alpha}{2}})),
	\end{align*}
	then the law of the process $\z$ is of full measure on
	$ \C_{\mathbf{0}}([0,T];\D(\A^{\frac{1}{4}+\frac{\alpha}{2}}))$, that is, 
	\begin{align*}
	\mathbb{P}\{\z(\cdot,\omega)\in \Gamma\}>0,
	\end{align*}
	for every non-empty open set $\Gamma \subset \C_{\mathbf{0}}([0,T];\D(\A^{\frac{1}{4}+\frac{\alpha}{2}}))$. Thus the law $\mathcal{L}(\z)$ is of full measure also in the spaces $\C_{\mathbf{0}}([0,T];\D(\A^{\alpha}))$ and $\mathrm{L}^{\frac{4}{1-2\alpha}}(0,T; \D(\A^{\frac{1}{4}+\frac{\alpha}{2}}))$ with the respective topologies. We are now ready to verify the three steps introduced earlier. 
	
	\begin{remark}
		By taking $\mu'=\frac{1}{4}+\frac{\alpha}{2},\gamma'=\alpha$ and $\beta'=\frac{1}{2}$,  we obtain  the following  interpolation inequality for $\frac{1}{4} \leq \alpha <\frac{1}{2}$:
		\begin{align}\label{IP}
		\|\A^{\frac{1}{4}+\frac{\alpha}{2}}\v\|^{2}_{\H} \leq C\|\A^{\alpha}\v\|_{\H}\|\A^{\frac{1}{2}}\v\|_{\H}.
		\end{align}
	\end{remark}
	Let us now establish Step (i). 
	\begin{proposition}\label{prop3.6}
		If $\mathrm{Im}(\G) \subseteq \D(\A^{\frac{1}{4}+\frac{\alpha}{2}+\e'})$ for some $\e'>0$, then for arbitrary $\x \in \D(\A^{\alpha}), \; \f \in \D(\A^{-\frac{1}{2}+\alpha})$, there exists a unique strong solution of the system \eqref{2.7} such that $$ \u(\cdot,\omega) \in \C([0,T];\D(\A^{\alpha}))\cap \mathrm{L}^{\frac{4}{1-2\alpha}}(0,T;\D(\A^{\frac{1}{4}+\frac{\alpha}{2}})) $$ 
		and $$ \u(\cdot, \omega)- \z(\cdot,\omega) \in \mathrm{L}^{2}(0,T;\D(\A^{\alpha+\frac{1}{2}})),$$
		for $\mathbb{P}$ a.a. $\omega \in \Omega$. It is a Markov process satisfying the Feller property in $\D(\A^{\alpha})$.
	\end{proposition}
	\begin{proof}
	One can make the following calculations rigorous by considering a Faedo-Galerkin approximated system corresponding to the system \eqref{3.2}. 	We take the inner product with $\A^{2\alpha}\v(\cdot)$ to the first equation in \eqref{3.2} to obtain 
		\begin{align}\label{Prop2.10a}
	&	\frac{1}{2}\frac{\d}{\d t}\|\A^{\alpha}\v(t)\|_{\H}^2 + \mu\| \A^{\alpha+\frac{1}{2}}\v(t)\|_{\H}^2 \nonumber\\&=  \langle \f, \A^{2\alpha}\v(t) \rangle - \langle  \B(\v(t)+\z(t),\v(t)+\z(t)), \A^{2\alpha}\v(t) \rangle - \beta \langle \mathcal{C}(\v(t)+\z(t)), \A^{2\alpha}\v(t) \rangle. 
		\end{align} 
		Using the Cauchy-Schwarz inequality and Young's inequality, we estimate $|\langle \f, \A^{2\alpha}\v \rangle |$ as 
		\begin{align*}
		|\langle \f, \A^{2\alpha}\v \rangle| &= |\langle \A^{\alpha-\frac{1}{2}} \f, \A^{\alpha+\frac{1}{2}}\v \rangle| \leq \|\A^{\alpha-\frac{1}{2}} \f\|_{\H}\|\A^{\alpha+\frac{1}{2}}\v \|_{\H}\leq C\|\A^{\alpha-\frac{1}{2}} \f\|^{2}_{\H}+\frac{\mu}{8} \|\A^{\alpha+\frac{1}{2}}\v \|^{2}_{\H}.
		\end{align*}
		Making use of the Cauchy-Schwarz inequality, Lemma \ref{lem2.2} with  $\theta=\rho=\frac{1}{4}+\frac{\alpha}{2}$ and $\delta= -\alpha+\frac{1}{2}$, the interpolation inequality \eqref{IP} and Young's inequality, we estimate $|\langle  \B(\v+\z,\v+\z), \A^{2\alpha}\v \rangle |$ as 
		\begin{align*}
		|\langle  \B(\v+\z,\v+\z), \A^{2\alpha}\v \rangle |&= \langle \A^{\alpha-\frac{1}{2}} \B(\v+\z,\v+\z), \A^{\alpha+\frac{1}{2}}\v \rangle \no \\&\leq \|\A^{\alpha-\frac{1}{2}} \B(\v+\z,\v+\z)\|_{\H}\|\A^{\alpha+\frac{1}{2}}\v\|_{\H}\no\\&\leq  C\|\A^{\frac{1}{4}+\frac{\alpha}{2}}(\v+\z)\|^{2}_{\H}\|\A^{\alpha+\frac{1}{2}}\v\|_{\H}\\& \leq C\big(\|\A^{\frac{1}{4}+\frac{\alpha}{2}}\v\|^{2}_{\H}+\|\A^{\frac{1}{4}+\frac{\alpha}{2}}\z\|^{2}_{\H}\big)\|\A^{\alpha+\frac{1}{2}}\v\|_{\H}\nonumber\\&\leq C\|\A^{\alpha}\v\|_{\H}\|\A^{\frac{1}{2}}\v\|_{\H}\|\A^{\alpha+\frac{1}{2}}\v\|_{\H}+\|\A^{\frac{1}{4}+\frac{\alpha}{2}}\z\|^{2}_{\H}\|\A^{\alpha+\frac{1}{2}}\v\|_{\H}\nonumber\\&\leq  \frac{\mu}{8}\|\A^{\alpha+\frac{1}{2}}\v\|^{2}_{\H}+C\|\A^{\alpha}\v\|^{2}_{\H}\|\A^{\frac{1}{2}}\v\|^{2}_{\H}+C\|\A^{\frac{1}{4}+\frac{\alpha}{2}}\z\|^{4}_{\H}. 
		\end{align*}
		For the final term in the right hand side of the equality \eqref{Prop2.10a}, we use the Cauchy-Schwarz inequality and the fact that $(a+b)^p\leq 2^{p-1}(a^p+b^p)$, for all $a,b\geq 0$ and $1\leq p<\infty$ to estimate it as 
		\begin{align*} 
		|\langle \mathcal{C}(\v+\z), \A^{2\alpha}\v \rangle | &\leq \|\mathcal{C}(\v+\z)\|_{\H}\|\A^{2\alpha}\v\|_{\H}\leq 2^{r-1}\left(\|\v\|_{\wi\L^{2r}}^r+\|\z\|_{\wi\L^{2r}}^r\right)\|\A^{2\alpha}\v\|_{\H}. 
		\end{align*}
		Using  the interpolation inequality \eqref{Interpolation}  with $\mu' = 2\alpha, \gamma' = \alpha+\frac{1}{2}, \beta' = 0$ (so that $\theta' = \frac{4\alpha}{2\alpha+1}$), Gagliardo-Nirenberg's inequality \eqref{2.1} for $p=2r, m=2\alpha, j=0, n=l=q=2$ (so that $\theta=\frac{r-1}{2\alpha r}$), interpolation inequality \eqref{Interpolation} with $\mu' =  \alpha, \gamma' = \frac{1}{2}, \beta' = 0$ (so that $\theta'=2\alpha$) and Young's inequality, we estimate $\|\v\|_{\wi\L^{2r}}^r\|\A^{2\alpha}\v\|_{\H}$ as 
		\begin{align*} 
		\|\v\|_{\wi\L^{2r}}^r\|\A^{2\alpha}\v\|_{\H}&\leq \|\v\|_{\wi\L^{2r}}^r\|\A^{\alpha+\frac{1}{2}}\v\|_{\H}^{\frac{4\alpha}{2\alpha+1}}\|\v\|^{\frac{1-2\alpha}{2\alpha+1}}_{\H}\nonumber\\&\leq\frac{\mu}{2^{r+2}} \|\A^{\alpha+\frac{1}{2}}\v\|_{\H}^{2}+C\|\v\|_{\wi\L^{2r}}^{r(1+2\alpha)} \|\v\|_{\H}^{1-2\alpha}\nonumber\\&\leq\frac{\mu}{2^{r+2}} \|\A^{\alpha+\frac{1}{2}}\v\|_{\H}^{2}+C\|\A^{\alpha}\v\|_{\H}^{\frac{(r-1)(1+2\alpha)}{2\alpha}}\|\v\|_{\H}^{\frac{(1-(1-2\alpha)r)(1+2\alpha)}{2\alpha}}\|\v\|_{\H}^{1-2\alpha}\nonumber\\&=\frac{\mu}{2^{r+2}} \|\A^{\alpha+\frac{1}{2}}\v\|_{\H}^{2}+C\|\A^{\alpha}\v\|_{\H}^{r-1}\|\A^{\alpha}\v\|_{\H}^{\frac{r-1}{2\alpha}}\|\v\|_{\H}^{\frac{(1-(1-2\alpha)r)(1+2\alpha)}{2\alpha}+1-2\alpha}\nonumber\\&\leq \frac{\mu}{2^{r+2}} \|\A^{\alpha+\frac{1}{2}}\v\|_{\H}^{2}+C\|\A^{\alpha}\v\|_{\H}^{r-1}\|\A^{\frac{1}{2}}\v\|_{\H}^{r-1}\|\v\|_{\H}^{3-r+2\alpha(r-1)},
		\end{align*}
		for $r= 1,2,3$. For the term  $\|\z\|_{\wi\L^{2r}}^r\|\A^{2\alpha}\v\|_{\H}$, once again we use the interpolation inequality \eqref{IP}  with $ \mu' =  2\alpha, \gamma' = \alpha, \beta' = \alpha+\frac{1}{2}$ (so that $ \theta' = 1-2\alpha$) and Young's inequality  to obtain 
		\begin{align*}
		\|\z\|_{\wi\L^{2r}}^r\|\A^{2\alpha}\v\|_{\H}&\leq  \|\z\|^{r}_{\wi\L^{2r}}\|\A^{\alpha}\v\|^{1-2\alpha}_{\H}\|\A^{\alpha+\frac{1}{2}}\v\|^{2\alpha}_{\H}\nonumber\\&\leq  \frac{\mu}{2^{r+2}}\|\A^{\alpha+\frac{1}{2}}\v\|^{2}_{\H} +C \|\z\|^{\frac{r}{1-\alpha}}_{\wi\L^{2r}}\|\A^{\alpha}\v\|^{\frac{1-2\alpha}{1-\alpha}}_{\H}\nonumber\\&\leq  \frac{\mu}{2^{r+2}}\|\A^{\alpha+\frac{1}{2}}\v\|^{2}_{\H} +\frac{1}{2^r} \|\A^{\alpha}\v\|^{2}_{\H}+C\|\z\|^{2r}_{\wi\L^{2r}}.
		\end{align*}
		Combining the above estimates and substituting it in \eqref{Prop2.10a}, we deduce that
		\begin{align}\label{Prop3.5.1}\nonumber
	&	\frac{1}{2} \frac{\d}{\d t}\|\A^{\alpha}\v(t)\|^{2}_{\H}+\frac{\mu}{2}\|\A^{\alpha+\frac{1}{2}}\v(t)\|^{2}_{\H} \nonumber\\&\leq C\|\A^{\alpha-\frac{1}{2}} \f\|^{2}_{\H}+C\|\A^{\frac{1}{4}+\frac{\alpha}{2}}\z(t)\|^{4}_{\H}+C\|\A^{\alpha}\v(t)\|^{2}_{\H}\|\A^{\frac{1}{2}}\v(t)\|^{2}_{\H} \nonumber \\& \quad +C \|\A^{\alpha}\v(t)\|^{r-1}_{\H}\|\A^{\frac{1}{2}}\v(t)\|^{r-1}_{\H}\|\v(t)\|^{3-r+2\alpha(r-1)}_{\H}+\frac{1}{2} \|\A^{\alpha}\v(t)\|^{2}_{\H}+C\|\z(t)\|^{2r}_{\wi\L^{2r}},
		\end{align}
		for $r=1,2,3$,
for a.e. $t\in[0,T]$,		where  $C$ denotes various positive constants. For the case $r=3$, applying Gronwall's inequality in \eqref{Prop3.5.1}, we find 
		\begin{align*}
		\|\A^{\alpha}\v(t)\|^{2}_{\H}&\leq\left\{\|\A^{\alpha}\x\|_{\H}^2+CT\|\A^{\alpha-\frac{1}{2}} \f\|^{2}_{\H}+C\int_0^T\|\A^{\frac{1}{4}+\frac{\alpha}{2}}\z(t)\|^{4}_{\H}\d t+C\int_0^T\|\z(t)\|_{\wi\L^{2r}}^{2r}\d t\right\}\nonumber\\&\quad\times\exp\left\{C\left(\sup_{t\in[0,T]}\|\v(t)\|_{\H}^{4\alpha}+1\right)\int_0^T\|\A^{\frac{1}{2}}\v(t)\|_{\H}^2\d t+T\right\}\nonumber\\&\leq \left\{\|\A^{\alpha}\x\|_{\H}^2+CT\|\A^{\alpha-\frac{1}{2}} \f\|^{2}_{\H}+CT^{2\alpha}\left(\int_0^T\|\A^{\frac{1}{4}+\frac{\alpha}{2}}\z(t)\|^{\frac{4}{1-2\alpha}}_{\H}\d t\right)^{1-2\alpha}\right.\nonumber\\&\qquad\left.+CT\sup_{t\in[0,T]}\|\A^{\alpha}\z(t)\|_{\H}^{2r}\right\}\exp\left\{C\left(\sup_{t\in[0,T]}\|\v(t)\|_{\H}^{4\alpha}+1\right)\int_0^T\|\A^{\frac{1}{2}}\v(t)\|_{\H}^2\d t+T\right\},
		\end{align*}
		for all $t\in[0,T]$ and $\alpha\in[\frac{1}{3},\frac{1}{2})$, where we used \eqref{26} also. 	By  Proposition \ref{prop3.2},  we know that $\v \in  \C([0,T];\H)\cap \mathrm{L}^{2}(0,T;\V)$ and $\z\in\C_{\mathbf{0}}([0,T];\D(\A^{\alpha}))\cap\mathrm{L}^{\frac{4}{1-2\alpha}}(0,T; \D(\A^{\frac{1}{4}+\frac{\alpha}{2}}))$. Thus, the right hand side of the above inequality is finite and for $\x\in\D(\A^{\alpha})$, it is immediate that  $\v\in\mathrm{L}^{\infty}(0,T;\D(\A^{\alpha}))\cap\mathrm{L}^2(0,T;\D(\A^{\alpha+\frac{1}{2}}))$.  Now, we use the classical argument (see Chapter 3, \cite{RT}) to prove that the solution $\v\in \C([0,T];\D(\A^{\alpha}))$ by showing $\v \in \mathrm{L}^{\infty}(0,T;\D(\A^{\alpha}))\cap \mathrm{L}^{2}(0,T;\D(\A^{\frac{1}{2}+\alpha}))\cap \mathrm{H}^{1}(0,T;\D(\A^{-\frac{1}{2}+\alpha}))$. Thus, it is left to show that $
		\frac{\d \v}{\d t} \in \mathrm{L}^{2}(0,T;\D(\A^{-\frac{1}{2}+\alpha}),
		$ $\mathbb{P}$-a.s.
		For any function $\phi \in \mathrm{L}^{2}(0,T;\D(\A^{\frac{1}{2}-\alpha}))$, we have 
		\begin{align}\label{Prop3.5} \nonumber
		&	\int_{0}^{T} \bigg| \bigg\langle \frac{\d \v(t)}{\d t}, \phi(t) \bigg\rangle\bigg| \d t \nonumber\\&\leq \mu 	\int_{0}^{T} \big| \langle \A\v(t), \phi(t) \rangle\big| \d t+	\int_{0}^{T} \big| \langle \f,\phi(t) \rangle\big|\d t \nonumber \\&\quad+ 	\int_{0}^{T} \big| \langle \B(\v(t)+\z(t)), \phi(t) \rangle \big| \d t+\beta 	\int_{0}^{T} \big| \langle \mathcal{C}(\v(t)+\z(t)),\phi(t)\rangle\big| \d t\nonumber\\&\leq\mu\int_0^T\|\A^{\alpha+\frac{1}{2}}\v(t)\|_{\H} \|\A^{\frac{1}{2}-\alpha}\phi(t)\|_{\H}  \d t +\int_0^T\|\A^{\alpha-\frac{1}{2}}\f\|_{\H}\|\A^{\frac{1}{2}-\alpha}\phi(t)\|_{\H}\d t \nonumber\\&\quad+\int_{0}^{T}\|\A^{\alpha-\frac{1}{2}}\B(\v+\z)\|_{\H}\|\A^{\frac{1}{2}-\alpha}\phi(t)\|_{\H}\d t+\beta \int_{0}^{T}\|\A^{\alpha-\frac{1}{2}}\mathcal{C}(\v+\z)\|_{\H}\|\A^{\alpha-\frac{1}{2}}\phi(t)\|_{\H} \d t\nonumber\\&\leq \left\{\mu	\left(\int_{0}^{T}   \|\A^{\alpha+\frac{1}{2}}\v(t)\|^{2}_{\H}   \d t\right)^{\frac{1}{2}}+T^{\frac{1}{2}}\|\A^{\alpha-\frac{1}{2}}\f\|_{\H}+C \left( \int_{0}^{T} \|\A^{\frac{1}{4}+\frac{\alpha}{2}}(\v(t)+\z(t))\|^{2}_{\H}\d t\right)^{\frac{1}{2}}\right.\nonumber\\&\qquad\left.+\bigg(\int_{0}^{T}\|\v(t)+\z(t)\|^{2r}_{\wi\L^{2r}}\d t\bigg)^{\frac{1}{2}}\right\}\left(\int_{0}^{T}\|\A^{\frac{1}{2}-\alpha}\phi(t)\|^{2}_{\H}\d t\right)^{1/2}.
		\end{align} 
		Using Sobolev's inequality and the interpolation inequality \eqref{Interpolation}, we infer that the term inside the parenthesis in the right hand side of the inequality \eqref{Prop3.5} is finite for any  $ \phi \in \mathrm{L}^{2}(0,T;\D(\A^{\frac{1}{2}-\alpha}))$, which implies $ \frac{\d \v}{\d t} \in \mathrm{L}^{2}(0,T;\D(\A^{-\frac{1}{2}+\alpha}))$ and hence we easily obtain   $\v\in \C([0,T];\D(\A^{\alpha}))$.	Once again using the interpolation inequality \eqref{Interpolation}, we find 
		\begin{align}\label{Aestimate}
		\|\A^{\frac{\alpha}{2}+\frac{1}{4}}\v\|_{\H} \leq \|\A^{\alpha}\v\|^{\frac{1}{2}+\alpha}_{\H}\|\A^{\alpha+\frac{1}{2}}\v\|^{\frac{1}{2}-\alpha}_{\H},
		\end{align}
		which immediately gives $\v \in \mathrm{L}^{\frac{4}{1-2\alpha}}(0,T;\D(\A^{\frac{1}{4}+\frac{\alpha}{2}}))$. Therefore, we deduce that $\u=\v+\z \in \C([0,T];\D(\A^{\alpha}))\cap \mathrm{L}^{\frac{4}{1-2\alpha}}(0,T;\D(\A^{\frac{1}{4}+\frac{\alpha}{2}})), \; \mathbb{P}$-a.s. Rest of the proof for the cases $r=1,2$ can be carry out in a similar way. 
	\end{proof}

\begin{remark}\label{rem3.7}
If the assumption \eqref{1p6} is in hand, then $ \mathrm{Im}(\G)\subset\D(\A^{\frac{\ell}{2}+\e})$ implies that the solution $\z(\cdot)$ of the Ornstein-Uhlenbeck equation \eqref{3.1} has $\mathbb{P}$-a.s., $\C([0,T];\D(\A^{\frac{\ell}{2}}))$ paths. For $1<\ell<\frac{3}{2}$ and $\f\in\D(\A^{\frac{\ell-1}{2}})$, taking the inner product with $\A^{\ell}\v(\cdot)$  to the first equation in \eqref{3.2}, we find 
\begin{align}\label{310}
&\frac{1}{2}\frac{\d}{\d t}\|\A^{\frac{\ell}{2}}\v(t)\|_{\H}^2+\mu\|\A^{\frac{\ell+1}{2}}\v(t)\|_{\H}^2  \nonumber\\&=  \langle \f, \A^{\ell}\v(t) \rangle - \langle  \B(\v(t)+\z(t),\v(t)+\z(t)), \A^{\ell}\v(t) \rangle - \beta \langle \mathcal{C}(\v(t)+\z(t)), \A^{\ell}\v(t) \rangle. 
\end{align}
Using the Cauchy-Schwarz and Young's inequalities, we estimate $|\langle \f, \A^{\ell}\v \rangle| $ as 
\begin{align*}
|\langle \f, \A^{\ell}\v \rangle|=|\langle\A^{\frac{\ell-1}{2}}\f,\A^{\frac{\ell+1}{2}}\v\rangle|\leq\|\A^{\frac{\ell-1}{2}}\f\|_{\H}\|\A^{\frac{\ell+1}{2}}\v\|_{\H}\leq\frac{\mu}{8}\|\A^{\frac{\ell+1}{2}}\v\|_{\H}^2+\frac{2}{\mu}\|\A^{\frac{\ell-1}{2}}\f\|_{\H}^2. 
\end{align*}
We estimate the term $| \langle  \B(\v+\z,\v+\z), \A^{\ell}\v \rangle|$ using \eqref{2p13}, \eqref{2p14} and Young's inequality  as
\begin{align*}
| \langle  \B(\v+\z,\v+\z), \A^{\ell}\v \rangle|&=| \langle\A^{\frac{\ell-1}{2}}  \B(\v+\z,\v+\z), \A^{\frac{\ell+1}{2}}  \v \rangle|\nonumber\\&\leq\|\A^{\frac{\ell-1}{2}}  \B(\v+\z,\v+\z)\|_{\H}\|\A^{\frac{\ell+1}{2}} \v\|_{\H}\nonumber\\&\leq C\|\A^{\frac{\ell}{2}}(\v+\z)\|_{\H}^2\|\A^{\frac{\ell+1}{2}} \v\|_{\H}\nonumber\\&\leq\frac{\mu}{8}\|\A^{\frac{\ell+1}{2}} \v\|_{\H}^2+C\|\A^{\frac{\ell}{2}}\v\|_{\H}^4+C\|\A^{\frac{\ell}{2}}\z\|_{\H}^2\|\A^{\frac{\ell}{2}}\v\|_{\H}^2+C\|\A^{\frac{\ell}{2}}\z\|_{\H}^4, 
\end{align*}
provided $1<\ell<\frac{3}{2}$. Making use of \eqref{2p19}, \eqref{2p22}, and Cauchy-Schwarz, Gagliardo-Nirenberg's, Sobolev's  and Young's inequalities, we estimate $|\langle\mathcal{C}(\v+\z), \A^{\ell}\v\rangle|$ as 
\begin{align*}
& |\langle\mathcal{C}(\v+\z), \A^{\ell}\v\rangle|\nonumber\\&=|\langle \A^{\frac{\ell-1}{2}}\mathcal{C}(\v+\z), \A^{\frac{\ell+1}{2}}\v\rangle|\leq\|\A^{\frac{\ell-1}{2}}\mathcal{C}(\v+\z)\|_{\H}\|\A^{\frac{\ell+1}{2}}\v\|_{\H}\nonumber\\&\leq C\|\v+\z\|_{\wi\L^{2r}}^{r-1}\|\v+\z\|_{\mathbb{W}^{\ell-1,2r}}\|\A^{\frac{\ell+1}{2}}\v\|_{\H}\nonumber\\&\leq C\|\v\|_{\wi\L^{2r}}^{\frac{\ell r^2-\ell r+2r}{\ell r+1}}\|\A^{\frac{\ell+1}{2}}\v\|_{\H}^{\frac{2\ell r+1-r}{\ell r+1}}+C\|\v\|_{\wi\L^{2r}}^{r-1}\|\z\|_{\wi\L^{2r}}^{\frac{1}{r(\ell-1)+1}}\|\A^{\frac{\ell}{2}}\z\|_{\H}^{\frac{r(\ell-1)}{r(\ell-1)+1}}\|\A^{\frac{\ell+1}{2}}\v\|_{\H}\nonumber\\&\quad+C\|\z\|_{\wi\L^{2r}}^{r-1}\|\v\|_{\wi\L^{2r}}^{\frac{r+1}{\ell r+1}}\|\A^{\frac{\ell+1}{2}}\v\|_{\H}^{\frac{2\ell r+1-r}{\ell r+1}}+C\|\z\|_{\wi\L^{2r}}^{\frac{r[(\ell-1)(r-1)+1]}{r(\ell-1)+1}}\|\A^{\frac{\ell}{2}}\z\|_{\H}^{\frac{r(\ell-1)}{r(\ell-1)+1}}\|\A^{\frac{\ell+1}{2}}\v\|_{\H}\nonumber\\&\leq\frac{\mu}{4}\|\A^{\frac{\ell+1}{2}}\v\|_{\H}^2+ C\|\v\|_{\wi\L^{2r}}^{\frac{2(r-1)(\ell r+1)}{r+1}}\|\A^{\frac{\ell}{2}}\v\|_{\H}^2+C\|\A^{\frac{\ell}{2}}\z\|_{\H}^2\|\v\|_{\wi\L^{2r}}^{2(r-1)}\nonumber\\&\quad+C\|\A^{\frac{\ell}{2}}\z\|_{\H}^{\frac{2(\ell r+1)(r-1)}{r+1}}\|\A^{\frac{\ell}{2}}\v\|_{\H}^2+C\|\A^{\frac{\ell}{2}}\z\|_{\H}^{2r},
\end{align*}
for $\ell>1$. Combining the above estimates, substituting in \eqref{310} and then integrating from $0$ to $t$, one can deduce that 
\begin{align}
&\|\A^{\frac{\ell}{2}}\v(t)\|_{\H}^2+\mu\int_0^t\|\A^{\frac{\ell+1}{2}}\v(s)\|_{\H}^2\d s\nonumber\\&\leq \|\A^{\frac{\ell}{2}}\x\|_{\H}^2+\frac{4t}{\mu}\|\A^{\frac{\ell-1}{2}}\f\|_{\H}^2+C\int_0^t\left(\|\A^{\frac{\ell}{2}}\v(s)\|_{\H}^2+\|\A^{\frac{\ell}{2}}\z(s)\|_{\H}^2+\|\v(s)\|_{\wi\L^{2r}}^{\frac{2(r-1)(\ell r+1)}{r+1}}\right.\nonumber\\&\qquad\left.+\|\A^{\frac{\ell}{2}}\z(s)\|_{\H}^2\|\v(s)\|_{\wi\L^{2r}}^{2(r-2)}+\|\A^{\frac{\ell}{2}}\z(s)\|_{\H}^{\frac{2(\ell r+1)(r-1)}{r+1}}\right)\|\A^{\frac{\ell}{2}}\v(s)\|_{\H}^2\d s\nonumber\\&\quad+C\int_0^t\left(\|\A^{\frac{\ell}{2}}\z(s)\|_{\H}^4+\|\A^{\frac{\ell}{2}}\z(s)\|_{\H}^{2r}\right)\d s,
\end{align}
for all $t\in[0,T]$. An application of Gronwall's inequality easily implies
\begin{align}
&\sup_{t\in[0,T]}\|\A^{\frac{\ell}{2}}\v(t)\|_{\H}^2+\mu\int_0^T\|\A^{\frac{\ell+1}{2}}\v(t)\|_{\H}^2\d t\nonumber\\&\leq \left\{\|\A^{\frac{\ell}{2}}\x\|_{\H}^2+\frac{4T}{\mu}\|\A^{\frac{\ell-1}{2}}\f\|_{\H}^2+C\int_0^T\left(\|\A^{\frac{\ell}{2}}\z(t)\|_{\H}^4+\|\A^{\frac{\ell}{2}}\z(t)\|_{\H}^{2r}\right)\d t\right\}\nonumber\\&\quad\times\exp\left\{C\int_0^T\left[\|\A^{\frac{\ell}{2}}\v(t)\|_{\H}^2+\|\A^{\frac{r-1}{2r}}\v(t)\|_{\H}^{\frac{2(r-1)(\ell r+1)}{r+1}}+\|\A^{\frac{\ell}{2}}\z(t)\|_{\H}^2\left(1+\|\A^{\frac{r-1}{2r}}\v(t)\|_{\H}^{2(r-2)}\right)\right.\right.\nonumber\\&\qquad\left.\left.+\|\A^{\frac{\ell}{2}}\z(t)\|_{\H}^{\frac{2(\ell r+1)(r-1)}{r+1}}\right]\d t\right\},
\end{align} 
 where we used Sobolev's inequality also. Since $\frac{1}{2}<\frac{\ell}{2}<\frac{3}{4}\leq2\alpha+\frac{1}{2}<1$, we easily have $\int_0^T\|\A^{\frac{\ell}{2}}\v(t)\|_{\H}^2\d t\leq \int_0^T\|\A^{2\alpha+\frac{1}{2}}\v(t)\|_{\H}^2\d t<\infty$, and hence we find $\v\in\mathrm{L}^{\infty}(0,T;\D(\A^{\frac{\ell}{2}}))\cap\mathrm{L}^2(0,T;\D(\A^{\frac{\ell+1}{2}}))$, $\mathbb{P}$-a.s., for $1<\ell<\frac{3}{2}$. A calculation similar to the estimate \eqref{Prop3.5} gives $\frac{\d\v}{\d t}\in\mathrm{L}^2(0,T;\D(\A^{\frac{\ell-1}{2}}))$, $\mathbb{P}$-a.s., and one can conclude that $\v\in\C([0,T];\D(\A^{\frac{\ell}{2}}))\cap\mathrm{L}^2(0,T;\D(\A^{\frac{\ell+1}{2}}))$, and $\u=\v+\z\in \C([0,T];\D(\A^{\frac{\ell}{2}}))$, $\mathbb{P}$-a.s.

For $\ell=1$, using the assumption \eqref{1p7}, we know that $\mathrm{Im}(\G)\subset\D(\A^{\frac{1}{2}+\e_1})$, and hence $\z\in\C([0,T];\D(\A^{\frac{1}{2}+\e})),$ $\mathbb{P}$-a.s., for some $\e>0$. Choosing $\delta=0$, $\theta=\frac{1}{2}$ and $\rho=\frac{1}{2}+\e$ in \eqref{2.14}, we get $\B:\D(\A^{\frac{1}{2}})\times\D(\A^{\frac{1}{2}+\e})\to\H$.
From \eqref{310}, we infer that 
\begin{align*}
&\frac{1}{2}\frac{\d}{\d t}\|\A^{\frac{1}{2}}\v(t)\|_{\H}^2+\mu\|\A\v(t)\|_{\H}^2  \nonumber\\&=  \langle \f, \A\v(t) \rangle - \langle  \B(\v(t)+\z(t),\v(t)+\z(t)), \A\v(t) \rangle - \beta \langle \mathcal{C}(\v(t)+\z(t)), \A\v(t) \rangle\nonumber\\&\leq\|\f\|_{\H}\|\A\v(t)\|_{\H}+C\|\v(t)\|_{\H}^{1/2}\|\A^{\frac{1}{2}}\v(t)\|_{\H}\|\A\v(t)\|_{\H}^{3/2}+C\|\v(t)\|_{\H}^{1/2}\|\A^{\frac{1}{2}}\z(t)\|_{\H}\|\A\v(t)\|_{\H}^{3/2}\nonumber\\&\quad+C\|\z(t)\|_{\H}^{1/2}\|\A^{\frac{1}{2}}\z(t)\|_{\H}^{1/2}\|\A^{\frac{1}{2}}\v(t)\|_{\H}^{1/2}\|\A\v(t)\|_{\H}^{3/2}+C\|\A^{\frac{1}{2}}\z(t)\|_{\H}\|\A^{\frac{1}{2}+\e}\z(t)\|_{\H}\|\A\v(t)\|_{\H}\nonumber\\&\quad+\beta\|\v(t)+\z(t)\|_{\wi\L^{2r}}^{r}\|\A\v(t)\|_{\H}\nonumber\\&\leq\frac{\mu}{2}\|\A\v(t)\|_{\H}^2+\frac{1}{2\mu}\|\f\|_{\H}^2+C\|\v(t)\|_{\H}^2\|\A^{\frac{1}{2}}\v(t)\|_{\H}^4+C\|\A^{\frac{1}{2}}\z(t)\|_{\H}^4\|\A^{\frac{1}{2}}\v(t)\|_{\H}^2\nonumber\\&\quad+C\|\A^{\frac{1}{2}}\z(t)\|_{\H}^2\|\A^{\frac{1}{2}+\e}\z(t)\|_{\H}^2+C\|\v(t)\|_{\H}^2\|\A^{\frac{1}{2}}\v(t)\|_{\H}^{2(r-1)}+C\|\A^{\frac{1}{2}}\z(t)\|_{\H}^{2r}, 
\end{align*}
where we used \eqref{219}, \eqref{220}. Integrating the above inequality from $0$ to $t$, we find 
\begin{align*}
&\|\A^{\frac{1}{2}}\v(t)\|_{\H}^2+\mu\int_0^t\|\A\v(s)\|_{\H}^2\d s\nonumber\\&\leq\|\A^{\frac{1}{2}}\x\|_{\H}^2+\frac{t}{\mu}\|\f\|_{\H}^2+C\int_0^t\left(\|\v(s)\|_{\H}^2\|\A^{\frac{1}{2}}\v(s)\|_{\H}^2+\|\A^{\frac{1}{2}}\z(s)\|_{\H}^4\right.\nonumber\\&\qquad\left.+\|\v(s)\|_{\H}^2\|\A^{\frac{1}{2}}\v(s)\|_{\H}^{2(r-2)}\right)\|\A^{\frac{1}{2}}\v(s)\|_{\H}^2\d s\nonumber\\&\quad+C\int_0^t\left(\|\A^{\frac{1}{2}}\z(s)\|_{\H}^2\|\A^{\frac{1}{2}+\e}\z(s)\|_{\H}^2+\|\A^{\frac{1}{2}}\z(s)\|_{\H}^{2r}\right)\d s,
\end{align*}
for all $t\in[0,T]$. An application of Gronwall's inequality yields 
\begin{align*}
&\sup_{t\in[0,T]}\|\A^{\frac{1}{2}}\v(t)\|_{\H}^2+\mu\int_0^T\|\A\v(t)\|_{\H}^2\d t\nonumber\\&\leq\left\{\|\A^{\frac{1}{2}}\x\|_{\H}^2+\frac{T}{\mu}\|\f\|_{\H}^2+C\int_0^t\left(\|\A^{\frac{1}{2}}\z(t)\|_{\H}^2\|\A^{\frac{1}{2}+\e}\z(t)\|_{\H}^2+\|\A^{\frac{1}{2}}\z(t)\|_{\H}^{2r}\right)\d t\right\}\nonumber\\&\quad\times\exp\left\{C\int_0^T\left(\|\v(t)\|_{\H}^2\|\A^{\frac{1}{2}}\v(t)\|_{\H}^2+\|\A^{\frac{1}{2}}\z(t)\|_{\H}^4+\|\v(t)\|_{\H}^2\|\A^{\frac{1}{2}}\v(t)\|_{\H}^{2(r-2)}\right)\d t\right\},
\end{align*}
and the right hand side is finite for $r=2,3$. Thus, one can easily conclude that $\v\in\C([0,T];\D(\A^{\frac{1}{2}}))\cap\mathrm{L}^2(0,T;\D(\A))$, and $\u=\v+\z\in \C([0,T];\D(\A^{\frac{1}{2}}))$, $\mathbb{P}$-a.s.
\end{remark}

	Now, we are in a position  to verify 	Steps (ii) and (iii). We connect the irreducibility with a control problem (cf. \cite{GDJZ}). 
	Given $\x, \y \in \D(\A^{\alpha}),\ T>0$, choose any $0<t_{0}<t_{1}<T$, let us define $\bar{\u},$ which connects $\x$ to $\y$ in the interval $[0,T]$ as 
	\begin{align*}
	\bar{\u}(t) &=e^{-t\A}\x, \quad t\in [0,t_{0}],\\
	\bar{\u}(t)& = e^{-(T-t)\A}\y, \quad t\in [t_{1},T], \\
	\bar{\u}(t)&= \bar{\u}(t_{0}) + \frac{t-t_{0}}{t_{1}-t_{0}}(\bar{\u}(t_{1})-\bar{\u}(t_{0})), \ t\in (t_{0},t_{1}). 
	\end{align*}
	For $\x\in\D(\A^{\alpha})$, using spectral theorem, we obtain 
	\begin{align*}
	\int_0^T\|\A^{\alpha+\frac{1}{2}}e^{-t\A}\x\|_{\H}^2\d t=\sum_{j=1}^{\infty}\lambda_j^{2\alpha+1}\int_0^Te^{-2t\lambda_j}|(\x,e_j)|^2\d t\leq\frac{1}{2}\sum_{j=1}^{\infty}\lambda_j^{2\alpha}|(\x,e_j)|^2=\frac{1}{2}\|\A^{\alpha}\x\|_{\H}^2<\infty.
	\end{align*}
	Thus, one can deduce that the function $\bar{\u} \in \C([0,T];\D(\A^{\alpha}))\cap \mathrm{L}^{2}(0,T;\D(\A^{\alpha+\frac{1}{2}}))$ and using the interpolation inequality \eqref{Aestimate}, we further have  $\bar{\u} \in \mathrm{L}^{\frac{4}{1-2\alpha}}(0,T;\D(\A^{\frac{1}{4}+\frac{\alpha}{2}}))$. Now, we consider the solution $\bar{\v}$ of the following system:
	\begin{equation}
	\left\{
	\begin{aligned}\label{vsystem}
	\frac{\d\bar{\v}}{\d t}+\mu\A \bar{\v}  &= -\B(\bar{\u},\bar{\u})+\f- \beta\mathcal{C}(\bar{\u}), \\ 
	\bar{\v}(0)  &=\x.
	\end{aligned} 
	\right.
	\end{equation}
	Given $\x \in \D(\A^{\alpha})$ and $\f \in \D(\A^{-\frac{1}{2}+\alpha}) $, the solution $ \bar{\v} \in \C([0,T];\D(\A^{\alpha}))\cap \mathrm{L}^{2}(0,T;\D(\A^{\alpha+\frac{1}{2}}))$, since $-\B(\bar{\u},\bar{\u})- \beta\mathcal{C}(\bar{\u})\in\mathrm{L}^2(0,T;\D(\A^{-\frac{1}{2}+\alpha})) $. 
	Therefore $\bar{\z} =\bar{\u}-\bar{\v} \in  \C_{\mathbf{0}}([0,T];\D(\A^{\alpha}))\cap \mathrm{L}^{\frac{4}{1-2\alpha}}(0,T;\D(\A^{\frac{1}{4}+\frac{\alpha}{2}}))$.
	
	Next, we prove that	$\Phi(\cdot)$ is continuous in the assigned topologies. 
	Suppose we are given $\z^{(1)}, \z^{(2)} \in \C_{\mathbf{0}}([0,T];\D(\A^{\alpha}))\cap \mathrm{L}^{\frac{4}{1-2\alpha}}(0,T;\D(\A^{\frac{1}{4}+\frac{\alpha}{2}}))$ and $\v^{(1)}, \v^{(2)} \in
	\C([0,T];\D(\A^{\alpha}))\cap \mathrm{L}^{2}(0,T;\D(\A^{\alpha+\frac{1}{2}}))$. Let us denote $\mathrm{V} = \v^{(1)}- \v^{(2)}$ \text{ and } $\Z=\z^{(1)}- \z^{(2)}$, then $\mathrm{V}$ satisfies the following abstract form in the weak sense: 
	\begin{equation}\label{Vsystem}
	\left\{
	\begin{aligned}
	\frac{\d\mathrm{V}}{\d t} +\mu \A \mathrm{V} +\beta (\mathcal{C}(\v^{(1)}+\z^{(1)})-\mathcal{C}(\v^{(2)}+\z^{(2)})) & + \B(\v^{(1)}+\z^{(1)},\v^{(1)}+\z^{(1)})\\&-\B(\v^{(2)}+\z^{(2)},\v^{(2)}+\z^{(2)}) = \mathbf{0}, \\
	& \hspace{4cm}\mathrm{V}(0)=\mathbf{0}.
	\end{aligned}
	\right.
	\end{equation}
	The properties  of bilinear operator $\B(\cdot,\cdot)$ implies \begin{align*}
	&\B(\v^{(1)}+\z^{(1)},\v^{(1)}+\z^{(1)})-\B(\v^{(2)}+\z^{(2)},\v^{(2)}+\z^{(2)}) \nonumber\\&  = \B(\v^{(1)}+\z^{(1)},\mathrm{V}+\Z)+\B(\mathrm{V}+\Z,\v^{(2)}+\z^{(2)}).
	\end{align*}Using this in the first equation of \eqref{Vsystem}, we obtain
	\begin{align*}
	\frac{\d}{\d t}\mathrm{V}+\mu \A \mathrm{V}+\beta (\mathcal{C}(\v^{(1)}+\z^{(1)})-\mathcal{C}(\v^{(2)}+\z^{(2)}))&+\B(\v^{(1)}+\z^{(1)},\mathrm{V}+\Z)\\&+\B(\mathrm{V}+\Z,\v^{(2)}+\z^{(2)})=\mathbf{0},
	\end{align*}
for a.e. $t\in[0,T]$.	Taking the inner product of the above expression with $\A^{2\alpha}\mathrm{V}(\cdot)$, we find
	\begin{align}\label{Step3}\nonumber
	&	\frac{1}{2}\frac{\d}{\d t}\|\A^{\alpha}\mathrm{V}(t)\|_{\H}^2 + \mu\|\A^{\frac{1}{2}+\alpha}\mathrm{V}(t)\|_{\H}^2\nonumber\\&= -\langle \B(\v^{(1)}(t)+\z^{(1)}(t),\mathrm{V}(t)+\Z(t)) + \B(\mathrm{V}(t)+\Z(t),\v^{(2)}(t)+\z^{(2)}(t)) ,\A^{2\alpha}\mathrm{V}(t) \rangle\nonumber \\& \quad -\beta \langle (\mathcal{C}(\v^{(1)}(t)+\z^{(1)}(t))-\mathcal{C}(\v^{(2)}+\z^{(2)}(t))), \A^{2\alpha}\mathrm{V}(t) \rangle,
	\end{align}
	for a.e. $t\in[0,T]$.	Using the Cauchy-Schwarz inequality, Lemma \ref{lem2.2} with  $\theta=\rho=\frac{1}{4}+\frac{\alpha}{2}$ and $\delta= -\alpha+\frac{1}{2}$, interpolation inequality and Young's inequality, we estimate the first term from the right hand side of the equality \eqref{Step3} as 
	\begin{align*}
	&	|\langle \B(\v^{(1)}+\z^{(1)},\mathrm{V}+\Z) + \B(\mathrm{V}+\Z,\v^{(2)}+\z^{(2)}) ,\A^{2\alpha}\mathrm{V} \rangle| \nonumber\\&= |\langle \A^{\alpha-\frac{1}{2}}\{\B(\v^{(1)}+\z^{(1)},\mathrm{V}+\Z)+ \B(\mathrm{V}+\Z,\v^{(2)}+\z^{(2)})\} ,\A^{\alpha+\frac{1}{2}}\mathrm{V} \rangle| \\& \leq C \left[\|\A^{\frac{1}{4}+\frac{\alpha}{2}}(\v^{(1)}+\z^{(1)})\|_{\H}+\|\A^{\frac{1}{4}+\frac{\alpha}{2}}(\v^{(2)}+\z^{(2)})\|_{\H} \right](\|\A^{\frac{1}{4}+\frac{\alpha}{2}}\mathrm{V}\|_{\H}+\|\A^{\frac{1}{4}+\frac{\alpha}{2}}\Z\|_{\H})\|\A^{\frac{1}{2}+\alpha}\mathrm{V}\|_{\H}\\&\leq    C \left[\|\A^{\frac{1}{4}+\frac{\alpha}{2}}(\v^{(1)}+\z^{(1)})\|_{\H}+\|\A^{\frac{1}{4}+\frac{\alpha}{2}}(\v^{(2)}+\z^{(2)})\|_{\H} \right]\nonumber\\&\quad\times\left(\|\A^{\frac{1}{2}+\alpha}\mathrm{V}\|_{\H}^{\frac{1}{2}}\|\mathrm{V}\|_{\H}^{\frac{1}{2}}+\|\A^{\frac{1}{4}+\frac{\alpha}{2}}\Z\|_{\H}\right)\|\A^{\frac{1}{2}+\alpha}\mathrm{V}\|_{\H}\\&  \leq \frac{\mu}{4}\|\A^{\frac{1}{2}+\alpha}\mathrm{V}\|^{2}_{\H}+ C \left[\|\A^{\frac{1}{4}+\frac{\alpha}{2}}(\v^{(1)}+\z^{(1)})\|^{4}_{\H}+\|\A^{\frac{1}{4}+\frac{\alpha}{2}}(\v^{(2)}+\z^{(2)})\|^{4}_{\H}\right]\|\A^{\alpha}\mathrm{V}\|^{2}_{\H}\\&\quad+C\left[\|\A^{\frac{1}{4}+\frac{\alpha}{2}}(\v^{(1)}+\z^{(1)})\|^{2}_{\H}+\|\A^{\frac{1}{4}+\frac{\alpha}{2}}(\v^{(2)}+\z^{(2)})\|^{2}_{\H}\right]\|\A^{\frac{1}{4}+\frac{\alpha}{2}}\Z\|^{2}_{\H} .
	\end{align*}
	For the final term in the right hand side of the equality \eqref{Step3}, we use the Cauchy-Schwarz inequality, Taylor's formula, interpolation inequality \eqref{Interpolation} with $\mu' =  2\alpha, \gamma' = \alpha+\frac{1}{2}, \beta' = \alpha$ (so that $\theta'=2\alpha$), Sobolev's inequality (for $\frac{r-1}{2r}\leq\alpha$) and Young's inequality  to estimate it as
	\begin{align*}
	&	|\langle \mathcal{C}(\v^{(1)}+\z^{(1)})-\mathcal{C}(\v^{(2)}+\z^{(2)}), \A^{2\alpha}\mathrm{V} \rangle| \\&=\left|\left<\int_0^1\mathcal{C}'(\theta(\v^{(1)}+\z^{(1)})+(1-\theta)(\v^{(2)}+\z^{(2)}))(\mathrm{V}+\Z)\d\theta, \A^{2\alpha}\mathrm{V}\right>\right|\\&\leq \int_0^1\|\mathcal{C}'(\theta(\v^{(1)}+\z^{(1)})+(1-\theta)(\v^{(2)}+\z^{(2)}))(\mathrm{V}+\Z)\|_{\H}\|\A^{2\alpha}\mathrm{V}\|_{\H}\d\theta\\&\leq r\sup_{0\leq\theta\leq 1}\|\theta(\v^{(1)}+\z^{(1)})+(1-\theta)(\v^{(2)}+\z^{(2)})\|_{\wi\L^{2r}}^{r-1}\|\mathrm{V}+\Z\|_{\wi\L^{2r}}\|\A^{\alpha+\frac{1}{2}}\mathrm{V}\|_{\H}^{2\alpha}\|\A^{\alpha}\mathrm{V}\|_{\H}^{1-2\alpha}\\&  \leq C (\|\A^{\alpha}(\v^{(1)}+\z^{(1)})\|^{r-1}_{\H}+\|\A^{\alpha}(\v^{(2)}+\z^{(2)})\|^{r-1}_{\H})(\|\A^{\alpha}\mathrm{V}\|_{\H}+\|\A^{\alpha}\Z\|_{\H})\|\A^{\alpha+\frac{1}{2}}\mathrm{V}\|_{\H}^{2\alpha}\|\A^{\alpha}\mathrm{V}\|_{\H}^{1-2\alpha}  \\&
	\leq  \frac{\mu}{4}\|\A^{\alpha+\frac{1}{2}}\mathrm{V}\|^{2}_{\H}+C\left(\|\A^{\alpha}(\v^{(1)}+\z^{(1)})\|^{2(r-1)}_{\H}+\|\A^{\alpha}(\v^{(2)}+\z^{(2)})\|^{2(r-1)}_{\H}\right)\nonumber\\&\quad\times\left(\|\A^{\alpha}\mathrm{V}\|_{\H}^2+\|\A^{\alpha}\mathrm{V}\|_{\H}^{\frac{1-2\alpha}{1-\alpha}}\|\A^{\alpha}\Z\|_{\H}^{\frac{1}{1-\alpha}}\right)\nonumber\\&\leq \frac{\mu}{4}\|\A^{\alpha+\frac{1}{2}}\mathrm{V}\|^{2}_{\H}+C\left(\|\A^{\alpha}(\v^{(1)}+\z^{(1)})\|^{2(r-1)}_{\H}+\|\A^{\alpha}(\v^{(2)}+\z^{(2)})\|^{2(r-1)}_{\H}\right)\left(\|\A^{\alpha}\mathrm{V}\|_{\H}^2+\|\A^{\alpha}\Z\|_{\H}^{2}\right).
	\end{align*} 
	Making use of all the above estimates in \eqref{Step3}, we deduce that 
	\begin{align*}
	&	\frac{\d }{\d t}\|\A^{\alpha}\mathrm{V}(t)\|^{2}_{\H} +\mu \|\A^{\frac{1}{2}+\alpha}\mathrm{V}(t)\|_{\H}^2\\&\leq C\left[\|\A^{\frac{1}{4}+\frac{\alpha}{2}}(\v^{(1)}(t)+\z^{(1)}(t))\|^{4}_{\H}+\|\A^{\frac{1}{4}+\frac{\alpha}{2}}(\v^{(2)}(t)+\z^{(2)}(t))\|^{4}_{\H}\right.\nonumber\\&\quad+\left.\|\A^{\alpha}(\v^{(1)}(t)+\z^{(1)}(t))\|^{2(r-1)}_{\H}+\|\A^{\alpha}(\v^{(2)}(t)+\z^{(2)}(t))\|^{2(r-1)}_{\H}\right]\|\A^{\alpha}\mathrm{V}(t)\|^{2}_{\H}\\& \quad +C\left[\|\A^{\frac{1}{4}+\frac{\alpha}{2}}(\v^{(1)}(t)+\z^{(1)}(t))\|^{2}_{\H}+\|\A^{\frac{1}{4}+\frac{\alpha}{2}}(\v^{(2)}(t)+\z^{(2)}(t))\|^{2}_{\H}\right]\|\A^{\frac{1}{4}+\frac{\alpha}{2}}\Z(t)\|^{2}_{\H}
	\\& \quad +C \left[\|\A^{\alpha}(\v^{(1)}(t)+\z^{(1)}(t))\|^{2(r-1)}_{\H}+\|\A^{\alpha}(\v^{(2)}(t)+\z^{(2)}(t))\|^{2(r-1)}_{\H}\right]\|\A^{\alpha}\Z(t)\|_{\H}^{2},	\end{align*}
	for a.e. $t\in[0,T]$. 
	Making use of Gronwall's inequality, we obtain 
	\begin{align*}
	&	\|\A^{\alpha}\mathrm{V}(t)\|^{2}_{\H} \nonumber\\&\leq\left\{ C\int_0^T\left[\|\A^{\frac{1}{4}+\frac{\alpha}{2}}(\v^{(1)}(t)+\z^{(1)}(t))\|^{2}_{\H}+\|\A^{\frac{1}{4}+\frac{\alpha}{2}}(\v^{(2)}(t)+\z^{(2)}(t))\|^{2}_{\H}\right]\|\A^{\frac{1}{4}+\frac{\alpha}{2}}\Z(t)\|^{2}_{\H}\d t\right.
	\\& \quad \left.+C\int_0^T \left[\|\A^{\alpha}(\v^{(1)}(t)+\z^{(1)}(t))\|^{2(r-1)}_{\H}+\|\A^{\alpha}(\v^{(2)}(t)+\z^{(2)}(t))\|^{2(r-1)}_{\H}\right]\|\A^{\alpha}\Z(t)\|_{\H}^{2}\d t\right\}\nonumber\\&\quad\times\exp\bigg\{C\int_0^T\left[\|\A^{\frac{1}{4}+\frac{\alpha}{2}}(\v^{(1)}(t)+\z^{(1)}(t))\|^{4}_{\H}+\|\A^{\frac{1}{4}+\frac{\alpha}{2}}(\v^{(2)}(t)+\z^{(2)}(t))\|^{4}_{\H}\right.\nonumber\\&\qquad+\left.\|\A^{\alpha}(\v^{(1)}(t)+\z^{(1)}(t))\|^{2(r-1)}_{\H}+\|\A^{\alpha}(\v^{(2)}(t)+\z^{(2)}(t))\|^{2(r-1)}_{\H}\right]\d t\bigg\},
	\end{align*}
	for all $t\in[0,T]$. Using the fact that $\v^{(i)},\z^{(i)}\in \C([0,T];\D(\A^{\alpha}))\cap \mathrm{L}^{\frac{4}{1-2\alpha}}(0,T;\D(\A^{\frac{1}{4}+\frac{\alpha}{2}})),$ for $i=1,2$, we conclude that 
	\begin{align}
	\sup_{t\in[0,T]}	\|\A^{\alpha}\mathrm{V}(t)\|_{\H}\leq C\left(\int_0^T\|\A^{\frac{1}{4}+\frac{\alpha}{2}}\Z(t)\|_{\H}^4\d t\right)^{1/4}+C\sup_{t\in[0,T]}\|\A^{\alpha}\Z(t)\|_{\H}.
	\end{align}
	Furthermore, we have 
	\begin{align*}
	\|\v^{(1)}-\v^{(2)}\|_{\C([0,T];\D(\A^{\alpha}))} \leq C\left( \|\z^{(1)}-\z^{(2)}\|_{\mathrm{L}^{\frac{4}{1-2\alpha}}(0,T;\D(\A^{\frac{1}{4}+\frac{\alpha}{2}}))}+\|\z^{(1)}-\z^{(2)}\|_{\C([0,T];\D(\A^{\alpha}))}\right).
	\end{align*}We know that $\u^{(i)} = \v^{(i)}+\z^{(i)},$  for $ i=1,2$. Thus, it is immediate that 
	\begin{align*}
	\|\u^{(1)}-\u^{(2)}\|_{\C([0,T];\D(\A^{\alpha}))}&
	\leq \|\v^{(1)}-\v^{(2)}\|_{\C([0,T];\D(\A^{\alpha}))}+\|\z^{(1)}-\z^{(2)}\|_{\C([0,T];\D(\A^{\alpha}))}\\& \leq C\left( \|\z^{(1)}-\z^{(2)}\|_{\mathrm{L}^{\frac{4}{1-2\alpha}}(0,T;\D(\A^{\frac{1}{4}+\frac{\alpha}{2}}))}+\|\z^{(1)}-\z^{(2)}\|_{\C([0,T];\D(\A^{\alpha}))}\right),
	\end{align*}
	and hence irreducibility follows. In fact, if $B_{\rho}(\y)$ denotes the ball centered at $\y$ with radius $\rho$ in $\D(\A^{\alpha})$, then $\mathrm{P}(T,\x,B_{\rho}(\y)) = \mathbb{P}\{\|\u(T,\x)-\y\|_{\D(\A^{\alpha})}<\rho\} $ and \begin{align*}
	\|\u(T,\x)-\y\|_{\D(\A^{\alpha})} &= \|\u(T,\x)-\bar{\u}(T,\x)\|_{\D(\A^{\alpha})} \leq \|\u(\cdot,\x)-\bar{\u}(\cdot,\x)\|_{\C([0,T];\D(\A^{\alpha}))} \\&
	\leq C\left( \|\z^{(1)}-\z^{(2)}\|_{\mathrm{L}^{\frac{4}{1-2\alpha}}(0,T;\D(\A^{\frac{1}{4}+\frac{\alpha}{2}}))}+\|\z^{(1)}-\z^{(2)}\|_{\C([0,T];\D(\A^{\alpha}))}\right),
	\end{align*}
and from the above relation, one can complete the proof of irreducibility. 
	\subsection{Strong Feller property}
	In this subsection, our aim is to prove the following result, which is the Storng Feller property. 
	\begin{proposition}\label{prop3.7}
		If $ \|\x-\y\|_{\D(\A^{\alpha})} \to 0,$ then $$ |(\P_{t}\varphi)(\x)-(\P_{t}\varphi)(\y)| \to 0, \ \text{ for arbitrary } \ t>0, \ \varphi \in \mathcal{B}_{b}(\D(\A^{\alpha})).$$ 
	\end{proposition}
	Firstly, we consider a finite dimensional Galerkin approximated system and prove that the associated Markov semigroup $\P_{t}^{R}$  is Lipschitz Feller, then as $R\to \infty$, we will prove that $\P_{t}$ is strong Feller.
	
	Let $\H_{n} =\mathrm{span}\{e_{1},e_{2},\ldots ,e_{n}\}$, where $\{e_1,e_2,\ldots,e_n,\ldots\} $ is the complete orthonormal system of eigenfunctions of the Stokes operator $\A$ and let $\Pi_{n}:\H\to\H_{n}$ be the orthogonal projection operator, that is,  $\Pi_{n}\x = \sum_{j=1}^{n}(\x,e_{j})$. We define 
	\begin{align*}
	\B_{n}(\x) = \Pi_{n}\B(\Pi_{n}\x,\Pi_{n}\x),\quad  \mathcal{C}_{n}(\x)= \Pi_{n}\mathcal{C}(\Pi_{n}\x), \quad \G_{n}=\Pi_{n}\G\Pi_{n}, \quad \f_{n} =\Pi_{n}\f.
	\end{align*}
	For fixed $R>0$ and $n>0$, we consider the following Galerkin approximated system in $\H_{n}$: 
	\begin{equation}\label{2.15}
	\left\{
	\begin{aligned}
	\d \u_{n}^{R}(t)&+\left\{\mu \A \u_{n}^{R} (t)+ \Theta_{R}(\|\A^{\alpha}\u_n^R(t)\|^{2}_{\H})\left[\B_{n}(\u_{n}^{R}(t),\u_{n}^{R}(t))+\beta \mathcal{C}_{n}(\u_{n}^{R}(t))\right]\right\}\d t\\& = \f_{n} \d t +\G_{n}\d \W(t),\\ 
	\u_{n}(0) & =\Pi_{n}\x = \x_{n},
	\end{aligned} 
	\right. 
	\end{equation}
	where for $R \in (0,\infty)$, $\Theta_{R}$ (cut-off function)  is a $\C_{0}^{\infty}$ function such that
	\begin{align*}
	\Theta_R(x)=	\left\{\begin{array}{cl}1&\text{ if }\  x\in[-R,R],\\
	0&\text{ if }\ x\in[-R-1,R+1]^c.\end{array}\right.
	\end{align*}
	If $R=\infty$, then $\Theta_R \equiv 1$ and the above system \eqref{2.15} changes to the finite dimensional version of \eqref{2.7} (see \eqref{5.8} below). 	Note that $\G_{n} \W(t)$ is an $n$-dimensional Wiener process with incremental covariance $\G_{n}\G_{n}^{*}$.
	This finite dimensional system has a unique progressively measurable solution with $\mathbb{P}$-a.s. trajectory $\u_{n}^{R}(\cdot,\x)\in \C([0,T];\H_{n})$ (using the locally Lipschitz properties of the operators $\B(\cdot)$ and $\mathcal{C}(\cdot)$, see \eqref{2p9} and \eqref{214}, and using an energy estimate satisfied by $\u_n^R(\cdot)$ by applying the finite dimensional It\^o formula, see \eqref{4.18} below).  Also, the solution generates a Markov process in $\H_{n}$ with the associated Markov semigroup $\P_{t,n}^{R}$ and is defined as $$\P_{t,n}^{R}\varphi(\x) = \mathbb{E}[\varphi(\u_{n}^{R}(t,\x))],\ \text{ for all }\ \varphi \in \C_{b}(\H_{n}).$$ 
	\begin{lemma}
		With the assumption of Theorem \ref{main} and for any $t>0, \ R>0,$ there exists a constant $L=L(R,T)$ such that 
		\begin{align}\label{3p14}
		|\P_{t,n}^{R}\varphi(\x)-\P_{t,n}^{R}\varphi(\y)| \leq L\|\x-\y\|_{\D(\A^{\alpha})},
		\end{align}
		for all $n\in \mathbb{N}, \x,\y \in \H_{n}$ and all $\varphi \in \C_{b}(\H_{n})$ with $\|\varphi\|_{\sup} \leq 1$.
	\end{lemma}
	\begin{proof}
		Since	Feller property holds, we know that $\P_{t,n}^{R}:\C_{b}(\H_{n})\to \C_{b}(\H_{n})$. Moreover, applying the Mean Value Theorem, we find  
		\begin{align}\label{315}
		|\P_{t,n}^{R}\varphi(\x)-\P_{t,n}^{R}\varphi(\y)| \leq \sup_{\k,\h\in \H_{n},\|\A^{\alpha}\h\|_{\H}\leq 1}|\D\P_{t,n}^{R}\varphi(\k)\cdot \h |\|\x-\y\|_{\D(\A^{\alpha})},
		\end{align}
		where $\D\P_{t,n}^{R}\varphi(\k)\cdot \h$ denotes the derivative of the mapping $\y\mapsto \P_{t,n}^{R}\varphi(\y)$ at the point $\k$ in the direction $\h$. The well-known Bismut-Elworthy-Li formula (\cite{GDP,BEL}) provides
		\begin{align*}
		[\D\P_{t,n}^{R}\varphi(\x) ]\cdot \h = \frac{1}{t}\mathbb{E} \left[\varphi(\u_{n}^{R}(t,\x))\int_{0}^{t} ( (\G_{n}\G_{n}^{*})^{-\frac{1}{2}}[\D\u_{n}^{R}(s,\x)]\cdot \h, \d \beta_{n}(s) ) \right],
		\end{align*}
		for all $\h\in \H_{n}$, where $\beta_{n}$ is a $n$-dimensional standard Wiener process. Then, an application of H\"older's inequality and It\^o's isometry yields 
		\begin{align}\label{314}
		|\D\P_{t,n}^{R}\varphi(\x) \cdot \h| &\leq \frac{1}{t}\|\varphi\|_{\sup}\left[\mathbb{E}\left|\int_{0}^{t}( (\G_{n}\G_{n}^{*})^{-\frac{1}{2}}\D\u_{n}^{R}(s,\x)\cdot \h, \d \beta_{n}(s) )  \right|^{2} \right]^\frac{1}{2}\nonumber\\&\leq 	\frac{1}{t}\|\varphi\|_{\sup}\left[\mathbb{E}\int_{0}^{t}\|(\G_{n}\G_{n}^{*})^{-\frac{1}{2}}\D\u_{n}^{R}(s,\x)\cdot \h\|^{2}_{\H}\d s\right]^\frac{1}{2}.
		\end{align}
		Our next aim is to prove that  
		\begin{align}\label{El1}
		\|(\G_{n}\G_{n}^{*})^{-\frac{1}{2}}{\y}\|_{\H}^{2} \leq C\|\A^{2\alpha}\y\|^{2}_{\H},\ \text{ for all }\ \y \in \H_{n}
		\end{align}
		and 
		\begin{align}\label{El2}
		\mathbb{E}\left[\int_{0}^{t}\|\A^{2\alpha}\D\u_{n}^{R}(s,\x)\cdot \h\|^{2}_{\H}\d s\right] \leq C\|\A^{\alpha}\h\|_{\H}^{2},
		\end{align}
		where the constant $C=C(R,t)$ is independent of $n\in \mathbb{N}$.
		Therefore using \eqref{El1} and \eqref{El2} in \eqref{314}, one obtains 
		\begin{align*}
		\sup_{\k,\h\in \H_{n},\|\A^{\alpha}\h\|_{\H}<1}|\D\P_{t,n}^{R}\varphi(\k)\cdot \h| \leq \frac{C(R,t)}{t}\|\varphi\|_{\sup},
		\end{align*}
		and hence from \eqref{315}, one can easily deduce  \eqref{3p14}. 
		
		Now it is left to show \eqref{El1} and \eqref{El2}. 	A proof of \eqref{El1} is available in Lemma 4.2, \cite{BF} and thus we omit it here.  
		Let us now prove \eqref{El2}. 
		For $0<R<\infty$, we denote the directional derivative by $\U_{n}^{R}(t)$ at $\x$ in the direction $\h$ of the mapping $\x \mapsto \v_{n}^{R}(t,\x)$, that is,
		\begin{align*}
		\U_{n}^{R}(t) =[\D\v_{n}^{R}(t,\x)]  \cdot \h ,
		\end{align*}
		for given $\x, \h \in \H_{n}$. Note that it is also the derivative of the mapping  $ \x \mapsto \u_{n}^{R}(t,\x) =\v_{n}^{R}(t,\x)+\Pi_{n}\z(t)$. Therefore, $\U_{n}^{R}(\cdot)$ is the solution of the first variation equation associated with the system  \eqref{2.15} and is given by 
		\begin{align*}
	&	\frac{\d}{\d t}\U_{n}^{R}+\mu \A \U_{n}^{R}+2 \Theta'_{R}(\|\A^{\alpha}\u_{n}^{R}\|^{2}_{\H})\langle \A^{\alpha}\u_{n}^{R},\A^{\alpha}\U_{n}^{R} \rangle \B_{n}(\u_{n}^{R},\u_{n}^{R})\\&\quad +\Theta_{R}(\|\A^{\alpha}\u_{n}^{R}\|^{2}_{\H})\big[\B_{n}(\u_{n}^{R},\U_{n}^{R})+\B_{n}(\U_{n}^{R},\u_{n}^{R})\big]+2\beta \Theta'_{R}(\|\A^{\alpha}\u_{n}^{R}\|^{2}_{\H})\langle \A^{\alpha}\u_{n}^{R},\A^{\alpha}\U_{n}^{R} \rangle  \mathcal{C}_{n}(\u_{n}^{R})\\&\quad+\beta \Theta_{R}(\|\A^{\alpha}\u_{n}^{R}\|^{2}_{\H})\mathcal{C}'_n(\u_n^R)\U_n^R=\mathbf{ 0},
		\end{align*}
		with the initial condition $\U_{n}^{R}(0)= \Pi_{n}\h$, where $ \h \in \D(\A^{\alpha})$ and $\u_{n}^{R}(\cdot)$ is the solution of \eqref{2.15}, which can be characterized  as in Proposition \ref{prop3.6}. Taking the inner product with $\A^{4\alpha-1}\U_{n}^{R}$, we obtain 
		\begin{align}\label{Lemma2.11} \nonumber
		&\frac{1}{2}\frac{\d }{\d t }\|\A^{2\alpha-\frac{1}{2}}\U_{n}^{R}\|_{\H}^2 +\mu\|\A^{2\alpha}\U_{n}^{R}\|_{\H}^2\nonumber\\& =- 2\Theta'_{R}(\|\A^{\alpha}\u_{n}^{R}\|^{2}_{\H})\langle \A^{\alpha}\u_{n}^{R},\A^{\alpha}\U_{n}^{R} \rangle  \langle \B_{n}(\u_{n}^{R},\u_{n}^{R}),\A^{4\alpha-1}\U_{n}^{R}\rangle \nonumber\\&\quad-\Theta_{R}(\|\A^{\alpha}\U_{n}^{R}\|^{2}_{\H})  \B_{n}(\u_{n}^{R},\U_{n}^{R})+\B_{n}(\U_{n}^{R},\u_{n}^{R}), \A^{4\alpha-1}\U_{n}^{R}\rangle \nonumber\\&\quad-2 \beta \Theta'_{R}(\|\A^{\alpha}\u_{n}^{R}\|^{2}_{\H})\langle \A^{\alpha}\u_{n}^{R},\A^{\alpha}\U_{n}^{R} \rangle \langle \mathcal{C}_{n}(\u_{n}^{R}), \A^{4\alpha-1}\U_{n}^{R}\rangle\nonumber\\&\quad-\beta\Theta_{R}(\|\A^{\alpha}\u_{n}^{R}\|^{2}_{\H})\langle\mathcal{C}'_n(\u_n^R)\U_n^R, \A^{4\alpha-1}\U_{n}^{R}\rangle.
		\end{align}
		Using the Cauchy-Schwarz inequality, Lemma \ref{lem2.2} with  $\theta=\rho=\alpha$ and $\delta= 1-2\alpha$, the interpolation inequality \eqref{Interpolation} with $ \mu' = \alpha, \gamma'=2\alpha-\frac{1}{2}, \beta' = 2\alpha$ (so that $\theta'=2\alpha$) and Young's inequality with $p=\frac{1}{\alpha}, q = \frac{1}{1-\alpha}$, we estimate the first term from the right hand side of the inequality \eqref{Lemma2.11} as
		\begin{align*}
		&	2|\Theta'_{R}(\|\A^{\alpha}\u_{n}^{R}\|^{2}_{\H})\langle \A^{\alpha}\u_{n}^{R},\A^{\alpha}\U_{n}^{R} \rangle\langle \B_{n}(\u_{n}^{R},\u_{n}^{R}),\A^{4\alpha-1}\U_{n}^{R}\rangle|\nonumber\\& \leq  |\Theta'_{R}(\|\A^{\alpha}\u_{n}^{R}\|^{2}_{\H})| \|\A^{\alpha}\u_{n}^{R}\|_{\H}\|\A^{\alpha}\U_{n}^{R}\|_{\H}|\langle \A^{2\alpha-1}\B_{n}(\u_{n}^{R},\u_{n}^{R}),\A^{2\alpha}\U_{n}^{R}\rangle| \\&\leq |\Theta'_{R}(\|\A^{\alpha}\u_{n}^{R}\|^{2}_{\H})| \|\A^{\alpha}\u_{n}^{R}\|^{3}_{\H}\|\A^{\alpha}\U_{n}^{R}\|_{\H}\|\A^{2\alpha}\U_{n}^{R}\|_{\H}\\& \leq C(R)\|\A^{2\alpha-\frac{1}{2}}\U_{n}^{R}\|^{2\alpha}_{\H}\|\A^{2\alpha}\U_{n}^{R}\|^{2(1-\alpha)}_{\H}  \\& \leq \frac{\mu}{8} \|\A^{2\alpha}\U_{n}^{R}\|^{2}_{\H}+C(R,\mu)\|\A^{2\alpha-\frac{1}{2}}\U_{n}^{R}\|^{2}_{\H}.
		\end{align*} 
		Similarly, one can estimate $|\Theta_{R}(\|\A^{\alpha}\u_{n}^{R}\|^{2}_{\H})\langle \B_{n}(\u_{n}^{R},\U_{n}^{R})+\B_{n}(\U_{n}^{R},\u_{n}^{R}), \A^{4\alpha-1}\U_{n}^{R}\rangle|$ as
		\begin{align*}
		&	|\Theta_{R}(\|\A^{\alpha}\u_{n}^{R}\|^{2}_{\H})\langle \B_{n}(\u_{n}^{R},\U_{n}^{R})+\B_{n}(\U_{n}^{R},\u_{n}^{R}), \A^{4\alpha-1}\U_{n}^{R}\rangle| \\& \leq 	|\Theta_{R}(\|\A^{\alpha}\u_{n}^{R}\|^{2}_{\H})| \|\A^{\alpha}\u_{n}^{R}\|_{\H}
		\|\A^{\alpha}\U_{n}^{R}\|_{\H}\|\A^{2\alpha}\U_{n}^{R}\|_{\H} \\&\leq \frac{\mu}{8}\|\A^{2\alpha}\U_{n}^{R}\|^{2}_{\H}+
		C(R,\mu)\|\A^{2\alpha-\frac{1}{2}}\U_{n}^{R}\|^{2}_{\H}.
		\end{align*}  
		For the term $|2 \beta\Theta'_{R}(\|\A^{\alpha}\u_{n}^{R}\|^{2}_{\H})\langle \A^{\alpha}\u_{n}^{R},\A^{\alpha}\U_{n}^{R} \rangle \langle\mathcal{C}_{n}(\u_{n}^{R}), \A^{4\alpha-1}\U_{n}^{R}\rangle|$, we use  the Cauchy-Schwarz inequality, Sobolev's inequality \eqref{26} with $p=2r$ (so that $ \frac{r-1}{2r}\leq \alpha$), the interpolation inequality \eqref{Interpolation} with $ \mu' = \alpha, \gamma'=2\alpha-\frac{1}{2}, \beta' = 2\alpha$ (so that $\theta'=2\alpha$),  also with $\mu' =  4\alpha-1, \gamma' = 2\alpha-\frac{1}{2}, \beta' = 2\alpha$ (so that $\theta'=2(1-2\alpha)$) and Young's inequality to estimate it as 
		\begin{align*}
	&	|2 \beta \Theta'_{R}(\|\A^{\alpha}\u_{n}^{R}\|^{2}_{\H})\langle \A^{\alpha}\u_{n}^{R},\A^{\alpha}\U_{n}^{R} \rangle  \langle\mathcal{C}_{n}(\u_{n}^{R}), \A^{4\alpha-1}\U_{n}^{R}\rangle|\nonumber\\&\leq 2\beta | \Theta'_{R}(\|\A^{\alpha}\u_{n}^{R}\|^{2}_{\H})|\|\A^{\alpha}\u_{n}^{R}\|_{\H}\|\A^{\alpha}\U_{n}^{R}\|_{\H}\|\u_{n}^{R}\|_{\wi\L^{2r}}^r\|\A^{4\alpha-1}\U_{n}^{R}\|_{\H}\nonumber\\&\leq 2\beta | \Theta'_{R}(\|\A^{\alpha}\u_{n}^{R}\|^{2}_{\H})\|\A^{\alpha}\u_{n}^{R}\|_{\H}^{r+1}\|\A^{2\alpha-\frac{1}{2}}\U_{n}^{R}\|_{\H}^{2\alpha}\|\A^{2+2\alpha}\U_{n}^{R}\|_{\H}^{2\alpha}\nonumber\\&\leq\frac{\mu}{8}\|\A^{2\alpha}\U_{n}^{R}\|^{2}_{\H}+
	C(R,\mu,\beta)\|\A^{2\alpha-\frac{1}{2}}\U_{n}^{R}\|^{2}_{\H}.
		\end{align*}
		For the final term in the right hand side of the equality \eqref{Lemma2.11}, we use the Cauchy-Schwarz inequality, interpolation inequality \eqref{Interpolation} with $\mu' =  4\alpha-1, \gamma' = 2\alpha-\frac{1}{2}, \beta' = 2\alpha$ (so that $\theta'=2(1-2\alpha)$), Gagliardo-Nirenberg's inequality \eqref{2.1} with $p=2r,j=0,m=4\alpha,l=q=2$ (so that $\theta = \frac{r-1}{4r\alpha}$), Young's inequality with $p=\frac{8r\alpha}{r-1-(1-4\alpha)4r\alpha}, q= \frac{8r\alpha}{1-r+4r\alpha(3-4\alpha)}$ and Sobolev's embedding \eqref{26} with $p=2r$ (so that $ \frac{r-1}{2r}\leq \alpha$) to estimate it as 
		\begin{align*}
	&	|\beta\Theta_{R}(\|\A^{\alpha}\u_{n}^{R}\|^{2}_{\H})\langle\mathcal{C}'_n(\u_n^R)\U_n^R, \A^{4\alpha-1}\U_{n}^{R}\rangle| \nonumber\\& \leq r\beta |\Theta_{R}(\|\A^{\alpha}\u_{n}^{R}\|^{2}_{\H})| \|\u_{n}^{R}\|^{r-1}_{\wi\L^{2r}}\|\U_{n}^{R}\|_{\wi\L^{2r}}\|\A^{4\alpha-1}\U_{n}^{R}\|_{\H} \\& \leq  r\beta |\Theta_{R}(\|\A^{\alpha}\u_{n}^{R}\|^{2}_{\H})|\|\A^{\alpha}\u_{n}^{R}\|^{r-1}_{\H}\|\U_{n}^{R}\|_{\wi\L^{2r}}
		\|\A^{2\alpha-\frac{1}{2}}\U_{n}^{R}\|^{2(1-2\alpha)}_{\H}\|\A^{2\alpha}\U_{n}^{R}\|^{4\alpha-1}_{\H} \\& \leq C(R,\beta)\|\U_{n}^{R}\|^{\frac{1-(1-4\alpha)}{4r\alpha}}_{\H}\|\A^{2\alpha}\U_{n}^{R}\|^{\frac{r-1+(4\alpha-1)4r\alpha}{4r\alpha}}_{\H}\|\A^{2\alpha-\frac{1}{2}}\U_{n}^{R}\|^{2(1-2\alpha)}_{\H} \\& \leq C(R,\beta)\|\A^{2\alpha}\U_{n}^{R}\|^{\frac{r-1+(4\alpha-1)4r\alpha}{4r\alpha}}_{\H}\|\A^{2\alpha-\frac{1}{2}}\U_{n}^{R}\|^{\frac{1-(1-4\alpha)r+8r\alpha(1-2\alpha)}{4r\alpha}}_{\H},
		\\&  \leq \frac{\mu}{8} \|\A^{2\alpha}\U_{n}^{R}\|^{2}_{\H}+C(R,\mu,\beta)\|\A^{2\alpha-\frac{1}{2}}\U_{n}^{R}\|^{2}_{\H}.
		\end{align*} 
		Combining the above estimates and substituting it in \eqref{Lemma2.11}, we deduce that
		\begin{align}\label{FV4} 
		&	\frac{\d }{\d t }\|\A^{2\alpha-\frac{1}{2}}\U_{n}^{R}(t)\|^{2}_{\H} +\mu\|\A^{2\alpha}\U_{n}^{R}(t)\|_{\H}^2\leq C(R,\mu,\beta)\|\A^{2\alpha-\frac{1}{2}}\U_{n}^{R}(t)\|^{2}_{\H},
		\end{align}
		for a.e. $t\in[0,T]$. 
		Applying Gronwall's inequality, we obtain 
		\begin{align}\label{321}
		&\|\A^{2\alpha-\frac{1}{2}}\U_{n}^{R}(t)\|^{2}_{\H} \leq \|\A^{2\alpha-\frac{1}{2}}\h\|^{2}_{\H}e^{C(R,\mu,\beta)t},
		\end{align} for all $t\in[0,T]$.  
		Substituting \eqref{321} in \eqref{FV4}	 and then integrating from $0$ to $t$, we further have 
		\begin{align*}
		\int_{0}^{t} \|\A^{2\alpha}\U_{n}^{R}(s)\|^{2}_{\H}\d s \leq C(R,t)\|\A^{2\alpha-\frac{1}{2}}\h\|^{2}_{\H},
		\end{align*} where $C(R,\mu,\beta,t) $ is a constant, which is independent of $n,\z,\x,\h,\f$. Finally by taking the expectation of the inequality, one can conclude \eqref{El2} the proof of \eqref{El2}. 
	\end{proof}

	Finally, taking the limit as $n\to\infty$, we get the following result for the Markov semigroup $\P_{t}^{R}$. A proof of the following lemma is same as that of Lemma 3.3, \cite{BF} and hence we omit it here. 
	\begin{lemma}\label{lem3.8}
		Under the assumptions of Theorem \ref{main} and for every $t>0$ and $R>0$, there exists a constant $L=L(R,t)$, such that \begin{align*}
		|\P_{t}^{R}\varphi(\x)-\P_{t}^{R}\varphi(\y)| \leq L\|\x-\y\|_{\D(\A^{\alpha})},
		\end{align*} 
		for all $\x,\y \in \D(\A^{\alpha}), \ \varphi \in \C_{b}(\D(\A^{\alpha}))$ with $\|\varphi\|_{\sup}\leq 1$. 
		Moreover, $\P_{t}^{R}\varphi$ is Lipschitz for arbitrary $\varphi \in \mathcal{B}_{b}(\D(\A^{\alpha}))$.
	\end{lemma}
	Let us next pass the limit as $R\to\infty$. 
	\begin{lemma}[Lemma 4.4, \cite{BF}]\label{lem3.10}
		Under the assumption of Theorem \ref{main}, for all $t>0$, we have  \begin{align*}
		\lim_{R\to  \infty}\|\P(t,\x,\cdot)-\P^{R}(t,\x,\cdot)\|_{\mathrm{var}}=0,
		\end{align*}
		uniformly with respect to $\x$ in bounded sets of $\D(\A^{\alpha})$.
	\end{lemma}
	\begin{proof}[Proof of Proposition \ref{prop3.7}]
		By Lemma \ref{lem3.10}, for  given $t>0$ and arbitrary $\epsilon>0$, there exists $R_{\epsilon}>0$ such that 
		\begin{align*}
		\|\P(t,\x,\cdot)-\P^{R}(t,\x,\cdot)\|_{\mathrm{var}}+\|\P(t,\y,\cdot)-\P^{R}(t,\y,\cdot)\|_{\mathrm{var}} <2\epsilon,
		\end{align*}
		for $\x,\y$ given in a ball of $\D(\A^{\alpha})$. For such $R_{\epsilon}$, taking $\y$ close enough to $\x$ in $\D(\A^{\alpha})$, by Lemma \ref{lem3.8}, we get
		\begin{align*}
		\|\P(t,\x,\cdot)-\P(t,\y,\cdot)\|_{\mathrm{var}} &\leq \|\P(t,\x,\cdot)-\P^{R}(t,\x,\cdot)\|_{\mathrm{var}}+\|\P^{R}(t,\x,\cdot)-\P^{R}(t,\y,\cdot)\|_{\mathrm{var}}\\& \quad \quad+\|\P^{R}(t,\y,\cdot)-\P(t,\y,\cdot)\|_{\mathrm{var}} \\&\leq 3\epsilon,
		\end{align*}
		which completes the proof. 
	\end{proof} 
	The Proof of Theorem \ref{main} can be completed by an application of Theorem \ref{thm3.1}. 
	
	\begin{remark}
		From Remark \ref{rem3.7}, it is clear that the uniqueness of invariant measure for the 2D SCBF equations \eqref{2.7} can be obtained for the more general assumption on $\G$ given in \eqref{1p6} for $1<\ell<\frac{3}{2}$ and \eqref{1p7} for $\ell=1$. In the case of periodic domains, the technical difficulty $1\leq\ell<\frac{3}{2}$  on $\ell$ also can be removed. This will be discussed in a separate work.  
	\end{remark}

	\section{Large Deviation Principle}\label{sec4}\setcounter{equation}{0}
	In this section, we prove the LDP with respect to the topology $\tau$ and Donsker-Varadhan LDP of the occupation measure for the 2D SCBF equations \eqref{2.7} with $\x,\f\in\H$. 	In the sequel, $\mathbb{P}_{\x}$ denotes the law on $\C(\R^{+};\H)$ of the Markov process with $\x \in \H$ as initial state.  For any initial measure $\nu $ on $\H$, let us define $\mathbb{P}_{\nu}(\cdot)= \int_{\H} \mathbb{P}_{\x}\nu(\d \x)$. 
	Our aim is to establish the LDP for the occupation measure $L_{t}$ of the solution $\u(\cdot)$ to the system \eqref{2.7}, which is given by 
	$$	L_{t}(A) := \frac{1}{t} \int_{0}^{t} \delta_{\u (s)}(A) \d s , \ \text{for all }\ A \in \mathcal{B}(\H),$$
	where $\delta_{a}$ is Dirac measure at $a$, and $\mathcal{B}(\H)$ is the Borelian $\sigma$-field in $\H$. We denote $\{\u_t^{\x}\}_{t\geq 0}$ for the solution of the 2D SCBF equations \eqref{2.7} with the initial condition $\x$, defined on $(\Omega,\mathscr{F},\{\mathscr{F}_t\}_{t\geq 0},\mathbb{P})$. 	Let us first state our main result of this section. 
	\begin{theorem}\label{main-t}
		Let $\f \in \H$ and let $ \frac{1}{4} <\alpha < \frac{1}{2} $ for $r=2,3$ and $\frac{1}{3} <\alpha < \frac{1}{2}$ for $r=3$ be a fixed number such that \eqref{1p5} holds. Let $  0<\lambda_{0} < \frac{ \mu \lambda_{1}}{2\|\Q\|_{\mathcal{L}(\H)}} ,$ where $\|\Q\|_{\mathcal{L}(\H)} $ is the norm of $\Q :=\G\G^{*}$ as an operator in $\H$ and 
		\begin{align}\label{4.1} \mathcal{M}_{ \lambda_{0},R}:= \left\{\nu \in \M_{1}(\H) : \int_{\H} e^ {\lambda_{0}\|\x\|^{2}_{\H}} \nu(\d\x) \leq R \right\} . 
		\end{align}  
		The family $ \mathbb{P}_{\nu}(L_{T} \in \cdot) \text{ as } T \rightarrow +\infty  $ satisfies the Donsker-Varadhan LDP with respect to the topology $\tau$, with speed $T$ and rate funtion $\J$ uniformly for any initial measure $\nu$ in $\mathcal{M}_{ \lambda_{0},R}, \text{ where } R >1$ is any fixed number. Here the rate function $ \J:\M_{1}(\H) \rightarrow [0,+\infty] $ is the level-2 entropy of Donsker-Varadhan. Moreover, we have 
		\begin{itemize}
			\item[(i)] $\J$ is a good rate function on $\M_{1}(\H)$.
			\item[(ii)] For all open sets $\G \text{ in } \M_{1}(\H)$ with respect to the topology $\tau$,
			$$  \liminf_{T\to\infty} \frac{1}{T} \log \inf_{\nu \in \mathcal{M}_{ \lambda_{0},R}}\mathbb{P}_{\nu}(L_{T}\in \G) \geq -\inf_{\G} \J.$$ 
			\item[(iii)] For all closed sets $\F \text{ in } \M_{1}(\H)$ with respect to the topology $\tau$,
			$$  \limsup_{T\to\infty} \frac{1}{T} \log \sup_{\nu \in \mathcal{M}_{ \lambda_{0},R}}\mathbb{P}_{\nu}(L_{T}\in \F) \leq -\inf_{\F} \J.$$ 
		\end{itemize}
		Furthermore, we have for the invariant measure $\nu$, and for all $ \varrho \in \M_{1}(\H)$,
		\begin{align}\label{Claim}
		\J(\nu) < +\infty \Rightarrow \nu\ll \varrho, \quad \nu(\V) =1\  \text{ and }\  \int_{\V} \|\A^{\frac{1}{2}}\x\|_{\H}^{2}\d\nu < +\infty. \end{align} 
	\end{theorem}
	The class \eqref{4.1} of initial distributions for the uniform LDP is sufficiently rich. For example, choosing $R$ large enough, it includes all the Dirac probability measure $\delta_{\x}$ with $\x$ in any ball of $\H$.  The LDP with respect to the topology $\tau$ is stronger than that with respect to the usual weak convergence topology as in Donsker-Varadhan \cite{MDD}. 
	\begin{corollary}\label{cor4.2}
		Let $(\mathbb{B}, \|\cdot\|_{\mathbb{B}})$ be a separable Banach space, and $\psi: \D(\A^{\frac{1}{2}}) \to \mathbb{B}$ be a measurable function, bounded on balls $\{ \x:\|\A^{\frac{1}{2}}\x\|_{\H} \leq R\}$ and satisfying 
		\begin{align}\label{4.2} \lim_{\|\A^{\frac{1}{2}}\x\|_{\H} \to \infty}  \frac{\|\psi(\x)\|_{\mathbb{B}}}{\|\A^{\frac{1}{2}}\x\|^{2}_{\H}} = 0.
		\end{align}
		Then $\mathbb{P}_{\nu}(L_{T}(\psi) \in \cdot)$ satisfies the Donsker-Varadhan LDP on $\mathbb{B}$ with speed $T$ and the rate function $\I_{\psi}$ given by $$ \I_{\psi}(\y) = \inf \{\J(\nu): \J(\nu) < +\infty,\ \nu(\psi) = \y\}, \ \text{ for all }\ \y \in \mathbb{B},$$
		uniformly over initial distributions $\nu \in \mathcal{M}_{\lambda_{0}, R}$ (for any $R >1$).
	\end{corollary}
	One can prove Corollary \ref{cor4.2} in a similar way as in the proof of Corollary 1.2, \cite{MG1}.	As a particular case of Corollary \ref{cor4.2}, by choosing $\psi(\x)=\x$ on $\mathbb{B}=\D(\A^{\frac{1}{2}})$, we have
	\begin{proposition}
		As $ T \to \infty$, the family $$ \mathbb{P}_{\nu}\left(\frac{1}{T}\int_{0}^{T}\u_{t}\d t \in \cdot \right)$$
		satisfies an LDP on $\D(\A^{\frac{1}{2}}) \text{ with speed } T \text{ rate function } \I$ defined by  $$ 
		\I(\z) = \inf \{\J(\nu):\J(\nu) < +\infty,\ \nu(\x) = \z\}, \ \text{for all }\ \z \in \D(\A^{\frac{1}{2}}),$$ 
		uniformly over initial distributions $\nu \in \mathcal{M}_{\lambda_{0},R}$ (for any $R>1$).		
	\end{proposition}
	
	\subsection{Notations and entropy of Donsker-Varadhan} Let us now recall general results on the LDP for strong Feller, topologically irreducible Markov processes.	We consider an $\mathrm{E}$-valued continuous Markov process, $(\Omega,\{\mathscr{F}_{t}\}_{t\geq 0},\mathscr{F},\{\u_{t}\}_{t\geq 0}, \{\mathbb{P}_{\x}\}_{\x \in \mathrm{E}})$,
	whose semigroup of Markov transition kernel is denoted by $\{\P_{t}(\x,\d \y)\}_{t\geq 0}$, where
	\begin{itemize}
		\item  $\Omega = \C(\R^{+};\mathrm{E})$ is the space of continuous functions from $\R^{+}$ to $\mathrm{E}$ equipped with the compact covergence topology,
		\item  the natural filtration is $\mathscr{F}_{t} = \sigma(\u_{s}; 0 \leq s \leq t)$ for any $t\geq 0$ and $\mathscr{F} = \sigma(\u_{s}; 0 \leq s)$.
	\end{itemize}
	As usual, we denote the law of Markov process with the initial state $\x \in \mathrm{E} $ by $ \mathbb{P}_{\x}$, and for any initial measure $\nu $ on $\mathrm{E}$, we define $\mathbb{P}_{\nu}(\cdot)=\int_{\mathrm{E}}\mathbb{P}_{\x}(\cdot)\nu (\d \x)$. 
	The empirical measure of level-3 is given by $$ R_{t} := \frac{1}{t}\int_{0}^{t} \delta_{\theta_{s}\u} \d s, $$ where $(\theta_{s}\u)_{t}= \u_{s+t},$ for all $t,s \geq 0 $ are the shifts on $\Omega$. Therefore, $R_{t}$ is a random element of $\M_{1}(\Omega)$, the space of probability measures on $\Omega$.	The level-3 entropy functional of Donsker-Varadhan $\mathrm{H}:\M_{1}(\Omega) \to [0,+\infty]$ is defined by 
	\begin{equation}\label{HQ}
	\begin{aligned}
	\mathrm{H}(Q):= \left\{\begin{array}{cl}
	\mathbb{E}^{\bar{Q}}h_{\mathscr{F}^{0}_{1}} (\bar{Q}_{\omega(-\infty,0]};\mathbb{P}_{\omega(0)}), &  \text{ if } Q \in \M^{s}_{1}(\Omega), \\ 
	+\infty , & \text{ otherwise},
	\end{array}\right. \end{aligned}
	\end{equation}
	where \begin{itemize}
		\item $\M^{s}_{1}(\Omega)$ is the subspace of $\M_{1}(\Omega)$, whose element are moreover stationary; 
		\item $\bar{Q}$ is the unique stationary extension of $Q \in \M^{s}_{t}(\Omega)$ to $\bar{\Omega} := \C(\R;\mathrm{E})$;  $\mathscr{F}^{s}_{t} = \sigma \{\u(r); s \leq r \leq t \}, \ \text{for all }\ s,  t \in \R, \ s \leq t$;
		\item $\bar{Q}_{\omega(-\infty,t]}$ is the regular conditional distribution of $\bar{Q}$ knowing $\mathscr{F}^{-\infty}_{t}$;
		\item   $h_{\mathcal{G}}(\nu,\varrho)$ is the usual relative entropy or Kullback information of $\nu $ with respect to $\varrho$ restricted to the $\sigma $ -field $\mathcal{G} $, is given by 
		\begin{equation*}
		\begin{aligned}h_{\mathcal{G}}(\nu;\varrho):=\left\{ \begin{array}{cl}		\int \frac{\d\nu}{\d\varrho}\big|_{\mathcal{G}} \log \big(\frac{\d\nu}{\d\varrho}\big|_{\mathcal{G}}\big)\d\varrho, & \text{ if } \nu \ll \varrho \text{ on } \mathcal{G}, \\
		+\infty, & \text{ otherwise}.
		\end{array}\right. \end{aligned}
		\end{equation*}
	\end{itemize}
	The level-2 entropy functional $\J: \M_{1}(\mathrm{E}) \rightarrow [0,\infty),$ which governs the LDP in our main Theorem \ref{main-t} is 
	\begin{align}\label{Def.J}
	\J(\varrho) = \inf \{\mathrm{H}(Q): Q \in \M^{s}_{1}(\Omega) \text{ and }  Q_{0}= \varrho \}, \quad \varrho \in \M_{1}(\mathrm{E}),\end{align}
	where $ Q_{0}(\cdot) =Q(\u(0) \in \cdot)$ is the marginal law at $t=0$. A proof of the following result is available in Proposition 3.1, \cite{MG1}, and hence we omit it here. 
	\begin{proposition}\label{prop4.4}
		For our system, $\J(\nu) < +\infty\Rightarrow  \nu \ll \varrho$. Moreover, a necessary and sufficient condition for $\J(\nu) = 0$ is $\nu = \varrho$.
	\end{proposition}
	\subsection{The hyper-exponential recurrence criterion}
	The first step in the proof of our main Theorem \ref{main-t} is the proof of LDP for the initial measures in $\M_{1}(\D (\A^{\alpha}))$, since $\u_{t}$ is strongly Feller and topologically irreducible in $\D(\A^{\alpha})$. In that case, we have the  criterion of  hyper-exponential recurrence established in  Theorem 2.1, \cite{LW1} for a general polish space $\mathrm{E}$. The following result is a slight extension of the result in \cite{LW1} to a uniform LDP over a non-empty family of initial measures (cf. \cite{MG1}).
	
	\begin{theorem}[Theorem 2.1, \cite{LW1}, Theorem 3.2, \cite{MG1}]\label{thm4.5}
		Let $\mathcal{A} \subset \M_{1}(\mathrm{E})$ and assume that 
		\begin{align}\label{4.4} \text{$\P_{t}$ is strong Feller and topologically irreducible on $\mathrm{E}$.} \end{align}
		If for all $\lambda >0$, there exists some compact $K \Subset \mathrm{E}$, such that   \begin{align}\label{4.5} \sup_{\nu \in \mathcal{A}}\mathbb{E}^{\nu}[e^{\lambda \tau_{K}}] <\infty, \end{align}
		and 
		\begin{align}\label{4.6} \sup_{\x \in K}\mathbb{E}^{\x} [e^{\lambda \tau^{(1)}_{K}}] < \infty, \end{align}
		where $\tau_{K} = \inf \{t \geq 0; \u_{t} \in K\} \text{ and } \tau^{(1)}_{K}:= \inf \{t \geq 1; \u_{t} \in K\}$, then the family $\mathbb{P}_{\nu}(L_{t}\in \cdot)$ satisfies the LDP on $\M_{1}(\mathrm{E})$ with respect to the $\tau$-topology with the rate function $\J$ defined as level-2 entropy functional, and uniformly for initial measures $\nu$ in the subset $\mathcal{A}$.
		More precisely, the following three properties hold:
		\begin{itemize}
			\item[(a1)] $\J:\M_{1}(\mathrm{E}) \to [0,+\infty]$ is inf-compact with respect to the $\tau$-topology,
			\item[(a2)] (the lower bound) for any $\tau$-open $\G$ in $\M_{1}(\mathrm{E})$,
			\begin{align}\label{4.7}
			\liminf_{T\to\infty} \frac{1}{T}\log \inf_{\nu \in \mathcal{A}} \mathbb{P}_{\nu}(L_{T}\in \G) \geq -\inf_{\G} \J, \end{align}
			\item[(a3)] (the upper bound) for any $\tau$-closed $\F$ in $\M_{1}(\mathrm{E})$,
			\begin{align}\label{4.8}
			\limsup_{T\to\infty} \frac{1}{T}\log \sup_{\nu \in \mathcal{A}} \mathbb{P}_{\nu}(L_{T}\in \F) \leq -\inf_{\F} \J. \end{align}
		\end{itemize}
	\end{theorem}
	\subsection{Exponential estimates for the solution}
	In this subsection, we prove a crucial exponential estimate for the solution $\u(\cdot)$ to the 2D SCBF equations \eqref{2.7}, which is helpful in obtaining the LDP results.
	\begin{proposition}\label{prop4.6}
		For any fixed $0<\lambda_{0} < \frac{\mu \lambda_{1}}{2\|\Q\|_{\mathcal{L}(\H)}} $, where $\|\Q\|_{\mathcal{L}(\H)}$ is the norm of $\Q$ as an operator in $\H$ and any $ \x \in \H$, the process $\u(\cdot)$ satisfies the following estimate: 
		\begin{align*} &\mathbb{E}^{\x}\left\{ \exp \left(\lambda_{0}\|\u(t)\|^{2}_{\H}+ \mu \lambda_{0} \int_{0}^{t} \|\A^{\frac{1}{2}}\u(s)\|^{2}_{\H}\d s+\beta \lambda_{0}\int_{0}^{t}\|\u(s)\|^{r+1}_{\L^{r+1}}\d s \right)\right\} \\& \leq e^{t\lambda_{0}\big(\Tr(\Q)+ \frac{\|\f\|_{\H}^{2}}{\mu \lambda_{1}-2\|\Q\|_{\mathcal{L}(\H)}\lambda_{0}}\big)}e^{\lambda_{0}\|\x\|^{2}_{\H}}.
		\end{align*}
		In particular, the following estimates hold:
		\begin{align}\label{5.1} \mathbb{E}^{\x}\left\{ \exp \big(\lambda_{0}\|\u(t)\|^{2}_{\H}\big)\right\} \leq e^{t\lambda_{0}\big(\Tr(\Q)+ \frac{\|\f\|_{\H}^{2}}{\mu \lambda_{1}-2\|\Q\|_{\mathcal{L}(\H)}\lambda_{0}}\big)}e^{\lambda_{0}\|\x\|^{2}_{\H}},
		\end{align}
		and 
		\begin{align}\label{5.2} \mathbb{E}^{\x} \left\{\exp\bigg(\mu \lambda_{0}\int_{0}^{t} \|\A^{\frac{1}{2}}\u(s)\|^{2}_{\H}\d s \bigg)\right\}  \leq e^{t\lambda_{0}\big(\Tr(\Q)+ \frac{\|\f\|_{\H}^{2}}{ \mu \lambda_{1}-2\|\Q\|_{\mathcal{L}(\H)}\lambda_{0}}\big)}e^{\lambda_{0}\|\x\|^{2}_{\H}},
		\end{align} and
		\begin{align}\label{ESC}
		\mathbb{E}^{\x} \left\{\exp\bigg(\beta \lambda_{0}\int_{0}^{t}\|\u(s) \|^{r+1}_{\L^{r+1}} \d s\bigg)\right\}\leq e^{t\lambda_{0}\big(\Tr(\Q)+ \frac{\|\f\|_{\H}^{2}}{\mu \lambda_{1}-2\|\Q\|_{\mathcal{L}(\H)}\lambda_{0}}\big)}e^{\lambda_{0}\|\x\|^{2}_{\H}}. 
		\end{align}
		Furthermore, for any fixed $R>1$, we have 
		\begin{align}\label{5.3} \sup_{\nu \in \mathcal{M}_{\lambda_{0},R}}\mathbb{E}^{\nu} \left\{ \exp\bigg(\mu \lambda_{0}\int_{0}^{t} \|\A^{\frac{1}{2}}\u(s)\|^{2}_{\H}\d s \bigg)\right\} \leq e^{t\lambda_{0}\big(\Tr(\Q)+ \frac{\|\f\|_{\H}^{2}}{\mu \lambda_{1}-2\|\Q\|_{\mathcal{L}(\H)}\lambda_{0}}\big)}R,
		\end{align}
		where $\mathcal{M}_{\lambda_{0},R}$ is the set of initial measures defined by \eqref{4.1}.
	\end{proposition}
	Before going to the proof of Proposition \ref{prop4.6}, let us discuss some consequences of the exponential estimates. 
	\begin{corollary}[Corollary 4.2, \cite{MG1}]
		Under the estimate \eqref{5.2}, the family of laws $\mathbb{P}_{\nu}(L_{t}\in \cdot)$ is uniformly exponential tight over $\mathcal{M}_{\lambda_{0},R}$. More precisely, for any $\epsilon>0$, there is some compact subset $\K=\K_{\epsilon}$ in $\M_{1}(\H)$ in the weak convergence topology such that 
		\begin{align*} \limsup_{t \to \infty } \frac{1}{t} \log \sup_{\nu \in \mathcal{M}_{\lambda_{0},R}} \mathbb{P}_{\nu}(L_{t} \notin \K) \leq - \frac{1}{\epsilon}.
		\end{align*}
		Consequently for any closed set $\F \text{ in } \M_{1}(\H)$ equipped with the weak convergence topology $\sigma(\M_{1}(\H), \C_{b}(\H))$, we have 
		\begin{align}\label{5.4} \limsup_{t \to \infty } \frac{1}{t} \log \sup_{\nu \in \mathcal{M}_{\lambda_{0},R}} \mathbb{P}_{\nu}(L_{t} \in \F) \leq - \inf_{\F}\J,
		\end{align}
		where the entropy of Donsker-Varadhan $\J:\nu \in \M_{1}(\H) \to \J(\nu) \in [0,+\infty] $ satisfies 
		\begin{align}\label{5.5} \mu\lambda_{0} \int_{\H} \|\A^{\frac{1}{2}}\x\|^{2}_{\H}\d \nu \leq \J(\nu) + \lambda_{0} \bigg(\Tr(\Q)+ \frac{\|\f\|_{\H}^{2}}{\mu \lambda_{1}-2 \|\Q\|_{\mathcal{L}(\H)}\lambda_{0}}\bigg), \end{align}
		for any $0 <\lambda_{0}< \frac{\mu \lambda_{1}}{2\|\Q\|_{\mathcal{L}(\H)}}$.
	\end{corollary}
	The exponential estimates discussed above grants us an alternative approach to prove the existence of an invariant measure based on LDP. In particular, we have 
	\begin{corollary}[Corollary 4.3, \cite{MG1}]
		Assume that a Feller-Markov process $\u$ on $\H$ satisfies the exponential estimate \eqref{5.2}, then $\u$ admits at least one invariant measure.
	\end{corollary}
	We need the following result to discuss about the proof of Proposition \ref{prop4.6}. 	Let us first consider the following finite dimensional Galerkin approximations of the system \eqref{2.7}:
	\begin{equation}\label{5.8}
	\left\{
	\begin{aligned}
	\d\u_{n}(t) + [\mu \A\u_{n}(t) + \B_{n}(\u_{n}(t)) + \beta \mathcal{C}_{n}(\u_{n}(t)) ] \d t & = \f_{n} \d t + \G_{n}\d \W(t), \\
	\u_{n}(0)&=\x_{n}:= \Pi_{n}\x. 
	\end{aligned}
	\right.
	\end{equation} 
	Since \eqref{5.8} is a finite dimensional stochastic equation with locally Lipschitz coefficients (see \eqref{2p9} and \eqref{214}), and satisfying the a-priori energy estimate (cf. \cite{MTM1})
	\begin{align}\label{4.18}
	&	\E\left[\sup_{t\in[0,T]}\|\u_n(t)\|_{\H}^2+\mu\int_0^T\|\A^{\frac{1}{2}}\u_n(t)\|_{\H}^2\d t+\beta\int_0^T\|\u_n(t)\|_{\wi\L^{r+1}}^{r+1}\d t\right]\nonumber\\&\leq C\left(\|\x\|_{\H}^2+\Tr(\G\G^*)T+\|\f\|_{\H}^2T\right),
	\end{align}
	one can deduce that the system has a unique solution $\u_n\in\C([0,T];\H_n)$. 
	\begin{lemma}\label{lem4.9}
		Let $\u_n(\cdot)$ and $\u(\cdot)$ be solutions to the systems  \eqref{5.8} and  \eqref{2.7}, respectively. Then, we have 
		\begin{align}\label{419}
		&	\|\u_{n}(t)\|^{2}_{\H}+ 2\mu \int_{0}^{t} \|\A^{\frac{1}{2}}\u_{n}(s)\|^{2}_{\H}\d s+ 2\beta \int_{0}^{t} \|\u_{n}(s)\|^{r+1}_{\L^{r+1}}\d s\nonumber\\& \xrightarrow{\text{a.s.}}  \|\u(t)\|^{2}_{\H} + 2\mu \int_{0}^{t} \|\A^{\frac{1}{2}}\u(s)\|^{2}_{\H}\d s + 2\beta \int_{0}^{t} \|\u(s)\|^{r+1}_{\L^{r+1}}\d s\ \text{  as } \ n\to \infty,
		\end{align} 
		for all $t\in[0,T]$.
		
		Furthermore, we have the following exponential estimate: 
		\begin{align}\label{5.12} \nonumber
		&	\mathbb{E}^{\x}\left\{\exp\bigg(\lambda_{0}\|\u_{n}(s)\|^{2}_{\H}+ \mu\lambda_{0}\int_{0}^{t}\| \A^\frac{1}{2}\u_{n}(s)\|^{2}_{\H}\d s + \beta \lambda_{0}\int_{0}^{t}\|\u_{n}(s)\|^{r+1}_{\L^{r+1}}\d s\bigg)\right\} \\& \leq 
		e^{\lambda_{0}t\bigg(\Tr(\Q)+\frac{\|\f\|^{2}}{\mu \lambda_{1}-2\lambda_{0}\|\Q\|_{\mathcal{L}(\H)}}\bigg)}e^{\lambda_{0}\|\x\|^{2}_{\H}}.
		\end{align}
	\end{lemma} 
	\begin{proof} 
		Using the energy estimate and the Banach-Alaoglu theorem, one can easily extract a subsequence of $\{\u_n\}$ (still denoted by $\{\u_n\}$) such that 
		\begin{equation}\label{420}
		\left\{
		\begin{aligned}
		\u_n&\xrightarrow{w^*}\u\ \text{ in }\mathrm{L}^2(\Omega;\mathrm{L}^{\infty}(0,T;\H)),\\
		\u_n&\xrightarrow{w}\u\ \text{ in }\mathrm{L}^2(\Omega;\mathrm{L}^{2}(0,T;\V)),\\
		\u_n&\xrightarrow{w}\u\ \text{ in }\mathrm{L}^{r+1}(\Omega;\mathrm{L}^{r+1}(0,T;\wi\L^{r+1})),
		\end{aligned}\right.
		\end{equation}
		where $\u(\cdot)$ is the unique strong solution of the system \eqref{2.7} (by using the local monotonicity result \eqref{fe2} and a stochastic generalization of Minty-Browder technique, see \cite{MTM1} for more details). Note that the whole sequence also converges, since $\u_n(\cdot)$ and $\u(\cdot)$ are unique solutions to the systems  \eqref{5.8} and  \eqref{2.7}, respectively.	Let us apply finite dimensional  It\^o's formula to the process $\|\u_{n}(\cdot)\|_{\H}^2$ to obtain 
		\begin{align}\label{418}
		&	\|\u_{n}(t)\|^{2}_{\H}+ 2\mu \int_{0}^{t} \|\A^{\frac{1}{2}}\u_{n}(s)\|^{2}_{\H}\d s+ 2\beta \int_{0}^{t} \|\u_{n}(s)\|^{r+1}_{\L^{r+1}}\d s\nonumber\\&= \|\x_{n}\|^{2}_{\H} + 2\int_{0}^{t}(\f_{n}, \u_{n}(s) ) \d s +\Tr(\G_{n}\G^{*}_{n})t+ 2\int_{0}^{t}	( \G_{n} \d \W(s), \u_{n}(s)) ,
		\end{align}
		$\mathbb{P}$-a.s., where we used the fact that $
		\langle \B_{n}(\u_{n}),\u_{n}\rangle= 0$. 	Taking the expectation on both sides, we find
		\begin{align}\label{Lemma5.4a} \nonumber
		&	\mathbb{E}\left[\|\u_{n}(t)\|^{2}_{\H}+ 2\mu \int_{0}^{t} \|\A^{\frac{1}{2}}\u_{n}(s)\|^{2}_{\H}\d s+ 2\beta \int_{0}^{t} \|\u_{n}(s)\|^{r+1}_{\L^{r+1}}\d s\right]
		\nonumber \\&=	\|\x_{n}\|^{2}_{\H}+\Tr(\G_{n}\G^{*}_{n})t +2\mathbb{E}\bigg( \int_{0}^{t}( \f_{n}, \u_{n}(s) ) \d s\bigg),
		\end{align} where we used the fact that the final term appearing in the right hand side of the equality \eqref{418} is a local martingale. Similarly, for  the strong solution $\u(\cdot)$ of equation \eqref{2.7},  one can deduce that 
		\begin{align}\label{Lemma5.4b}\nonumber
		&\mathbb{E}\bigg[\|\u(t)\|^{2}_{\H}+ 2\mu \int_{0}^{t} \|\A^{\frac{1}{2}}\u(s)\|^{2}_{\H}\d s+ 2\beta \int_{0}^{t} \|\u(s)\|^{r+1}_{\L^{r+1}}\d s\bigg]\\&=\|\x\|^{2}_{\H}+\Tr(\G\G^{*})t \quad+2\mathbb{E}\bigg( \int_{0}^{t}(\f, \u(s) )\d s\bigg),
		\end{align}
		for all $t\in[0,T]$. From  \eqref{Lemma5.4a} and \eqref{Lemma5.4b}, it is immediate that 
		\begin{align}\label{Lemma5.4c} \nonumber
		&	\bigg|\mathbb{E}\bigg[\|\u_{n}(t)\|^{2}_{\H}+ 2\mu \int_{0}^{t} \|\A^{\frac{1}{2}}\u_{n}(s)\|^{2}_{\H}\d s+ 2\beta \int_{0}^{t} \|\u_{n}(s)\|^{r+1}_{\L^{r+1}}\d s\bigg] \nonumber\\&\quad-\mathbb{E}\bigg[\|\u(t)\|^{2}_{\H} + 2\mu \int_{0}^{t} \|\A^{\frac{1}{2}}\u(s)\|^{2}_{\H}\d s + 2\beta \int_{0}^{t} \|\u(s)\|^{r+1}_{\L^{r+1}}\d s\bigg]\bigg|\nonumber\\& =\left| \|\x_{n}\|^{2}_{\H}-\|\x\|^{2}_{\H}+\Tr(\G_{n}\G^{*}_{n})t -\Tr(\G\G^{*})t+2\mathbb{E}\bigg( \int_{0}^{t}( \f_{n}, \u_{n}(s) )-( \f, \u(s) ) \d s\bigg)  \right|\nonumber \\&\leq  |\|\x_{n}\|_{\H}-\|\x\|_{\H}|(\|\x_{n}\|_{\H}+\|\x\|_{\H})+ |(\Tr (\G_{n}\G^{*}_{n})-\Tr(\G\G^{*}))|t\nonumber\\& \quad+2\mathbb{E}\bigg(\int_{0}^{t}|(\f_{n}-\f,\u_{n}(s))|\d s\bigg) +2 \mathbb{E}\bigg(\int_{0}^{t}|(\f,\u_{n}(s)-\u(s))| \d s  \bigg).
		\end{align}
		It can be easily seen that 
		\begin{align*}
		|	\|\x_{n}\|_{\H}-\|\x\|_{\H} |\leq 	\|\x_{n}-\x\|_{\H} = \|\Pi_{n}\x-\x\|_{\H} =\left(\sum_{j=n+1}^{\infty}|(\x,e_j)|^2\right)^{1/2} \to 0 \ \text{ as }\  n \to \infty,
		\end{align*}
		and 
		\begin{align*}
		\mathbb{E}\bigg(\int_{0}^{t}|(\f_{n}-\f,\u_{n}(s))|\d s\bigg)\leq \|\f_{n}-\f\|_{\H}T \left\{\mathbb{E}\bigg(\sup_{t\in[0,T]}\|\u_{n}(s)\|_{\H}^2\bigg)\right\}^{1/2} \to 0\ \text{ as }\ n \to \infty,
		\end{align*}
		using \eqref{4.18}. Similarly, we find 
		\begin{align*}
		\Tr(\G_{n}\G^{*}_{n})-\Tr(\G\G^{*})& =\Tr(\Pi_{n}\G\G^{*}-\G\G^{*}) \leq \|\Pi_{n}-\I\|_{\mathcal{L}(\H)}\Tr(\G\G^{*}) \to 0 \ \text{  as }\  n \to \infty. 
		\end{align*}
		Finally, using the weak convergence given in \eqref{420}, we deduce that 
		\begin{align*}
		\mathbb{E} \bigg(\int_{0}^{t}|(\f,\u_{n}(s)-\u(s))|\d s\bigg) \to 0 \text{ as } n \to \infty.
		\end{align*} 
		Using all the above convergence in \eqref{Lemma5.4c}, we finally obtain 
		\begin{align*}
		&\mathbb{E}\bigg[\|\u_{n}(t)\|^{2}_{\H}+ 2\mu \int_{0}^{t} \|\A^{\frac{1}{2}}\u_{n}(s)\|^{2}_{\H}\d s+ 2\beta \int_{0}^{t} \|\u_{n}(s)\|^{r+1}_{\L^{r+1}}\d s\bigg]\\&  \to \mathbb{E}\bigg[\|\u(t)\|^{2}_{\H} + 2\mu \int_{0}^{t} \|\A^{\frac{1}{2}}\u(s)\|^{2}_{\H}\d s + 2\beta \int_{0}^{t} \|\u(s)\|^{r+1}_{\L^{r+1}}\d s\bigg] \ \text{  as }\  n\to \infty.
		\end{align*} Therefore along a subsequence, we deduce that 
		\begin{align*}
		&\|\u_{n_{k}}(t)\|^{2}_{\H}+ 2\mu \int_{0}^{t} \|\A^{\frac{1}{2}}\u_{n_{k}}(s)\|^{2}_{\H}\d s+ 2\beta \int_{0}^{t} \|\u_{n_{k}}(s)\|^{r+1}_{\L^{r+1}}\d s\\& \xrightarrow{\text{a.s.}}  \|\u(t)\|^{2}_{\H} + 2\mu \int_{0}^{t} \|\A^{\frac{1}{2}}\u(s)\|^{2}_{\H}\d s + 2\beta \int_{0}^{t} \|\u(s)\|^{r+1}_{\L^{r+1}}\d s\ \text{  as }\  k\to \infty,
		\end{align*} for all $t\in[0,T]$. The final convergence holds true for the sequence $\{\u_{n}\}$ itself, since $\u_{n}, \u$ are the unique  solutions of the systems \eqref{5.8} and \eqref{2.7}, respectively. 
		
		Let us define $$ \Z_n(t) :=\|\u_{n}(t)\|^{2}_{\H}+ \mu \int_{0}^{t} \|\A^{\frac{1}{2}}\u_{n}(s)\|^{2}_{\H}\d s+ \beta \int_{0}^{t} \|\u_{n}(s)\|^{r+1}_{\L^{r+1}}\d s.$$ Then from \eqref{418}, we infer that 
		\begin{align*}
		\Z_n(t) &= \|\x_{n}\|^{2}_{\H} -\mu \int_{0}^{t} \|\A^{\frac{1}{2}}\u_{n}(s)\|^{2}_{\H}\d s  - \beta \int_{0}^{t} \|\u_{n}(s)\|^{r+1}_{\L^{r+1}}\d s+\Tr(\Q_{n})t + 2\int_{0}^{t}( \f_{n}, \u_{n} (s)) \d s\\&\quad+ 2\int_{0}^{t}	( \G_{n} \d \W(s), \u_{n}(s)),
		\end{align*}
		$\mathbb{P}$-a.s., for all $t\in[0,T]$.	Applying It\^o's formula to the process $e^{\lambda_{0}\Z_n(t)},$ using chain rule, we obtain
		\begin{align}\label{425}
		\d (e^{\lambda_{0}\Z_n(t)})& = e^{\lambda_{0}\Z_n(t)}\bigg[\lambda_{0}\d \Z_n(t)+\frac{\lambda_{0}^{2}}{2} \d [\Z_n,\Z_n]_{t}\bigg]
		\nonumber	\\&
		=\lambda_0e^{\lambda_{0}\Z_n(t)}\big[ -\mu \|\A^{\frac{1}{2}}\u_{n}(t)\|^{2}_{\H}  - \beta \|\u_{n}(t)\|^{r+1}_{\L^{r+1}} +\Tr(\Q_{n}) + 2(\f_{n}, \u_{n}(t) )    \nonumber\\&\quad+2 \lambda_{0}\|\G^*_n\u_{n}(t)\|^{2}_{\H}\big]\d t+ 2\lambda_0e^{\lambda_{0}\Z_n(t)}	( \G_{n} \d \W(t), \u_{n}(t))  .
		\end{align}
		The following inequalities are easy to establish: 
		\begin{align}\label{5.9}
		\Tr(\Q_{n}) \leq \Tr(\Q),  \quad \|\x_{n}\|_{\H} \leq \|\x\|_{\H}, \quad \| \f_{n}\|_{\H} \leq \|\f\|_{\H}. 
		\end{align} 
		Furthermore, we obtain 
		\begin{align}\label{5.10}
		\|\G^*_{n}\u_{n}(t)\|^{2}_{\H} \leq \|\G^*_{n}\|^{2}_{\mathcal{L}(\H)} \|\u_{n}(t)\|^{2}_{\H} \leq \|\Q\|_{\mathcal{L}(\H)}\|\u_{n}(t)\|^{2}_{\H}.
		\end{align}
		Making use of Young's inequality, we find 
		\begin{align}\label{5.11}
		2|(\f_{n}, \u_{n}(t) )| \leq \epsilon \|\u_{n}(t)\|^{2}_{\H} +\frac{ \| \f_{n}\|^{2}_{\H}}{\epsilon}, \ \text{for all }\ \epsilon >0.
		\end{align}
		Using \eqref{5.9}-\eqref{5.11} in \eqref{425}, we obtain 
		\begin{align*}
		\d (e^{\lambda_{0}\Z_n(t)}) &\leq \lambda_0e^{\lambda_{0}\Z_n(t)}\bigg[  \bigg(-\mu \| \A^\frac{1}{2}\u_{n}(t)\|^{2}_{\H} +\Tr(\Q)+ 2\lambda_{0} \|\Q\|_{\mathcal{L}(\H)}\|\u_{n}(t)\|^{2}_{\H}  \\&\quad + \epsilon \|\u_{n}(t)\|^{2}_{\H} +\frac{ \|\f_{n}\|^{2}_{\H}}{\epsilon}\bigg)\d t+2	( \G_{n} \d \W(t), \u_{n}(t))   \bigg], 				
		\end{align*}
		
		For $\epsilon>0$, which will be fixed later, we estimate the drift of the process 
		\begin{align*}
		\mathrm{V}_{n}(t):= \exp\left\{-\lambda_{0}\bigg(\Tr(\Q)+\frac{\|\f\|^{2}_{\H}}{\epsilon}\bigg)t\right\}e^{\lambda_{0}\Z_n(t)}.
		\end{align*}
		By applying It\^o's product formula, we have  
		\begin{align*}
		\d \mathrm{V}_{n}(t) &= \exp\left\{-\lambda_{0}\bigg(\Tr(\Q)+\frac{\|\f\|^{2}_{\H}}{\epsilon}\bigg)t\right\}\d (e^{\lambda_{0}\Z_n(t)}) \nonumber\\&\quad- \lambda_{0}\bigg(\Tr(\Q)+\frac{\|\f\|^{2}_{\H}}{\epsilon}\bigg)\exp\left\{-\lambda_{0}\bigg(\Tr(\Q)+\frac{\|\f\|^{2}_{\H}}{\epsilon}\bigg)t\right\}e^{\lambda_{0}\Z_n(t)}\d t 
		\\&  \leq \lambda_{0}  \exp\left\{-\lambda_{0}\bigg(\Tr(\Q)+\frac{\|\f\|^{2}_{\H}}{\epsilon}\bigg)t\right\}\bigg[  -\mu\| \A^\frac{1}{2}\u_{n}(t)\|^{2}_{\H} + \epsilon \|\u_{n}(t)\|^{2}_{\H} +2\lambda_{0}\|\Q\|_{\mathcal{L}(\H)}\|\u_{n}(t)\|^{2}_{\H}\bigg]\d t\\&\quad +
		2 \lambda_{0}  \exp\left\{-\lambda_{0}\bigg(\Tr(\Q)+\frac{\|\f\|^{2}_{\H}}{\epsilon}\bigg)t\right\}e^{\lambda_{0}\Z_n(t)} ( \G_{n} \d \W(t), \u_{n}(t)).
		\end{align*}
		Using Poincar\'e's inequality (see \eqref{2.3}), we have $\|\u_{n}(t)\|^{2}_{\H} \leq \frac{ \|\A^{\frac{1}{2}}\u_{n}(t)\|^{2}_{\H} }{\lambda_{1}}$, and thus from the above inequality, we deduce that 
		\begin{align*}
		\d \mathrm{V}_{n}(t) &\leq
		\exp\left\{-\lambda_{0}\bigg(\Tr(\Q)+\frac{\|\f\|^{2}_{\H}}{\epsilon}\bigg)t\right\}e^{\lambda_{0}\Z_n(t)}\bigg[-\| \A^\frac{1}{2}\u_{n}(t)\|^{2}_{\H}  \bigg(\mu- \frac{2\lambda_{0}\|\Q\|_{\mathcal{L}(\H)}+\epsilon}{\lambda_{1}}\bigg)\d t\bigg] \\& \quad +2 \lambda_{0} \exp\left\{-\lambda_{0}\bigg(\Tr(\Q)+\frac{\|\f\|^{2}_{\H}}{\epsilon}\bigg)t\right\} e^{\lambda_{0}\Z_n(t)} ( \G_{n} \d \W(t), \u_{n}(t)). 
		\end{align*}                                                       
		Hence, for $ 0 <\lambda_{0} \leq \frac{\mu \lambda_{1}- \epsilon}{2\|\Q\|_{\mathcal{L}(\H)}}$, the drift of the process  $\mathrm{V}_{n}(t)$ is non positive. To be specific, for all $ \lambda_{0} \in \big(0 , \frac{\mu \lambda_{1}}{2\|\Q\|_{\mathcal{L}(\H)}}\big)$, the positive number $\epsilon := \mu\lambda_{1}-2 \lambda_{0}\|\Q\|_{\mathcal{L}(\H)} $ satisfies the above condition.  The choice of $\epsilon$ can be made in \eqref{5.11} also.   Thus, from the above inequality, we have     
		\begin{align*}
		\mathrm{V}_{n}(t) \leq \mathrm{V}_{n}(0)+ \int_{0}^{t} 2 \lambda_{0} \exp\left\{-\lambda_{0}\bigg(\Tr(\Q)+\frac{\|\f\|^{2}_{\H}}{\epsilon}\bigg)s\right\} e^{\lambda_{0}\Z_n(s)} ( \G_{n} \d \W(s), \u_{n}(s)).
		\end{align*} Taking the expectation on both sides and using the fact that the final term appearing in the right hand side of the above inequality is a local martingale, we get $$\mathbb{E}^{\x}[\mathrm{V}_{n}(t)] \leq     \mathbb{E}^{\x}[\mathrm{V}_{n}(0)], \ \text{ for all }\ t\in[0,T],$$ and hence the exponential estimate \eqref{5.12} follows. 
	\end{proof}
	\begin{remark}
	1. 	In fact, one can show that $$\u_n\xrightarrow{\mathrm{a.s.}}\u\ \text{ in }\  \C([0,T];\H)\cap\mathrm{L}^2(0,T;\V)\cap\mathrm{L}^{r+1}(0,T;\wi\L^{r+1}), \ \text{ as }\ n\to\infty,$$ see \cite{MTM1}, for more details. 
	
	2. If we consider the system \eqref{1p1}-\eqref{1p4} with Darcy coefficient $\rho>0$, then one has to choose $ \lambda_{0} \in \big(0 , \frac{\mu \lambda_{1}+\rho}{2\|\Q\|_{\mathcal{L}(\H)}}\big)$ in  \eqref{5.1}-\eqref{5.3} or \eqref{5.12}. 
	\end{remark}
	
	Let us now complete the proof of Proposition \ref{prop4.6} using Lemma \ref{lem4.9}. 
	
	\begin{proof}[Proof of Proposition \ref{prop4.6}]
		Note that the function 
		\begin{align*}
		\mathcal{J}(\u)&:= e^{-\lambda_{0}t\bigg(\Tr(\Q)+\frac{\|\f\|^{2}}{\mu \lambda_{1}-2\lambda_{0}\|\Q\|_{\mathcal{L}(\H)}}\bigg)} \exp\bigg(\mu \lambda_{0}\int_{0}^{t}\| \A^\frac{1}{2}\u(s)\|^{2}_{\H}\d s + \lambda_{0}\|\u(s)\|^{2}_{\H} \\& \quad+\beta \lambda_{0}\int_{0}^{t}\|\u(s)\|^{r+1}_{\L^{r+1}}\d s\bigg),
		\end{align*}     
		is a lower semicontinuous function (using the convergence \eqref{419}) and an increasing limit of the continuous functions
		\begin{align*}
		\mathcal{J}_{n}(\u)&:= e^{ -\lambda_{0}t\bigg(\Tr(\Q)+\frac{\|\f\|^{2}}{\mu \lambda_{1}-2\lambda_{0}\|\Q\|_{\mathcal{L}(\H)}}\bigg)} \exp\bigg(\mu \lambda_{0}\int_{0}^{t}\| \A^\frac{1}{2}\u_{n}(s)\|^{2}_{\H}\d s + \lambda_{0}\|\u_{n}(s)\|^{2}_{\H}\\& \quad+\beta \lambda_{0}\int_{0}^{t}\|\u_{n}(s)\|^{r+1}_{\L^{r+1}}\d s\bigg).
		\end{align*}
		Taking $n \to \infty $ in \eqref{5.12}, we find 
		\begin{align*}
		\mathbb{E}^{\x}[\mathcal{J}(\u)] \leq \liminf_{n\to\infty} \mathbb{E}^{\x}[\mathcal{J}(\u_{n})] \leq e^{\lambda_{0}\|\x\|^{2}_{\H}},
		\end{align*}                                                           
		and hence we obtain the required estimates \eqref{5.1}, \eqref{5.2} and \eqref{ESC}.                                              				
	\end{proof}
	\subsection{The LDP on $\M_{1}(\D(\A^{\alpha}))$}\label{sub4.4}
	As in the case of 2D NSE (\cite{BF}), our  proof of Theorem \ref{main-t} is two folded. In the first part, we establish the LDP for the occupation measures of the 2D SCBF equations \eqref{2.7} for the initial measure in $\mathrm{E} :=\M_{1}(\D(\A^{\alpha}))$. We complete  the proof of Theorem \ref{main-t} by extending the LDP for initial conditions, open and closed subsets in $\M_{1}(\H)$.
	\begin{lemma}\label{lem4.11}
		Let $ \f\in \H \text{ and let } \frac{1}{4} < \alpha < \frac{1}{2}$ for $r=1,2$ and $\frac{1}{3} < \alpha < \frac{1}{2}$ for $r=3$ be fixed numbers such that \eqref{1p5} holds. Let $ 0<\lambda_{0} < \frac{\mu \lambda_{0}}{2\|\Q\|_{\mathcal{L}(\H)}}$ and 
		\begin{align}\label{6.1}  
		\phi(\x) = e^{\lambda_{0}\|\x\|^{2}_{\H}},  \quad \mathcal{M}^{*}_{\lambda_{0},R} := \bigg\{\nu \in \M_{1}(\D(\A^{\alpha})): \int \phi(\x)\nu (\d \x) \leq R\bigg\}.
		\end{align}
		Then the family $\mathbb{P}_{\nu}(L_{T}\in \cdot) \text{ as } T\to +\infty$ satisfies the LDP on $\M_{1}(\D(\A^{\alpha}))$ with respect to the topology $\tau$ with the speed $T$ and the rate function $\J$, uniformly for any initial measure in $\mathcal{M}^{*}_{\lambda_{0},R},$ where $R >1$ is any fixed number, and $\J$ is the level-2 entropy of Donsker-Varadhan. More precisely, the statements (i), (ii) and (iii) of Theorem \ref{main-t}  hold true with $\M_{1}(\H)$ replaced by $\M_{1}(\D(\A^{\alpha}))$.
	\end{lemma} 
	\begin{proof}
		From Theorem \ref{main}, we know that the process  $\{\u_{t}\}_{t\geq 0}$ is strongly Feller and topologically irreducible in $\D(\A^{\alpha})$.  We infer from Theorem \ref{thm4.5} that, in order to establish LDP for our model in $\D(\A^{\alpha})$,  it is  enough to show the estimates  \eqref{4.5} and \eqref{4.6}.
		Let us define a  bounded subset $K$ of $\D(\A^{\frac{1}{2}})$  as
		\begin{align}\label{6.2}
		K:= \big\{\x \in \D(\A^{\frac{1}{2}}); \|\A^{\frac{1}{2}}\x\|_{\H}\leq M \big\},
		\end{align}
		where $M$ is a real number, which we will fix  later. We know that $\D(\A^{\frac{1}{2}}) $ is compactly embedded in $ \D(\A^{\alpha}),$ for $\alpha <\frac{1}{2}$ and since  $K$ is bounded, it is a compact subset in $\D(\A^{\alpha})$.
		From the definition of the occupation measure we have 
		\begin{align*}
		\mathbb{P}_{\nu}(\tau^{(1)}_{K}>n) \leq 	\mathbb{P}_{\nu}\bigg(L_{n}(K)\leq \frac{1}{n}\bigg) = 	\mathbb{P}_{\nu}\bigg(L_{n}(K^{c})\geq 1- \frac{1}{n}\bigg).
		\end{align*}
		For the set $K$ defined in \eqref{6.2}, an application of Markov's inequality yields  $L_{n}(K^{c}) \leq \frac{1}{M^{2}} L_{n}(\|\A^{\frac{1}{2}}\x\|^{2}_{\H})$. Hence, for any fixed $\lambda_{0}$ such that $ 0 < \lambda_{0} < \frac{\mu \lambda_{1}}{2\|\Q\|_{\mathcal{L}(\H)}}$, using Markov's inequality, we obtain 
		\begin{align}\label{432}
		\mathbb{P}_{\nu}(\tau^{(1)}_{K}>n) &\leq 	\mathbb{P}_{\nu}\bigg(L_{n}(\|\A^{\frac{1}{2}}\x\|^{2}_{\H})\geq M^{2}\bigg(1-\frac{1}{n}\bigg)\bigg)\nonumber \\& =\mathbb{P}_{\nu}\bigg(\frac{\mu\lambda_{0}}{n}\int_{0}^{n}\|\A^{\frac{1}{2}}\u(s)\|^{2}_{\H}\d s\geq \mu\lambda_0 M^{2}\bigg(1-\frac{1}{n}\bigg)\bigg)
		\nonumber	\\
		&\leq \exp{\bigg(-n\mu \lambda_{0}M^{2}\bigg(1-\frac{1}{n}\bigg)\bigg)}\mathbb{E}^{\nu}\left\{\exp \bigg(\mu\lambda_{0}\int_{0}^{n}\|\A^{\frac{1}{2}}\u(s)\|^{2}_{\H}\d s\bigg)\right\}.
		\end{align}
		For the initial measure $\nu \in \M_{1}(\D(\A^{\alpha}))$, integrating the estimate \eqref{5.2} with respect to $\nu(\d \x)$, we obtain \begin{align*}
		\mathbb{E}^{\nu}\left\{\exp \bigg(\mu\lambda_{0}\int_{0}^{n}\|\A^{\frac{1}{2}}\u(s)\|^{2}_{\H}\d s\bigg) \right\}\leq e^{n\lambda_{0}\big(\Tr(\Q)+ \frac{\|\f\|_{\H}^{2}}{ \mu \lambda_{1}-2\|\Q\|_{\mathcal{L}(\H)}\lambda_{0}}\big)}\nu\big(e^{\lambda_{0}\|\cdot\|^{2}_{\H}}\big).
		\end{align*}
		Substituting the above expression in \eqref{432}, we find  
		\begin{align}\label{6.3}
		\mathbb{P}_{\nu}(\tau^{(1)}_{K}>n) \leq \nu(e^{\lambda_{0}\|\cdot\|^{2}})e^{-n\lambda_{0}C}, \ \text{ for all }\ n \geq 2,
		\end{align}
		where 
		\begin{align}\label{434}
		C:= \frac{\mu M^{2}}{2}- \Tr(\Q) - \frac{\|\f\|^{2}}{\mu\lambda_{1}-2\|\Q\|_{\mathcal{L}(\H)}\lambda_{0}}.
		\end{align}
		Let $\lambda >0$ be fixed. Using the integration by parts formula and \eqref{6.3}, we have 
		\begin{align*}
		\mathbb{E}^{\nu}[ e^{\lambda\tau^{(1)}_{K}}] &= 1+ \int_{0}^{+\infty} \lambda e^{\lambda t}\mathbb{P}_{\nu}(\tau^{(1)}_{K}>t)\d t\leq 1+\sum_{n=0}^{\infty}  \lambda e^{\lambda n}\mathbb{P}_{\nu}(\tau^{(1)}_{K}>n) \\ &
		\leq e^{2\lambda} + \sum_{n\geq 2} \lambda e^{\lambda(n+1)}\mathbb{P}_{\nu}(\tau^{(1)}_{K}>n) \\ & \leq e^{2\lambda} \bigg(1+\lambda\nu (e^{\lambda_{0}\|\cdot\|^{2}})\sum_{n\geq 2}e^{-n(\lambda_{0}C-\lambda)} \bigg).
		\end{align*} 	
		From the definition \eqref{6.3}, we can choose the  $M$ appearing  in the definition \eqref{6.2}  such that $\lambda_{0}C- \lambda\geq 1$. Note that for $\x \in K$, using Poincar\'e's inequality, we have $ \|\x\|^{2}_{\H} \leq \frac{\|\A^{\frac{1}{2}}\x\|^{2}_{\H}}{\lambda_{1}} \leq \frac{M^{2}}{\lambda_{1}}$. Taking the supremum over the set $\{\nu=\delta_{\x},\x \in K\}$, we obtain
		\begin{align*}
		\sup_{\x\in K}\mathbb{E}^{\x} [e^{\lambda\tau^{(1)}_{K}}] \leq e^{2\lambda} \bigg(1+\lambda e^{\frac{\lambda_{0}M^{2}}{\lambda_{1}}}\sum_{n\geq 2}e^{-n(\lambda_{0}C-\lambda)}\bigg) < \infty,\end{align*}	
		and hence the bound \eqref{4.5} holds. We obtain \eqref{4.6} by using a similar  procedure.  By the definition of $\tau_K$, we know that  $ \tau_{K}	\leq \tau^{(1)}_{K}$ and  hence we deduce that 
		\begin{align*}
		\sup_{\nu \in \mathcal{M}^{*}_{\lambda_{0},R}} \mathbb{E}^{\nu}[e^{\lambda\tau_{K}}]
		& \leq  \sup_{\nu \in \mathcal{M}^{*}_{\lambda_{0},R}} \mathbb{E}^{\nu}[e^{\lambda\tau^{(1)}_{K}}]	\leq e^{2\lambda}\bigg(1+ \lambda R \sum_{n\geq 2}
		e^{-n(\lambda_{0}C-\lambda)}\bigg)	< \infty,\end{align*}
		which completes the proof.
	\end{proof}
	\subsection{Proof of Theorem \ref{main-t}}
	In this subsection, we  finish the proof of Theorem \ref{main-t}. In the previous subsection \ref{sub4.4}, we proved the LDP for initial measure in $\M_{1}(\D(\A^{\alpha}))$.  Now, we extend Lemma \ref{lem4.11} for initial conditions, open and closed subsets in $\M_{1}(\H)$, in order to establish the claim \eqref{Claim}. The first statement  $ \J(\nu) < \infty \Rightarrow \nu \ll \varrho $  in \eqref{Claim} is already proved in Proposition \ref{prop4.4} (see Proposition 3.1, \cite{MG1} also). The second statement that for  $\nu$ such that $ \J(\nu) < \infty $ implies $\nu(\|\A^{\frac{1}{2}}\x\|^{2}_{\H}) <\infty $ follows from \eqref{5.5}.	

	The proof of upper and lower bounds of LDP in Theorem \ref{main-t} can be carried out in a similar way as in \cite{MG1}, and hence we omit it here.

	\medskip\noindent
	{\bf Acknowledgments:} The first author would like to thank Ministry of Education, Government of India - MHRD for financial assistance. M. T. Mohan would  like to thank the Department of Science and Technology (DST), India for Innovation in Science Pursuit for Inspired Research (INSPIRE) Faculty Award (IFA17-MA110). 
		
\end{document}